\documentclass[a4paper,11pt, dvipsnames]{scrartcl}
\usepackage{etex}
\addtokomafont{disposition}{\rmfamily}

\usepackage{amsthm}
\usepackage{amsmath}
\usepackage{amssymb}
\usepackage{mathrsfs} 
\usepackage{graphicx}
\usepackage{xcolor}
\usepackage{multirow}
\usepackage[%
  natbibapa
]{apacite} 

\usepackage[colorlinks=true,linkcolor=NavyBlue, citecolor=NavyBlue, urlcolor=NavyBlue, linktocpage=true]{hyperref} 
\usepackage{bbm}
\usepackage{textcomp}
\usepackage{enumerate}
\usepackage{mathtools}

\usepackage{framed}
\usepackage{comment}
\definecolor{shadecolor}{gray}{0.9}

\usepackage{dsfont} 
\usepackage{caption} 
\usepackage{subcaption} 
\usepackage{graphicx}
\usepackage{pdfpages}
\usepackage{url}
\usepackage[english]{babel}
\usepackage{rotating} 
\usepackage{lscape} 

\specialcomment{extra}{\begin{shaded}}{\end{shaded}}

\theoremstyle{plain}  
\newtheorem{thm}{Theorem}[section] 
\newtheorem{lem}[thm]{Lemma} 
\newtheorem{prop}[thm]{Proposition} 
\newtheorem{cor}[thm]{Corollary} 

\theoremstyle{definition} 
\newtheorem{defn}[thm]{Definition}

\newtheorem{rem}[thm]{Remark}

\newtheoremstyle{assumption}
{3pt}
{3pt}
{}
{}
{\bf}
{.}
{.5em}
{\thmname{#1} (\thmnote{#3}\thmnumber{#2})}

\theoremstyle{assumption}
\newtheorem{ass}{Assumption}

\theoremstyle{remark} 

\newcommand{\diff}{\mathrm{d}}
\newcommand{\dint}{\,\mathrm{d}}
\newcommand{\E}{\mathbf{E}}

\newcommand{\ES}{ES}

\newcommand{\eps}{\varepsilon}

\newcommand{\F}{\mathcal{F}}
\newcommand{\R}{\mathbb{R}}
\newcommand{\A}{\mathsf{A}}
\renewcommand{\O}{\mathsf{O}}
\newcommand{\one}{\mathds{1}}
\newcommand{\interior}{\operatorname{int}}

\newcommand{\VaR}{\operatorname{VaR}}
\newcommand{\EVaR}{\operatorname{EVaR}}
\DeclareMathOperator*{\argmin}{arg\,min}

\newcommand{\N}{\mathbb N}
\renewcommand{\P}{\mathcal P}
\renewcommand{\L}{\mathcal L}

\renewcommand{\a}{\alpha}

\newcommand{\M}{\mathcal M}
\renewcommand{\r}{\rho}
\newcommand{\Y}{\mathcal Y}

\renewcommand{\rm}{\normalfont \rmfamily}
\renewcommand{\bf}{\normalfont \bfseries}

\def\be{\begin{equation} \label}
\def\ee{\end{equation}}

\numberwithin{equation}{section} 

\usepackage[colorinlistoftodos,textsize=tiny]{todonotes}
\newcommand{\Comments}{1}
\newcommand{\mynote}[2]{\ifnum\Comments=1\textcolor{#1}{#2}\fi}
\newcommand{\mytodo}[2]{\ifnum\Comments=1%
  \todo[linecolor=#1!80!black,backgroundcolor=#1,bordercolor=#1!80!black]{#2}\fi}

\ifnum\Comments=1               
\fi

\ifnum\Comments=1               
\fi

\ifnum\Comments=1               
\fi

\begin{document}

\title{Elicitability and Identifiability of Systemic Risk Measures}
\author{Tobias Fissler\thanks{Vienna University of Economics and Business, Institute for Statistics and Mathematics, Welthandels\-platz 1, 1020 Vienna, Austria, 
\texttt{tobias.fissler@wu.ac.at},
\texttt{jana.hlavinova@wu.ac.at} and \texttt{birgit.rudloff@wu.ac.at}}
\and Jana Hlavinov\'{a}\footnotemark[1]
\and Birgit Rudloff\footnotemark[1]
}
\maketitle

\begin{abstract}
\textbf{Abstract.}
Identification and scoring functions are statistical tools to assess the calibration and the relative performance of risk measure estimates, e.g., in backtesting. A risk measures is called identifiable (elicitable) it it admits a strict identification function (strictly consistent scoring function). We consider measures of systemic risk introduced in \citet{FeinsteinRudloffWeber2017}. Since these are set-valued, we work within the theoretical framework of \cite{FisslerHlavinovaRudloff_theory} for forecast evaluation of set-valued functionals. We construct oriented selective identification functions, which induce a mixture representation of (strictly) consistent scoring functions. Their applicability is demonstrated with a comprehensive simulation study.
\end{abstract}


\section{Introduction}

\subsection{Systemic risk measures}

In the financial mathematics literature, there is a great interest in various types of risk and, in particular, its quantitative measurement. The quantitative assessment of risk connected to a particular financial position dates back to \citet{ArtznerDelbaenETAL1999} and has since then been discussed, from several points of view, in many further works, see e.g.\ \citep{FoellmerSchied2002,ArtznerDelbaenKochMedina2009,FoellmerWeber2015}. For a thorough overview of risk measures we refer the reader to the textbook \cite{FoellmerSchied2004}. 

The financial crisis of 2007\,--\,2009 and its aftermaths in the last decade have starkly underpinned the need to quantitatively assess the risk of an entire financial system rather than merely its individual entities. One of the first academic works on systemic risk is the seminal paper by \citet{EisenbergNoe}. The focus of this work, however, lies on modelling the financial system rather than measuring its systemic risk. Since then, financial mathematicians have developed a rich strand of literature, encompassing different approaches and emphasising various aspects of systemic risk. The model of \cite{EisenbergNoe} has been generalised in different ways, for instance by considering illiquidity \citep{RogersVeraart2013} or central clearing \citep{Amini2015}. One strand of literature defines systemic risk measures by applying a scalar risk measure to the distribution of the total profits and losses of all firms in the system \citep{AcharyaPedersenETAL2016, AdrianBrunnermeier2016}. Recognising the drawbacks of treating the economy as a portfolio, \citet{ChenIyengarMoallemi2013} introduce an axiomatic approach to measuring systemic risk, further extended by \cite{KromerOverbeckZilch2016} and \cite{HoffmannMeyerBrandisSvindland2016}. The axiomatic approach of \cite{ChenIyengarMoallemi2013} is widely used and amounts to systemic risk measures of the form $\r(\Lambda(Y))$, where $Y$ is a $d$-dimensional random vector representing the financial system, $\r$ is a scalar risk measure and $\Lambda\colon\R^d\to\R$ a non-decreasing aggregation function.
However, this approach of aggregating first and then adding a total capital requirement of the system has the drawback that it results in the measurement of bailout costs rather than capital requirements that prevent a financial crisis. These types of risk measures are also called insensitive as they do not take into account the impact of capital regulation on the system. 
 
As an alternative, so called sensitive systemic risk measures have been introduced by \citet{FeinsteinRudloffWeber2017}; see also \cite{BiaginiETAL2019} and \cite{ArmentiETAL2018} for related approaches. Here, one first adds the capital requirements to the $d$ financial institutions and then applies an aggregation function. That is, one considers systemic risk measures of the form \be{eq:R intro}
R(Y) = \{k\in\R^d\,|\,\r(\Lambda(Y+k))\le0\}.
\ee
Thus, the impact of regulation on the system is taken into account. 

In this paper, we will mainly focus on this type of systemic risk measures as introduced in \cite{FeinsteinRudloffWeber2017}; see Section \ref{sec:measures of systemic risk} for more details. $R(Y)$ specifies the \emph{set} of all capital allocations $k\in\R^d$ such that the new system $Y+k$ is deemed acceptable with respect to $\r$ after being aggregated via $\Lambda$. As such, $R$ takes an \emph{ex ante} perspective prescribing the injections to (and withdrawals from) each financial firm adequate to prevent the system $Y$ from a crisis, whereas $\r(\Lambda(Y))$, as described above, can be interpreted as the bailout costs of the system after a systemic event has occurred.

\subsection{Elicitability and identifiability}
\label{intro:Elicitability and identifiability}

The field of quantitative risk management has seen a lively debate about which scalar risk measure is most appropriate in practice; see \cite{EmbrechtsETAL2014} and \cite{EmmerKratzTasche2015} for detailed academic discussions and \cite{BIS2014} for a regulatory perspective in banking. Besides differences in axiomatic properties such as \emph{coherence} \citep{ArtznerDelbaenETAL1999} and \emph{convexity} \citep{FoellmerSchied2002} of risk measures, 
the debate has also considered more statistical aspects of risk measures. The two most widely discussed statistical desiderata are \emph{robustness} in the sense of \citet{Hampel1971}---cf.\ \cite{ContDeguestETAL2010, KratschmeSchiedETAL2014}---and \emph{elicitability}.

The term elicitability is due to \citet{Osband1985} and \citet{LambertETAL2008}. Using the terminology of mathematical statistics, a real-valued law-invariant risk measure $\r$ is elicitable if it admits an $M$-estimator \citep{HuberRonchetti2009}. That means, there is a loss or scoring function $S\colon\R\times\R\to\R$ such that 
\be{eq:consistent}
\int S(\rho(F), y) \dint F(y) < \int S(x, y) \dint F(y)
\ee
for all $F$ in some class of distribution functions $\M$ and for all $x\neq \r(F)$. Any scoring function $S$ satisfying \eqref{eq:consistent} is called \emph{strictly $\M$-consistent} for $\r\colon\M\to\R$. Besides $M$-estimation, possibly in a regression framework, such as quantile regression \citep{KoenkerBasset1978, Koenker2005} or expectile regression \citep{NeweyPowell1987}, using strictly consistent scoring functions encourages truthful forecasting. This incentive compatibility opens the way to meaningful forecast comparison \citep{Gneiting2011} which is closely related to comparative backtests in finance \citep{FisslerETAL2016, NoldeZiegel2017}.
\citet{Ziegel2016} showed that expectiles are basically the only elicitable and coherent risk measures. In line with this, the prominent risk measure Value at Risk at level $\alpha\in(0,1)$ ($\VaR_\alpha$), which corresponds to a quantile under mild conditions, turns out to be elicitable but not coherent. On the other hand, Expected Shortfall at level $\alpha\in(0,1)$ ($\ES_\alpha$), a tail expectation, is coherent, but fails to be elicitable.
Interestingly, \citet{FisslerZiegel2016} and \citet{AcerbiSzekely2014} showed that the pair $(\VaR_{\alpha}, \ES_{\alpha})$ is elicitable despite ES's failure to have a strictly consistent scoring function on its own. A similar result has recently been established providing the elicitability of the triplet consisting of the risk measure Range Value at Risk along with VaR at two different levels \citep{FisslerZiegel2019_RVaR}.

Closely related to the notion of elicitability is the concept of \emph{identifiability}. While the former is useful for forecast comparison or model selection, the latter aims at model and forecast \emph{validation} or checks for calibration. Again invoking the language of mathematical statistics, a law-invariant real-valued risk measure $\r$ is identifiable if it is a $Z$-functional. That means if it admits a moment function or strict $\M$-identification function $V\colon\R\times\R\to\R$ such that 
\be{eq:id}
\int V(x,y) \dint F(y) = 0 \ \Longleftrightarrow \ x = \rho(F)
\ee
for all $F\in\M$ and for all $x\in\R$. 
\citet{SteinwartPasinETAL2014} showed that, under appropriate regularity conditions, the identifiability of a real-valued risk measure is equivalent to its elicitability. Coherently, $\VaR_\a$ is identifiable under mild regularity conditions, using a simple coverage check, whereas $\ES_\a$ fails to have a strict identification function.
%
For a discussion of identifiability and calibration in the context of evaluating risk measures, we refer the reader to \cite{Davis2016} and \cite{NoldeZiegel2017}.

\subsection{Novel contributions and structure of the paper}

The aim of this paper is to establish elicitability and identifiability results for systemic risk measures of the form at \eqref{eq:R intro}. Since these risk measures are set-valued rather than real-valued, we use the strict distinction between \emph{selective} and \emph{exhaustive} reports introduced in \citet{FisslerHlavinovaRudloff_theory} along with the corresponding notions of elicitability and identifiability. In a nutshell and translated to the setting of systemic risk measures of the form at \eqref{eq:R intro}, a selective forecast specifies a single capital allocation that makes the system acceptable. On the other hand, exhaustive forecasts are more ambitious, aiming at reporting all adequate capital allocations simultaneously in form of a set. Consequently, exhaustive scoring or identification functions take sets as their first argument, whereas their selective counterparts work with points as inputs.\\
The corresponding definitions along with basic properties and assumptions on systemic risk measures defined at \eqref{eq:R intro} and derived quantities such as \emph{efficient cash-invariant allocation rules} (EARs) \citep{FeinsteinRudloffWeber2017} are gathered in Section \ref{sec:notation}.\\
Section \ref{sec:main results} contains our main results, most notably Theorem \ref{thm:identification R_0} asserting the existence of oriented selective identification functions for $R_0(Y) = \{k\in\R^d\,|\,\Lambda(\r(Y+k))=0\}$, and Theorem \ref{thm:elicitability}, which uses these identification functions to construct strictly consistent exhaustive scoring functions for $R$. 
Interestingly, these scoring functions arise as an integral construction of elementary scores, exploiting the orientation of the identification function. This can be considered a higher-dimensional analogon to the \emph{mixture representation} of scoring functions for one-dimensional forecasts established in the seminal paper \cite{EhmETAL2016}. Similarly, this gives rise to the diagnostic tool of \emph{Murphy diagrams} facilitating the assessment of forecast dominance; see Subsection \ref{subsec:Murphy}. Thanks again to the orientation of the identification functions we derive order-sensitivity results of these consistent scoring functions (Proposition \ref{prop:order-sensitivity}). Concerning EARs mentioned above, Proposition \ref{prop:EAR identifiability} establishes strict selective identification functions for EARs, interestingly mapping to a function space. \\
Since systemic risk measures $R$ of the form at \eqref{eq:R intro} are translation equivariant in the sense that $R(Y+k) = R(Y)  -k$ for all $k\in\R^d$ and---under mild assumptions---homogeneous in that $R(cY) = cR(Y)$ for all $c>0$ (Lemma \ref{lem:R pos hom}), it makes sense to determine the subclasses of translation invariant or positively homogeneous consistent scoring functions for $R$, which is the content of Section \ref{sec:secondary}.\\
The elicitability results on $R$ rely on the identifiability\,/\,elicitability of the underlying scalar risk measure $\r$. This spells doom for the elicitability of systemic risk measures induced by ES as a scalar risk measure. Section \ref{sec:ES} outlines this issue and establishes a solution to this challenge at the cost of a higher forecast complexity. Similarly to the scalar case, considering a pair of $R$ based on ES together with a VaR-related quantity leads to selective identifiability and exhaustive elicitability results (Proposition \ref{prop:ES ident} and Theorem \ref{thm:ES elicitability}).\\
The practical applicability of our results is demonstrated in terms of a simulation study, being the content of Section \ref{sec:simulations}. Employing Diebold-Mariano tests, we examine how well the strictly consistent scores are able to distinguish different forecast performances. We also graphically illustrate the diagnostic tool of Murphy diagrams in a simulation example, utilising a traffic-light approach suggested in \cite{FisslerETAL2016}.\\
Section \ref{sec:discussion} closes the paper with a discussion and outlook of possible applications of our results and avenues of future research.\\
All proofs and purely technical results are deferred to the Appendix. 
Results concerning risk measures insensitive with respect to capital allocations along with some additional graphics of simulation results are collected in an online supplementary material \cite{Supplementary}.

\section{Notation and terminology}\label{sec:notation}
\subsection{Measures of systemic risk}\label{sec:measures of systemic risk}

We consider set-valued systemic risk measures studied in \cite{FeinsteinRudloffWeber2017}. In particular, we concentrate on law-invariant risk measures $R$ that are induced by some law-invariant scalar risk measure $\rho$. To settle some notation, let $(\Omega, \mathfrak F, \mathbb P)$ be an atomless probability space. For some integer $d\ge1$, let $\Y^d \subseteq L^0(\Omega;\R^d)$  be some subclass of $d$-dimensional random vectors. From a risk management perspective, a random vector $Y = (Y_1, \ldots, Y_d)\in\Y^d$ represents the respective gains and losses of a system of $d$ financial firms. That is, positive values of the component $Y_i$ represent gains of firm $i$ and negative values correspond to losses. Let $\M^d$ be the class of probability distributions of elements of $\Y^d$. 
Let $\Lambda \colon\R^d\to\R$ be an \emph{aggregation function} meaning that it is non-decreasing with respect to the componentwise order. An aggregation function is typically, but not necessarily, assumed to be continuous or even concave.
We introduce $\Y\subseteq L^0(\Omega;\R)$ where $ \{\Lambda(Y) \,|\, Y\in\Y^d\}\subseteq \Y$ and let $\M$ be the class of distributions of elements of $\Y$.
Where convenient we will tacitly assume that $\Y^d$ and $\Y$ are closed under translation meaning that $X\in\Y, Y\in\Y^d$ implies that $X+m\in\Y$ and $Y+k\in\Y^d$ for all $m\in\R$, $k\in\R^d$. 

We consider some scalar monetary law-invariant risk measure $\rho\colon\Y\to\R$ \citep{ArtznerDelbaenETAL1999}. That means, we can alternatively consider $\r$ as a map $\rho\colon \M\to\R$  such that for a random variable $X\in\Y$ with distribution $F_X\in\M$ we define $\r(F_X):= \r(X)$. We assume that $\r$ is cash-invariant, that is, $\r(X+m) = \r(X) - m$ for all $m\in\R$ and all $X\in\Y$, and monotone, meaning $X\ge Z$ $\mathbb P$-a.s.\ implies that $\r(X)\le \r(Z)$ for all $X,Z\in\Y$. We often dispense with the usual normalisation assumption that $\rho(0)=0$. \\
\begin{rem}
Often, the scalar risk measure is assumed to map to $\R^* = (-\infty,\infty]$. We could also do that at the costs of a more technical treatment. However, to avoid unnecessary technicalities, we refrain from that and will assume throughout the paper that any scalar risk measure will attain real values only.
\end{rem}
We present the two most natural law-invariant set-valued measures of systemic risk
that are based on $\r$ and $\Lambda$, namely
\begin{align}
\label{eq:R}
&R\colon \Y^d\to 2^{\R^d}, && Y \mapsto  R(Y) = \{k\in\R^d\,|\,\r(\Lambda (Y+k))\le0\},\\
\label{eq:Rins}
&R^{ins}\colon \Y^d \to 2^{\R^d}, && Y \mapsto  R^{ins}(Y) = \{k\in\R^d\,|\,\r(\Lambda (Y)+\bar k)\le0\}.
\end{align}
In \eqref{eq:Rins} and later on, we use the shorthand $\bar k := \sum_{i=1}^d k_i$ for some vector $k = (k_1, \ldots, k_d)\in\R^d$. 
Note the difference between $R$ and $R^{ins}$. The risk measure $R$ takes an \emph{ex ante} perspective in the sense that it specifies all capital allocations $k\in\R^d$ needed to be added to the system $Y$ to make the aggregated system $\Lambda(Y+k)$ acceptable under $\r$. On the other hand, $R^{ins}$ takes an \emph{ex post} perspective on quantifying the risk of the system $Y$. That means it first considers the current aggregated system $\Lambda(Y)$ and then specifies the \emph{total} capital requirement $\bar k$ one needs to add to make the aggregated system acceptable, which amounts to specifying the \emph{bail-out costs} of the aggregated system $\Lambda(Y)$ under $\r$. In particular, the risk measure $R^{ins}$ is \emph{insensitive} to the capital allocation to each financial firm, disregarding possible transaction costs or other dependence structures between the financial firms. This justifies the mnemonic terminology.
We would like to remark that both risk measures, $R$ and $R^{ins}$, can be of interest in applications, taking into regard the different perspectives on systemic risk. However, the mathematical treatment and complexity differ considerably: Due to the cash-invariance of $\r$, $R^{ins}$ takes the equivalent form
\[
R^{ins}(Y) = \{k\in\R^d\,|\,\r(\Lambda (Y))\le \bar k\}.
\]
This means that $R^{ins}$ is actually a bijection of the scalar risk measure $\r \circ \Lambda\colon\M^d\to\R$ considered in \cite{ChenIyengarMoallemi2013}.
Therefore, one has to evaluate the risk measure $\r$ only once to determine $R^{ins}$. In contrast, such an appealing equivalent formulation is generally not available for $R$, unless $\Lambda$ is additive, or is even the sum in which case $R$ and $R^{ins}$ coincide. Consequently, in general, one is bound to evaluate $\r$ infinitely often to compute $R$; see also the discussion in \cite{FeinsteinRudloffWeber2017}.
The main focus of this paper are elicitability and identifiability results of systemic risk measures of the form at \eqref{eq:R} and \eqref{eq:Rins}. However, since one can exploit the one-to-one relation between $R^{ins}$ and $\r \circ \Lambda$ and make use of the revelation principle \citep{Osband1985, Gneiting2011, Fissler2017} to establish (exhaustive) elicitability and identifiability results, we do not present results about $R^{ins}$ in the main body of the paper, but rather defer it to the online supplementary material \cite{Supplementary}.

%

For the sake of completeness, we evoke the most important properties of $R$ presented in \cite{FeinsteinRudloffWeber2017}.
Due to the properties that $\r$ is cash-invariant and that $\Lambda$ is increasing, we obtain that the values of $R$ defined at \eqref{eq:R} are \emph{upper sets}. That means for any $Y\in\Y^d$, $R(Y) = R(Y) +\R_+^d $, where $\R_+^d:= \{x\in\R^d\,|\,x_1, \ldots, x_d\ge0\}$ and where for any two sets $A, B\subseteq \R^d$, $A+B:= \{a+b\,|\,a\in A, \ b\in B\}$ denotes the usual \emph{Minkowski sum}. \emph{Mutatis mutandis}, the same is true for $R^{ins}$. Following the notation of \cite{FeinsteinRudloffWeber2017}, we denote the collection of upper sets in $\R^d$ with ordering cone $\R_+^d$ as 
\[
\P(\R^d; \R^d_+) := \{B\subseteq \R^d\,|\, B = B+\R^d_+\}.
\]
Note that both $\R^d$ and $\emptyset$ are elements of $\P(\R^d; \R^d_+) $. Moreover, $R$ defined at \eqref{eq:R} can attain these values even if the underlying scalar risk measure $\r$ maps to $\R$ only, e.g., when $\Lambda$ is bounded. While the case $R(Y) = \emptyset$ corresponds to the case that a scalar risk measure of a financial position is $+\infty$, meaning that the system $Y+k$ is deemed risky no matter how much capital is injected, the case $R(Y) = \R^d$ corresponds to $-\infty$ in the scalar case. The latter situation of ``cash cows'' with the possibility to withdraw any finite amount of money without rendering the position risky is usually deemed unrealistic and is excluded. Therefore, we shall usually only discuss the former case, but remark that a treatment of the latter were also possible for most results. \\
The monotonicity carries over to $R$ ($R^{ins}$) in that for all $Y,Z\in\Y^d$ with $Y\ge Z$ $\mathbb P$-a.s.\ componentwise,
$R(Y)\supseteq R(Z)$. The cash invariance carries over to $R$ such that for all $Y\in\Y^d$ and all $k\in\R^d$, $R(Y+k) = R(Y) - k$. Note, however, that $R^{ins}$ is in general not cash-invariant.
\\
To shorten the notation, we also introduce further subclasses of $\P(\R^d; \R^d_+)$ where $\mathcal B(\R^d)$ denotes the Borel-$\sigma$-algebra.

\begin{defn}
\begin{enumerate}[(i)]
\item
The class of Borel-measurable upper subsets of $\R^d$ is denoted with
$\widehat \P(\R^d;\R^d_+):= \big(\P(\R^d;\R^d_+)\cap \mathcal B(\R^d)\big) \setminus\{\R^d\}$.
\item
The class of closed upper subsets of $\R^d$ is denoted with 
$\F(\R^d;\R^d_+)$.
Note that $\F(\R^d;\R^d_+)\subset \widehat \P(\R^d;\R^d_+)$.
\end{enumerate}
\end{defn}
\noindent
We shall regularly make use of the following assumptions.
\begin{ass}\label{ass:measurability}
For all $Y\in\Y^d$, $R(Y)\in \widehat \P(\R^d;\R^d_+)$.
\end{ass}

\begin{ass}\label{ass:closedness}
For all $Y\in\Y^d$, $R(Y)\in\F(\R^d;\R^d_+)$ and the set $\{k\in\R^d\,|\,\r(\Lambda (Y+k))=0\}$ corresponds to the topological boundary $\partial R(Y)$ of $R(Y)$.
\end{ass}


If $\Lambda\colon\R^d\to\R$ is continuous and $\r$ satisfies the Fatou property (which means it is lower-semicontinuous), the values of $R$ are closed. Note that the law-invariance of $\r$ implies the Fatou property \citep{JouiniETAL2006}. 
If moreover the function $\Lambda\colon\R^d\to\R$ is strictly increasing, the second part of Assumption \eqref{ass:closedness} is satisfied as well. 
Similarly to $R$, 
we introduce the law-invariant map
\begin{align}
\label{eq:R_0}
&R_0\colon \Y^d\to 2^{\R^d}, \qquad Y \mapsto  R_0(Y) = \{k\in\R^d\,|\,\r(\Lambda (Y+k))=0\},
\end{align}
Note that for $R(Y) \neq \emptyset$, 
$R_0(Y)$ is non-empty if the first part of Assumption \eqref{ass:closedness} is satisfied. Since $\Lambda$ is increasing and $\r$ is cash-invariant, one then obtains the relation
\[
R(Y) = R_0(Y) + \R_+^d.
\]
That means that the values of $R_0$ 
determine $R$ 
completely. 
Moreover, if $\Lambda$ is strictly increasing, then $R_0(Y)$ can be characterised as the topological boundary of $R(Y)$ which has the interpretation that $R_0(Y)$ contains the \emph{efficient} capital allocations that make $Y$ acceptable under $R$. That means for such situations, $R$ and $R_0$ are connected via a one-to-one relation. Again, this means that exhaustive elicitability results for $R$ (Theorem \ref{thm:elicitability}) carry over to $R_0$ for such situations, invoking the revelation principle.

Finally, we introduce an important scalarization of the systemic risk measure $R$, called \emph{efficient cash-invariant allocation rule (EAR)}, as introduced in \cite{FeinsteinRudloffWeber2017}. Under certain circumstances, an $EAR$ can also be considered as a \emph{selection} of $R$, or alternatively, of $R_0$.\footnote{A selection of a non-empty set $A$ is any singleton $\{a\}\subseteq A$.}
Roughly speaking, for $Y\in\Y^d$, the value of $EAR(Y)$ gives the capital allocation(s) with minimal weighted costs of an allocation in $R(Y)$. For simplicity, we shall confine attention to $EARs$ with a \emph{fixed} price or weight vector $w\in\R^d_{++}:= \{x\in\R^d\,|\,x_1, \ldots, x_d>0\}$. To settle some notation, we introduce the following definition.

\begin{defn}[EAR]\label{defn:EAR}
An \emph{efficient cash-invariant allocation rule} for a fixed price vector $w\in\R^d_{++}$ is given by 
\be{eq:EAR}
EAR_w(Y)=\argmin\limits_{k\in R(Y)} w^\top k\,. 
\ee
\end{defn}
Note that $EAR_w(Y)$ is well defined and non-empty for $w\in\R^d_{++}$ if $R(Y)$ is closed and there is a supporting hyperplane for $R(Y)$ that is orthogonal to $w$. $EAR_w(Y)$ is then necessarily the intersection of $R_0(Y)$ and this hyperplane. If $R(Y)$ is not closed, the minimum in \eqref{eq:EAR} might not exist, resulting in $EAR_w(Y)=\emptyset$. If, on the other hand, there is no supporting hyperplane for $R(Y)$ orthogonal to $w$, $w^\top k$ for $k\in R(Y)$ is unbounded from below and we have again $EAR_w(Y)=\emptyset$. \\
As discussed in \cite{FeinsteinRudloffWeber2017}, $EAR_w(Y)$ is actually not necessarily a singleton. More precisely, for closed $R(Y)$ it fails to be a singleton if and only if $\partial R(Y)$ contains a line segment that is orthogonal to the price vector $w$.

Since the scalar risk measure $\r$ is assumed to be law-invariant, also the derived quantities $R$, $R^{ins}$, $R_0$ and $EAR_w$ are law-invariant. Therefore, we shall frequently abuse notation and write $R(F_Y):= R(Y)$ for $Y\in\Y^d$ with distribution $F_Y\in\M^d$; with analogous conventions for the other law-invariant maps.

\subsection{Elicitability and identifiability of set-valued functionals}
\label{subsec:definitions elicitability}

We have already mentioned the definitions of elicitability and identifiability for scalar risk measures $\rho\colon\M\to\R$ 
at \eqref{eq:consistent} and \eqref{eq:id}. All other risk measures considered, $R$, $R_0$ and $EAR$, are set-valued, assuming subsets of $\R^d$. Hence, we make use of the theoretical framework on forecast evaluation of set-valued functionals introduced in \cite{FisslerHlavinovaRudloff_theory}. The main idea is to have a thorough distinction in the form of the forecasts between a \emph{selective} notion where forecasts are single points and an \emph{exhaustive} mode where forecasts are set-valued. 
Moreover, corresponding notions of identifiability and elicitability are introduced and discussed in a very general setting, with the main result being that a set-valued functional is elicitable either in the selective, or the exhaustive sense, or it is not elicitable at all \citep[Theorem 2.14]{FisslerHlavinovaRudloff_theory}. 
To allow for a concise presentation, we confine ourselves to introducing only the notions we need and we shall do it directly in terms of $R$ and $R_0$; the case of $EAR$ will be considered later separately. 
In the sequel, let $\mathcal A\subseteq \mathcal B(\R^d)$. Moreover, for scoring functions $S\colon\mathcal A\times \R^d\to\R$ or identification functions $V\colon\R^d\times\R^d\to\R$ we will use the shorthands $\bar S(A,F) := \int S(A,y)\dint F(y)$ and $\bar V(x,F) := \int V(x,y)\dint F(y)$, $A\in\mathcal A$, $x\in\R$, and will tacitly assume that these integrals exist for all $F\in\M^d$.
\begin{defn}
\begin{enumerate}[(i)]
\item
An exhaustive scoring function $S\colon\mathcal A\times \R^d\to\R^* :=(-\infty, \infty]$ is called $\M^d$-\emph{consistent} for $R\colon\M^d\to\mathcal A$ if
\be{eq:strong cons}
\bar S(R(F), F) \le \bar S(A,F) \qquad \forall A\in\mathcal A, \ \forall F\in\M^d.
\ee
The exhaustive score $S$ is \emph{strictly $\M^d$-consistent} for $R$ if it is $\M^d$-consistent for $R$ and if equality in \eqref{eq:strong cons} implies that $A=R(F)$.
\item
The risk measure $R\colon\M^d\to\mathcal A$ is exhaustively elicitable if there is a strictly $\M$-consistent exhaustive scoring function for $R$.
\end{enumerate}
\end{defn}
Note that the \emph{strict} consistency of an exhaustive scoring function $S$ for $R$ implies that $\bar S(R(F),F) \in\R$ for all $F\in\M^d$.
%

\begin{defn}
\begin{enumerate}[(i)]
\item
A map $V\colon\R^d\times \R^d\to\R$ is a \emph{selective $\M^d$-identification function} for $R_0\colon\M^d\to\mathcal A$ if $\bar V(x,F)=0$ for all $x\in R_0(F)$ and for all $F\in\M^d$. Moreover, $V$ is a \emph{strict} selective $\M^d$-identification function for $R_0$ if 
\be{eq:V selective def}
\bar V(x,F) = 0 \quad \Longleftrightarrow \quad  x\in R_0(F), \qquad \forall x\in\R^d, \quad \forall  F\in\M^d.
\ee
\item
The risk measure $R_0\colon\M^d\to\mathcal A$ is selectively identifiable if there is a strict selective $\M^d$-identification function for $R_0$.
%
\end{enumerate}
\end{defn}

\section{Main results} \label{sec:main results}

We present some of the main results of the paper in this section where we gather identifiability results in Subsection \ref{subsec: Set-identifiability of R_0} and elicitability results are presented in Subsection \ref{subsec:elicitability results}.
Theorem \ref{thm:identification R_0} establishes the selective identifiability of $R_0$. 
Notably, the main assumption behind Theorem \ref{thm:identification R_0} and the subsequent results relying on this identifiability is the identifiability of the underlying scalar risk measure $\r$ in \eqref{eq:R}. \emph{A fortiori}, it needs to admit an oriented identification function. According to \cite{SteinwartPasinETAL2014}, a strict identification function $V\colon\R\to\R$ for a real-valued risk measure $\r \colon\M\to\R$ is called \emph{oriented} if for all $x\in\R$ and for all $F\in\M$
\[
\bar V(x,F) \begin{cases}
<0, & \text{if }x<\r(F) \\
=0, & \text{if }x =\r(F) \\
>0, & \text{if }x>\r(F).
\end{cases}
\]
Invoking Theorem 8 in \cite{SteinwartPasinETAL2014} the existence of an oriented identification function for $\r$ is equivalent to the elicitability of $\r$ under mild regularity conditions. Proposition \ref{prop:R0 to rho} 
establishes that under certain assumptions on the aggregation function $\Lambda$ also the converse holds. That is, the elicitability of $R$ implies the elicitability of $\r$.




\subsection{Identifiability results}\label{subsec: Set-identifiability of R_0}


\begin{thm}\label{thm:identification R_0}
Let $\r\colon\M\to\R$ be identifiable.
Then the following assertions hold for $R_0\colon\M^d\to 2^{\R^d}$ defined at \eqref{eq:R_0}.
\begin{enumerate}[\rm(i)]
\item
$R_0$ is selectively identifiable. If $V_\r\colon\R\times \R\to\R$ is a strict $\M$-identification function for $\r$, then 
\be{eq:def V_R_0}
V_{R_0}\colon \R^d\times \R^d\to\R, \qquad (k,y) \mapsto V_{R_0}(k,y) = V_\r(0,\Lambda(y+k))
\ee
is a strict selective $\M^d$-identification function for $R_0$.
\item
If $V_\r\colon\R\times \R\to\R$ is an oriented strict $\M$-identification function for $\r$, then $V_{R_0}$ defined at \eqref{eq:def V_R_0} is oriented for $R_0$ in the sense that for all $F\in\M^d$ it holds that
\be{eq:orientation 3}
\bar{V}_{R_0}(k,F) \begin{cases}
<0, & \text{if } k\notin R (F) \\
=0, & \text{if } k\in R_0(F) \\
>0, & \text{if } k\in  R(F) \setminus R_0(F).
\end{cases}
\ee
\end{enumerate}
\end{thm}

\begin{rem}\label{rem:orientation}
The orientation of $V_{R_0}$ can be considered as the multivariate counterpart of the orientation of $V_\r$ with respect to the componentwise order on $\R^d$. Indeed, in both cases, a negative expected identification function corresponds to the case of predicting a capital requirement too small to make the system $Y\in\Y^d$ acceptable with respect to $R$ or the single firm $X\in\Y$ acceptable with respect to $\r$.
%
\end{rem}

Note that if $V_{R_0}\colon\R^d\times \R^d\to\R$ is a strict selective $\M^d$-identification function for $R_0$ which is oriented in the sense of \eqref{eq:orientation 3} and which is such that the expected identification function $\bar V_{R_0}(\cdot,F)$ is continuous for any $F\in\M^d$, then the values of $R$ are closed sets. 

%
%

\begin{rem}\label{rem:uniqueness V_R_0}
Equation \eqref{eq:def V_R_0} states explicitly how to construct a strict selective $
\M^d$-identification function $V_{R_0}\colon\R^d\times \R^d\to\R$ for $R_0$, given a certain strict $\M$-identification function $V_\r\colon\R\times \R\to\R$ for $\r$. So $V_{R_0}$ definitely depends on the choice of $V_\r$. \citet[Proposition 3.2.1]{Fissler2017} states that under some richness assumptions on the class $\M$, any other strict identification function $\widetilde V_\r\colon \R\times \R\to\R$ for $\r$ is of the form
\be{eq:tilde V_r}
\widetilde V_\r(x,z) = g(x)V_\r(x,z),
\ee
where $g\colon \R\to\R$ is a non-vanishing function. Moreover, if $V_\r$ is oriented, then $\widetilde V_\r$ is oriented if and only if the function $g$ in \eqref{eq:tilde V_r} is strictly positive; see also \citet[Theorem 8]{SteinwartPasinETAL2014}. Consequently, starting with such an identification function $\widetilde V_\r$, the resulting (oriented) strict selective $\M^d$-identification function $\widetilde V_{R_0}\colon\R^d\times \R^d\to\R$ takes the form
\be{eq:tilde V_R_0}
\widetilde V_{R_0}(k,y) = \widetilde V_\r(0,\Lambda(y+k)) = g(0)V_\r(0,\Lambda(y+k)).
\ee
Hence, the only difference is that one ends up with a scaled version of $V_{R_0}$ where the scaling factor $g(0)$ is positive if both $V_{R_0}$ and $\widetilde V_{R_0}$ are oriented.
\end{rem}
In a similar spirit as Remark \ref{rem:uniqueness V_R_0}, one might also wonder whether the (oriented) strict selective identification functions constructed in Theorem \ref{thm:identification R_0} are the only (oriented) strict selective identification functions for $R_0$. This is definitely not the case since due to the linearity of the expectation, any function $V'_{R_0}\colon \R^d\times \R^d\to\R$ with
\be{eq:uniqueness V_R_0 2}
V'_{R_0}(k,y) = h(k)V_{R_0}(k,y)=  h(k)V_\r(0,\Lambda(y+k)),
\ee
where $h\colon\R^d\to\R$ is non-vanishing, is again a strict selective $\M^d$-identification function for $R_0$. Moreover, if $V_{R_0}$ is oriented, then $V'_{R_0}$ defined at \eqref{eq:uniqueness V_R_0 2} is oriented if and only if $h>0$. In particular, the constant $g(0)$ appearing in \eqref{eq:tilde V_R_0} can be incorporated into the function $h$ such that we see that it does not matter which (oriented) strict exhaustive identification function $\widetilde V_\r$ we choose to end up with the form at \eqref{eq:uniqueness V_R_0 2}. The following theorem establishes that basically all selective $\M^d$-identification functions for $R_0$ are of the form at \eqref{eq:uniqueness V_R_0 2}.

\begin{prop}\label{prop:characterization V}
Let $\A\subseteq \R^d$ and let $V_{R_0}, V'_{R_0}\colon \A\times \R^d\to\R$ be strict selective $\M^d$-identification functions for $R_0\colon\M^d\to 2^{\R^d}$. If for every $x\in\A$ there are $F_1, F_2\in\M^d$ such that $\bar{V}_{R_0}(x,F_1)>0$ and $\bar{V}_{R_0}(x,F_2)<0$ and $\M^d$ is convex, then there is a non-vanishing function $h\colon\A\to\R$ such that
\be{eq:char V}
\bar{V}'_{R_0}(x,F)=h(x)\bar{V}_{R_0}(x,F)
\ee
for all $x\in\A$ and all $F\in\M^d$.
\end{prop}

\begin{rem}
\begin{enumerate}[(i)]
\item
It is worth mentioning that the assumptions of Proposition \ref{prop:characterization V} imply that $R_0$ is surjective on $\A\subseteq \R^d$. That is why we formulated the proposition in terms of a general action domain $\A\subseteq \R^d$ rather than $\R^d$.
\item
If $\M^d$ is rich enough, and under additional regularity conditions on $V_{R_0}$, one can also establish a pointwise version of \eqref{eq:char V}; see \cite{FisslerZiegel2016, Erratum} for details. 
%
\end{enumerate}
\end{rem}
Finally, we turn our attention to the efficient cash-invariant allocation rules that represent a possibility to choose one capital allocation that would make the system acceptable, namely the cheapest one with respect to a weight vector $w$.
For any vector $w\in\R^d$, we use the notation $w^\bot:= \{x\in\R^d\,|\,w^\top x =0\}$ for the orthogonal complement of the subspace spanned by $w$. With $\R^{w^\bot}$ we denote the space of all functions mapping from $w^\bot$ to $\R$.

\begin{prop}\label{prop:EAR identifiability}
Let $\r$ be a scalar risk measure, $R$ and $R_0$ as defined in \eqref{eq:R} and \eqref{eq:R_0}, and suppose that Assumption (\ref{ass:closedness}) holds. Assume that $R_0$ is selectively identifiable with an oriented strict selective $\M^d$-identification function $V_{R_0}$. Let $w\in\R^d_{++}$ and define the map $V_{{EAR}_w}\colon\R^d\times\R^d\to\R^{w^\bot}$ via
\[
V_{EAR_w}(k,y)\colon w^\bot\to\R, \qquad w^\bot\ni x\mapsto V_{EAR_w}(k,y)(x)=V_{R_0}(k+x,y)
\]
for $(k,y)\in\R^d\times \R^d$.
Then $V_{{EAR}_w}$ defined in \refeq{eq:EAR} is a strict selective $\M^d$-identification function for $EAR_w$ defined at \eqref{eq:EAR} in the sense that for any $F\in\M^d$ and any $k\in\R^d$
\be{eq:identification EAR}
k\in EAR_w(F)\ \Longleftrightarrow \ \left(\bar{V}_{EAR_w}(k,F)\leq0 
\  \wedge\ \bar{V}_{EAR_w}(k,F)(0)=0\right).
\ee
\end{prop}

If the underlying risk measure $R$ is known to assume convex sets only (e.g.\ if $\r$ is convex and $\Lambda$ concave, see \cite{FeinsteinRudloffWeber2017}), it is even sufficient to evaluate $\bar{V}_{EAR_w}(k,F)(x)$, or its empirical counterpart, for $x\in\R^d$ in a neighbourhood of 0, which can also be seen nicely in Figure \ref{fig:proofEAR} in Appendix \ref{app:proofs main results}.
%

In Section 2.4 of \cite{FisslerHlavinovaRudloff_theory} different versions of the Convex Level Sets (CxLS) property are introduced and their necessity for identifiability and elicitability for set-valued functionals is discussed. 
Since the selective identifiability in Proposition \ref{prop:EAR identifiability} deviates from the usual definition, it is worth noting that, under the conditions of Proposition \ref{prop:EAR identifiability}, $EAR_w$ satisfies the selective CxLS property. However, we do not see how to establish the selective CxLS* property or, alternatively, the exhaustive CxLS property, which leaves the question open if and in what sense $EAR_w$ might be elicitable.

We would like to compare the concept of identifiability introduced in Proposition \ref{prop:EAR identifiability} to the discussion about the backtestability of \emph{loss value at risk} in Section 5 of \cite{BignozziBurzoniMunari2018}. One can interpret their proposal as using a function-valued identification function, too. Then, their analogue of \eqref{eq:identification EAR} is that the infimum of the function-valued identification function be 0 if using the correctly specified forecast. Interestingly, this version of identifiability does not imply that the functional under consideration has convex level sets.

%

We end this section by noting that the identifiability of $\r$ and the selective identifiability of $R_0$ are even equivalent if $\Lambda\colon\R^d\to\R$ possesses a measurable right inverse. 

\begin{prop}\label{prop:R0 to rho}
Let $\r\colon\M\to\R$ be a risk measure, $\Lambda\colon\R^d\to\R$ a surjective aggregation function, and $R_0\colon\M^d\to 2^{\R^d}$ as defined in \eqref{eq:R_0}. 
Assume that there exists a measurable right inverse $\eta\colon \R\to\R^d$ such that $\Lambda\circ \eta = \mathrm{id}_\R$, $\Y$ is closed under translations, and that for any $X\in\Y$, $\eta(X)$ belongs to $\Y^d$. Then it holds that 
$\r$ is (selectively) identifiable if and only if $R_0$ is selectively identifiable. 
\end{prop}


\subsection{Elicitability results and mixture representation}\label{subsec:elicitability results}

In the seminal paper \cite{EhmETAL2016} it is shown that, subject to regularity conditions, any non-negative scoring function $S\colon\R\times\R\to [0,\infty]$ which is consistent for the $\a$-quantile (the $\tau$-expectile) can be written as a \emph{mixture} or \emph{Choquet representation} 
\be{eq:mixture}
S(x,y) = \int_{\R} S_{\theta}(x,y)\dint H(\theta), \qquad x,y\in\R,
\ee
where $H$ is a non-negative measure on $\mathcal B(\R)$ and $S_{\theta}$, $\theta\in\R$, are non-negative elementary scoring functions for the $\a$-quantile (the $\tau$-expectile). In particular, $S_\theta$ take the form
\begin{align}\label{eq:elementary}
S_{\theta}(x,y) &= \big(\one\{\theta < x\} - \one\{\theta < y\}\big) V(\theta, y) 
\end{align}
where $V$ is an \emph{oriented} identification function for the $\a$-quantile (the $\tau$-expectile). The score at \eqref{eq:mixture} is strictly consistent if and only if the measure $H$ is strictly positive, that is, it puts positive mass on any open non-empty set. \cite{Ziegel_Discussion2016} and \cite{Dawid_Discussion2016} argued that this construction also works for more general one-dimensional functionals besides expectiles and quantiles which admit an oriented identification function; cf.\ \cite{JordanMuehlemannZiegel2019}.  \cite{SteinwartPasinETAL2014} showed that, for one-dimensional functionals satisfying certain regularity conditions, the existence of an oriented identification function is equivalent to the elicitability of the functional. While the orientation of the identification function immediately gives rise to the consistency of the elementary scores, and thus, of the mixtures at \eqref{eq:mixture}, an answer to the question as to whether all scoring functions for a certain functional are \emph{necessarily} of the form at \eqref{eq:mixture} can typically only be answered invoking Osband's Principle \citep{Osband1985, FisslerZiegel2016} hence assuming smoothness and regularity conditions.

Our construction of strictly consistent exhaustive scoring functions for the systemic risk measures $R$ also exploits the key result about the existence of oriented strict selective identification functions for $R_0$ and is similar in nature to the approach described above. 
For any $y\in\R^d$, we shall use the notation $R(y) := R(\delta_y)$.

\begin{thm}\label{thm:elicitability}
Let $V_{R_0}\colon\R^d\times \R^d\to\R$ be such that for all $F\in\M^d\cup\{\delta_y\,|\,y\in\R^d\}$ 
\be{eq:weak orientation}
\bar{V}_{R_0}(k,F) \in 
\begin{cases}
(-\infty,0], & \text{if } k\notin R (F) \\
[0,\infty), & \text{if } k\in  R(F).
\end{cases}
\ee
\begin{enumerate}[\rm (i)]
\item
Under Assumption \eqref{ass:measurability}, for each $k\in\R^d$, the map $S_{R,k}\colon \widehat \P(\R^d;\R^d_+)\times \R^d\to[0,\infty)$,
\be{eq:S_{R,k}}
 S_{R,k}(A,y) = \big(\one_{R(y) \setminus A}(k) - \one_{A\setminus R(y)}(k)\big)V_{R_0}(k,y)
\ee
is a non-negative $\M^d$-consistent exhaustive scoring function for $R\colon \M^d\to \widehat\P(\R^d;\R^d_+)$.
\item
Under Assumption \eqref{ass:measurability} and if $\pi$ is a $\sigma$-finite non-negative measure on $\mathcal B(\R^d)$, the map $S_{R,\pi}\colon \widehat \P(\R^d;\R^d_+)\times \R^d\to[0,\infty]$,
\be{eq:S_R,pi}
S_{R,\pi}(A,y) =  \int_{\R^d} S_{R,k}(A,y)\,\pi(\diff k)
\ee
is a non-negative $\M^d$-consistent exhaustive scoring function for $R\colon \M^d\to \widehat\P(\R^d;\R^d_+)$.
\item
If Assumption \eqref{ass:closedness} holds, if $V_{R_0}$ is a strict selective $\M^d$-identification function for $R_0$ and if $\pi$ is a $\sigma$-finite strictly positive measure on $\mathcal B(\R^d)$, then the restriction of $S_{R,\pi}$ defined at \eqref{eq:S_R,pi} to $\F(\R^d;\R^d_+)\times \R^d$ is strictly $ \M^d_0$-consistent for $R\colon  \M^d_0\to \F(\R^d;\R^d_+)$, where $\M^d_0\subseteq \M^d$ is such that $\bar S_{R,\pi}(R(F),F)<\infty$ for all $F\in \M^d_0$.
\end{enumerate}
\end{thm}

Note that the condition at \eqref{eq:weak orientation} is some weak form of orientation. However, it does not imply that $V_{R_0}$ is an identification function for $R_0$. In the one-dimensional setting, such a situation can occur in practice if the underlying risk measure is Value at Risk and the distributions are not continuous, implying that the corresponding quantile identification function will nowhere attain 0 in expectation.

Even though we defer the formal proof of Theorem \ref{thm:elicitability} to Appendix \ref{app:proofs main results}, we would still like to sketch and illustrate the idea, taking into account that Theorem \ref{thm:elicitability} constitutes one of the main results of the paper. The key observation is the identity
\begin{align}\label{eq:consistency 2}
\bar S_{R,\pi}(A,F) - \bar S_{R,\pi}(R(F),F) 
&= \int\limits_{R(F)\setminus A} \bar V_{R_0}(k,F) \,\pi(\diff k) - 
\int\limits_{A\setminus R(F)} \bar V_{R_0}(k,F) \,\pi(\diff k).
\end{align}
Then, one uses the weak orientation of $V_{R_0}$ given at \eqref{eq:weak orientation} to conclude that the first integral on the right hand side of \eqref{eq:consistency 2} is $\ge0$ while the second integral is $\le0$. A graphic illustration of the situation is provided in Figure \ref{fig:proof}. 
\begin{figure}
\centering
\includegraphics[width = 0.7\textwidth]{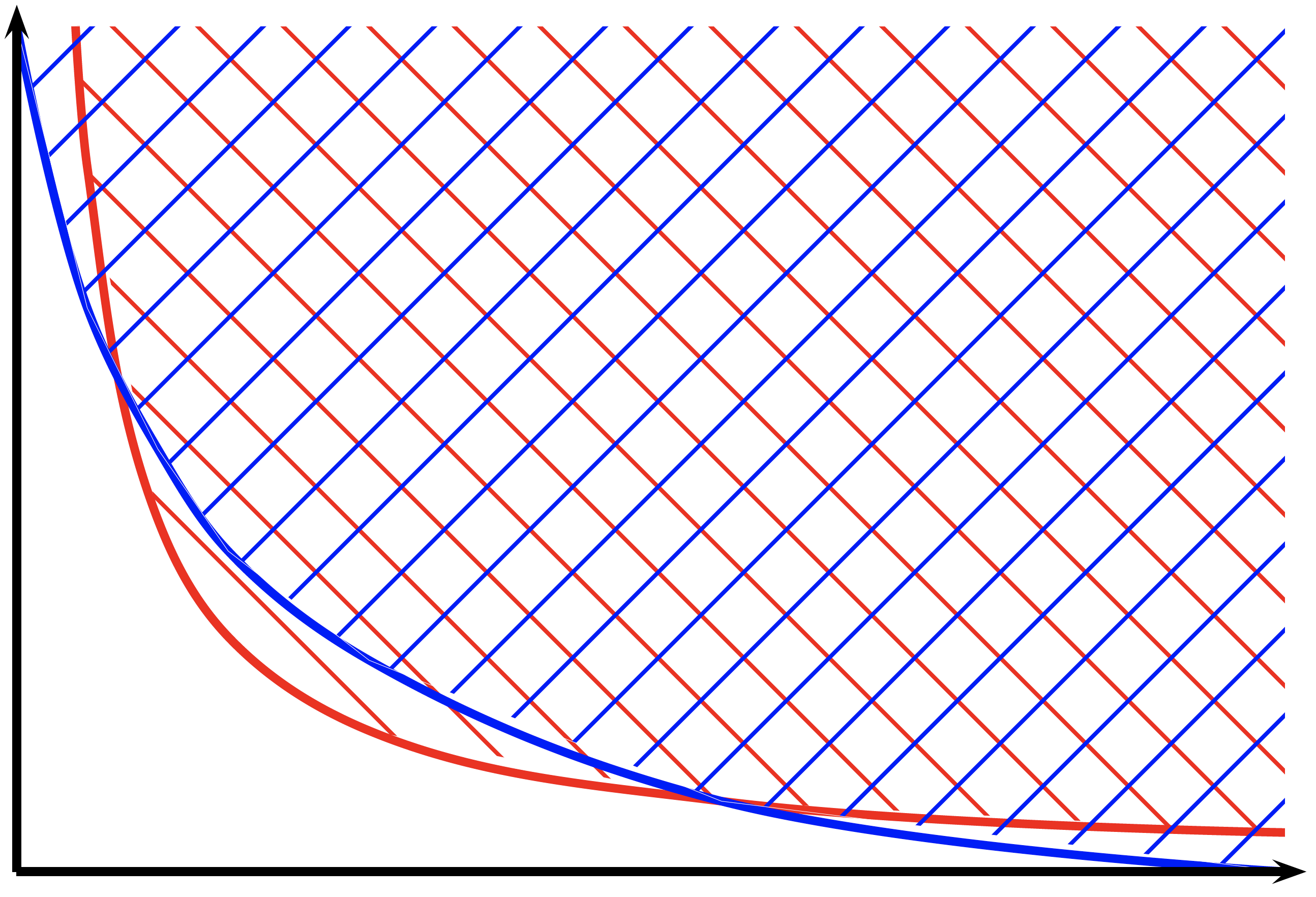}
\caption{A graphical illustration of Equation \eqref{eq:consistency 2} for dimension $d=2$. Suppose the red region corresponds to the correctly specified risk measure $R(F)$ and the blue region corresponds to some misspecified forecast $A$. 
The score difference $\bar S_{R,\pi}(A,F) - \bar S_{R,\pi}(R(F),F)$ is an integral of $\bar V_{R_0}(\cdot, F)$ over $R(F) \setminus A$ (the red only region), plus an integral of $-\bar V_{R_0}(\cdot, F)$ over $A\setminus R(F)$ (the blue only region).}
\label{fig:proof}
\end{figure}

%
%

It is in order to make some comments about the scoring functions constructed in Theorem \ref{thm:elicitability}.

\subsubsection{Comparison with one-dimensional case}

The similarity of the mixture representation at \eqref{eq:S_R,pi} and \eqref{eq:mixture} is obvious. With a closer look, one can also see the similarities on the level of the elementary scores given at \eqref{eq:S_{R,k}} and \eqref{eq:elementary}. Indeed, \eqref{eq:elementary} can be re-written as 
\[
S_{\theta}(x,y) = \big(\one_{[y,\infty)\setminus [x,\infty)}(\theta) - \one_{[x,\infty)\setminus [y,\infty)}(\theta)\big)V(\theta, y).
\]
The form of $R(y)$ can be described explicitly in the following lemma, where we use the fact that $\r(\r(0)) = \r(0) - \r(0) = 0$.

\begin{lem}\label{lem:R(y)}
Let $R(Y) = \{k\in\R^d\,|\,\r(\Lambda(Y+k))\le0\}$, $Y\in\Y^d$, where the scalar risk measure $\r$ is decreasing and cash-invariant, and the aggregation function $\Lambda\colon\R^d\to\R$ is increasing. Then, for each $y\in\R^d$ it holds that
\[
R(y) = \Lambda^{-1}([\r(0), \infty)) - y = \Lambda^{-1}(\{\r(0)\})+\R^d_+ - y. 
\]
\end{lem}


Accounting for the sign convention that the \emph{negative} of a quantile or expectile are a scalar risk measure, one can see that the elementary scores at \eqref{eq:S_{R,k}} essentially boil down to the ones at \eqref{eq:elementary} for dimension $d=1$. 

\subsubsection{Integrability}

The non-negativity of the elementary scores at \eqref{eq:S_{R,k}} guarantees that the integral at \eqref{eq:S_R,pi} always exists. However, as stated in part (iii) of Theorem \ref{thm:elicitability}, these scores are only strictly consistent if $\bar S_{R,\pi}(R(F),F)<\infty$, which suggests the question as to when the integral at \eqref{eq:S_R,pi} is finite. A sufficient condition for the latter is that 
\(
 \int_{\R^d} |V_{R_0}(k,y)|\,\pi(\diff k)<\infty.
\)
Therefore, a sufficient condition for the finiteness of $\bar S_{R,\pi}(A,F)$ is that $V_{R_0}$ is $\pi\otimes F$-integrable.

\subsubsection{Normalisation}

By construction, the elementary scores at \eqref{eq:S_{R,k}}, and therefore the scores at \eqref{eq:S_R,pi}, are non-negative. It is well known that if a scoring function $S(x,y)$ is (strictly) $\M$-consistent for some functional $T$ then for $\lambda>0$ and some $\M$-integrable function $a\colon\O\to\R$, the score $S'(x,y) = \lambda S(x,y) + a(y)$, $\lambda>0$ is also (strictly) $\M$-consistent for $T$. Following \cite{GneitingRaftery2007} we say that $S$ and $S'$ are equivalent. Therefore, if $\M$ contains all point measures, $S'$ is strictly $\M$-consistent for $T$ and $y\mapsto S'(T(\delta_y),y)$ is $\M$-integrable, the score $S(x,y)  = S'(x,y) - S'(T(\delta_y),y)$ is non-negative by construction. However, sometimes relaxing the normalisation condition that a score be non-negative can also help to relax integrability conditions on the scoring function. A standard example is the squared loss $S(x,y) = (x-y)^2$ which is non-negative, consistent for the mean relative to any class of distributions with a finite first moment, and strictly consistent for the mean relative to any class of distributions with a finite second moment. On the other hand, the equivalent score $S'(x,y) = x^2 - 2xy$ maps to $\R$, but is \emph{strictly} consistent for the mean relative to any class of distributions with a finite first moment.

In that light, it might be interesting to consider scores $S'_{R,k}$ which are equivalent to the elementary scores at \eqref{eq:S_{R,k}}. A natural choice might be $S'_{R,k}(A,y) = -\one_A(k) V_{R_0}(k,y)$; cf.\ \cite{Dawid_Discussion2016}. This leads to an alternative mixture representation akin to  \eqref{eq:S_R,pi} of the form
\[
S'_{R,\pi}(A,y) = \int_{\R^d} S'_{R,k}(A,y)\,\pi(\diff k) = -\int_A V_{R_0}(k,y)\,\pi(\diff k).
\]
However, since the integrand $S'_{R,k}(A,y)$ may attain both positive and negative values, one needs to impose that its negative part is $\pi$-integrable in order to guarantee the existence of the integral.

\subsubsection{Characterisation of all consistent scoring functions}

There is evidence that---under appropriate regularity conditions---all consistent scoring functions for the risk measure $R$ are equivalent to a score of the form given at \eqref{eq:S_R,pi}. That means, modulo equivalence, the choice of the consistent scoring function boils down to the choice of the measure $\pi$.\\
Firstly, note that Proposition \ref{prop:characterization V} implies that it does not matter what oriented strict $\M^d$-identification $V_{R_0}$ we actually start with. Indeed, if $V'_{R_0}$ were another such identification function, then $V'_{R_0}(k,y) = h(k) V_{R_0}(k,y)$ for some positive function $h$. But this solely amounts to a change of measure, since $V_{R_0}(k,y)\pi(\diff k )  = V'_{R_0}(k,y)\pi'(\diff k )$, where $\pi'$ has the density $1/h$ with respect to $\pi$.
Secondly, the class of scoring functions of the form \eqref{eq:S_R,pi} is convex, which is a necessary condition \citep{Gneiting2011}. 
Thirdly, as observed above, the mixture representation at \eqref{eq:S_R,pi} is the natural extension of the one-dimensional case. As remarked, for the one-dimensional case, one can typically establish this sort of necessary conditions only invoking Osband's principle. Since Osband's principle relies on a first-order-condition argument, it has only been established under smoothness conditions and for the finite dimensional case. 
We suspect that it is possible to generalise it to the infinite dimensional setting of predicting upper sets in $\F(\R^d;\R^d_+)$. Possible approaches might work by borrowing ideas from the \emph{calculus of variations} or by considering increments of scores rather than derivatives. However, the technical treatment of these approaches is beyond the scope of the paper at hand such that we defer it to future research.

\subsubsection{Order-sensitivity}

It is known that---under weak assumptions on $\r$---all strictly consistent scoring functions $S$ for $\r$ are \emph{order-sensitive} or \emph{accuracy-rewarding}; see \citep[Proposition 3]{Nau1985}, \citep[Proposition 2]{Lambert2013}, \citep[Proposition 3.4]{BelliniBignozzi2015}. In the scalar setting, this property means that $x_1\leq x_2\leq \r(F)$ or $\r(F)\leq x_2\leq x_1$ implies that  $\bar{S}(x_1,F)\geq\bar{S}(x_2,F)$. While one gets this useful property essentially `for free' in the scalar case, asking for order-sensitivity in a multivariate setting is a lot more involved; see \cite{FisslerZiegel2019}. One of the main questions in the multivariate setting is which order relation to use. 
In the present situation where our exhaustive action domain consists of closed upper subsets of $\R^d$, the canonical (partial) order relation is the subset relation. That means the canonical analogue of order-sensitivity in our setting is that for any distribution $F\in\M^d$ it holds that $A\subseteq B\subseteq R(F)$ or $A\supseteq B\supseteq R(F)$ implies that $\bar{S}_{R,\pi}(A,F)\geq\bar{S}_{R,\pi}(B,F)$. The following proposition establishes that this notion of order-sensitivity is fulfilled by all scoring functions introduced in Theorem \ref{thm:elicitability}(ii). 
The proof basically exploits the orientation of the underlying identification function $V_{R_0}$, which is a similar argument to the one given in \cite{SteinwartPasinETAL2014}.


\begin{prop}\label{prop:order-sensitivity}
Let the assumptions of Theorem \ref{thm:elicitability}(ii) prevail. Then, the scoring function $S_{R,\pi}$ defined at \eqref{eq:S_R,pi} is $\M^d$-order-sensitive for $R$ in the sense that for all $A, B\in\widehat\P(\R^d;\R^d_+)$ and for all $F\in\M^d$
\be{eq:order-sensitivity}
\big(A\subseteq B\subseteq R(F) \ \text{or}\  A\supseteq B\supseteq R(F)\big)
\ \implies\ \bar{S}_{R,\pi}(A,F)\geq\bar{S}_{R,\pi}(B,F).
\ee
Under the assumptions of Theorem \ref{thm:elicitability}(iii), if $\bar S_{R,\pi}(B,F)<\infty$ and the inclusions $A\subseteq B$ or $A\supseteq B$ on the left hand side of \eqref{eq:order-sensitivity} is strict, then the inequality on the right hand side is also strict.
\end{prop}

\subsubsection{Forecast dominance and Murphy diagrams}\label{subsec:Murphy}

The notion of (strict) consistency implies that---in expectation---a correctly specified forecast will score at most as high as (strictly less than) any misspecified score. On the level of the prediction space setting \citep{GneitingRanjan2013, StraehlZiegel2017}, \citet{HolzmannEulert2014} showed that for two ideal forecasts, the one measurable with respect to a strictly larger information set is preferred under any strictly consistent scoring function; cf.\ \cite{Tsyplakov2014}. \citet{Patton2017} demonstrated that, in general, two misspecified forecasts rank differently under different (consistent) scoring functions. Therefore, the choice of the scoring function used in practice matters and secondary quality criteria besides consistency, such as translation invariance or homogeneity, may guide the decision what scoring function to use; see Section \ref{sec:secondary}. For the rare situation when one forecast scores better than another one uniformly over all consistent scoring functions, \citet{EhmETAL2016} coined the term \emph{forecast dominance}. We give the corresponding definition here for the situation of exhaustive forecasts for systemic risk measures $R$.

\begin{defn}[Dominance]
Let $Y\in \Y^d$ and $A, B$ two (stochastic) forecasts for some systemic risk measure $R$ of the form at \eqref{eq:R}, taking values in $\widehat \P(\R^d; \R^d_+)$. Assume that Assumption \eqref{ass:measurability} holds. 
Then $A$ dominates $B$ if $\E[S_{R,\pi}(A,Y)]\le \E[S_{R,\pi}(B,Y)]$ for all consistent scoring functions $S_{R,\pi}$ of the form at \eqref{eq:S_R,pi} where $\pi$ is a $\sigma$-finite non-negative measure on $\mathcal B(\R^d)$.
\end{defn}
Note that the expectations are taken over the joint distribution of the forecasts and the observation.\\
Since the scores $S_{R,\pi}$ at \eqref{eq:S_R,pi} are parametrised by the class of non-negative $\sigma$-additive measures on $\mathcal B(\R^d)$, it is not very handy to check forecast dominance in practice using the definition. To this end, the following corollary is helpful. The proof is straightforward and therefore omitted.

\begin{cor}\label{cor:Murphy}
Let $Y\in \Y^d$ and $A, B$ two (stochastic) forecasts for some systemic risk measure $R$ of the form at \eqref{eq:R}, taking values in $\widehat \P(\R^d; \R^d_+)$. Assume that Assumption \eqref{ass:measurability} holds. 
Then $A$ dominates $B$ if and only if $\E[S_{R,k}(A,Y)]\le \E[S_{R,k}(B,Y)]$ for all elementary scores $S_{R,k}$ given at \eqref{eq:S_{R,k}}, where $k\in\R^d$.
\end{cor}

Corollary \ref{cor:Murphy} opens the way to an immediate multivariate analogue of \emph{Murphy diagrams} considered in \cite{EhmETAL2016}. That is, if $A$ is a $\widehat \P(\R^d; \R^d_+)$-valued forecast of a systemic risk measure $R$ and $Y$ is the corresponding $\R^d$-valued observation of a financial system, we can consider the map
\be{eq:Murphy diagram}
\R^d\ni k\mapsto s_A(k) = \E[S_{R,k}(A,Y)]
\ee
as a diagnostic tool. For an empirical setting with forecasts $A_1, \ldots, A_N\in \widehat \P(\R^d; \R^d_+)$ and observations $Y_1, \ldots, Y_N\in\R^d$, \eqref{eq:Murphy diagram} takes the form
\be{eq:Murphy emp}
\R^d\ni k\mapsto \hat s_{N,A}(k) = \frac{1}{N} \sum_{t=1}^N S_{R,k}(A_t,Y_t).
\ee
We illustrate the usage of Murphy diagrams in a simulation study presented in Subsection \ref{sec:Murphy_sim}.

\section{Homogeneous and translation invariant scoring functions}
\label{sec:secondary}

Recall that the systemic risk measure $R$ is cash-invariant, or translation equivariant, meaning that $R(Y + k) = R(Y) - k$ for all $Y\in\Y^d$ and for all $k\in\R^d$. Moreover, if $\r$ and $\Lambda$ are positively homogeneous, so is $R$ in the sense that $R(cY) = cR(Y)$ for all $Y\in\Y^d$ and $c>0$; see Lemma \ref{lem:R pos hom}. \citet{Patton2011}, \citet{NoldeZiegel2017} and \citet{FisslerZiegel2019} have argued that it is a reasonable requirement for a scoring function that ordering different sequences of forecasts in terms of their realised scores be invariant under transformations to which the functional of interest is equivariant. Therefore, we discuss translation invariance and positive homogeneity of scores (or score differences) for $R$. We start by gathering some elementary definitions.
%
\begin{defn}[Homogeneity, translation invariance]
\begin{enumerate}[(i)]
\item We call a function $f\colon \R^d\to\R$ positively homogeneous of degree $b\in\R$ if $f(cx)=c^bf(x)$ for all $c>0$ and for all $x\in \R^d$.
\item A scalar risk measure $\r\colon\Y\to\R$ is called positively homogeneous if $\r(cX)=c\r(X)$ for all $c>0$ and for all $X\in\Y$.
\item An exhaustive scoring function $S\colon \F(\R^d;\R^d_+)\times \R^d\to\R$ is said to have positively homogeneous score differences of degree $b\in\R$ if $S(cA,cy)-S(cB,cy)=c^b(S(A,y)-S(B,y))$ for all $A,B\in \F(\R^d;\R^d_+)$, $y\in \R^d$ and $c>0$.
\item A selective identification function $V:\R^d\times\R^d\to\R$ is called translation invariant if $V(k+l,y-l)=V(k,y)$ for all $k,l,y\in\R^d$.
\item An exhaustive scoring function $S\colon \F(\R^d,\R^d_+)\to\R^d$ is said to have translation invariant score differences if $S(A+l,y-l)-S(B+l,y-l)=S(A,y)-S(B,y)$ for all $A,B\in \F(\R^d,\R^d_+)$ and $y,l\in\R^d$.
\item
A measure $\pi$ on $\mathcal{B}(\R^d)$ is translation invariant if $\pi (A) = \pi(A+l)$ for all $A\in \mathcal{B}(\R^d)$ and for all $l\in\R^d$.
\item
A measure $\pi$ on $\mathcal{B}(\R^d)$ is positively homogeneous of degree $b\in\R$ if $\pi(cA) = c^b\pi(A)$ for all $A\in \mathcal{B}(\R^d)$ and for all $c>0$.
\end{enumerate}
\end{defn}
\noindent
With these definitions in mind, we can now state the following results.

\begin{lem}\label{lem:R pos hom}
If $\r$ is a positively homogeneous scalar risk measure and $\Lambda$ is positively homogeneous of any degree $b\in\R$, $R$ as defined at \eqref{eq:R} is positively homogeneous, i.e.\ for all $c>0$ and $Y\in\Y^d$, $R(cY)=cR(Y)$.
\end{lem}

\begin{lem}\label{lem:ident}
Assume that $\rho\colon\M\to\R$ admits a strict $\M$-identification function $V_\r\colon\R\times\R\to\R$. Then the following holds for $V_{R_0}\colon \R^d\times\R^d\to\R$ defined at \eqref{eq:def V_R_0}:
\begin{enumerate}[\rm (i)]
\item 
$V_{R_0}$ is translation invariant.
\item
Assume that $\M^d$ is convex and that for any $x\in\R^d$ there are $F_1,F_2\in\M^d$ such that $\bar V_{R_0}(x,F_1)>0$ and $\bar V_{R_0}(x,F_2)<0$. Then for any translation invariant strict $\M^d$-identification function $V'_{R_0}$ for $R_0$ there is some $\lambda \neq0$ such that 
\[
\bar V'_{R_0}(x,F) = \lambda \bar V_{R_0}(x,F)
\]
for all $x\in\R^d$ and for all $F\in\M^d$.
\item 
If $V_\r(0,\cdot)\colon\R\to\R$ is positively homogeneous of degree $a\in\R$ 
and $\Lambda$ is positively homogeneous of degree $b\in\R$, then $V_{R_0}$  is positively homogeneous of degree $ab$.
\end{enumerate}
\end{lem}

\begin{rem}\label{rem:translation invariance}
Interestingly, part (i) of Lemma \ref{lem:ident} implies that if $\r\colon\M\to\R$ is identifiable and cash-invariant with a strict $\M$-identification function $V_\r\colon\R\times\R\to\R$, then $V_\r^{inv}(x,y):= V_\r(0,x+y)$ is translation invariant. This result can also be generalised to translation equivariant functionals, hence establishing the converse of \cite[Proposition 4.7(i)]{FisslerZiegel2019}.
\end{rem}

\begin{prop}\label{prop:translation invariance}
Let Assumption \eqref{ass:closedness} hold and assume that $\rho$ is identifiable with an oriented strict $\M$-identification function $V_\r$. 
\begin{enumerate}[(i)]
\item
Let $\L^d$ be the $d$-dimensional Lebesgue measure and $S_{R,k}$ be an elementary score of the form at \eqref{eq:S_{R,k}} with identification function $V_{R_0}(k,y) = V_\r(0,\Lambda(y+k))$. Then the scoring function
\be{eq:translation inv}
S_{R, \L^d}(A,y) =\int_{\R^d} S_{R,k}(A,y)\,\L^d(\diff k), \qquad A\in \F(\R^d;\R^d_+), \ y\in\R^d
\ee
is translation invariant and $\M^d$-consistent for $R$.
\item
Any finite $\M^d$-consistent scoring function $S$ for $R$ of the form at \eqref{eq:S_R,pi} is translation invariant only if $S(A,y) = \gamma S_{R, \L^d}(A,y)$ at \eqref{eq:translation inv} for some $\gamma\ge0$.
\end{enumerate}
\end{prop}

%

Note that for all examples of $\r$ and $\Lambda$ we are aware of, it holds that for the score at \eqref{eq:translation inv} $S_{R, \L^d}(A,y)$ is finite if the symmetric difference $A\triangle R(y)$ is bounded and only if it has a finite Lebesgue measure. Hence, Proposition \ref{prop:translation invariance}(ii) implies that the only finite translation invariant consistent score is the 0-score.

\begin{prop}\label{prop: positive homogeneity}
Let Assumption \eqref{ass:closedness} hold and suppose that $V_{R_0}$ is an oriented strict selective $\M^d$-identification function for $R_0$ which is positively homogeneous of degree $a\in\R$. 
\begin{enumerate}[(i)]
\item
Let $\pi$ be a non-negative $\sigma$-finite positively homogeneous measure of degree $b\in\R$ and $S_{R,k}$ be an elementary score of the form at \eqref{eq:S_{R,k}} with identification function $V_{R_0}$. Then the scoring function
\be{eq:homogeneous}
S_{R, \pi}(A,y) =\int_{\R^d} S_{R,k}(A,y)\,\pi(\diff k), \qquad A\in \F(\R^d;\R^d_+), \ y\in\R^d
\ee
is positively homogeneous of degree $a+b$.
\item
Any finite $\M^d$-consistent scoring function $S$ for $R$ of the form at \eqref{eq:S_R,pi} is positively homogeneous of degree $a+b$ only if $S(A,y) = \gamma S_{R, \pi}(A,y)$ at \eqref{eq:homogeneous} for some $\gamma\ge0$ and for some non-negative $\sigma$-finite positively homogeneous measure $\pi$ of degree $b\in\R$.
\end{enumerate}
\end{prop}

\begin{rem}
For many measures $\pi$ and sets $A$ the score $S_{R,\pi}(A,y)$ defined at \eqref{eq:S_R,pi} might not be finite which diminishes the practical statistical applicability in the context of forecast comparison. More to the point, score differences involving $S_{R,\pi}(A,y)$ will not be finite or might even not be defined at all. To overcome this issue we suggest to work with the following convention of score differences
\begin{align}\label{eq: score difference}
S_{R,\pi}(A,y) - S_{R,\pi}(B,y) &:= \int_{\R^d} S_{R,k}(A,y) - S_{R,k}(B,y)\pi(\diff k)\\ \nonumber
&= \int_{B\setminus A} V_{R_0}(k,y)\pi(\diff k) - \int_{A\setminus B} V_{R_0}(k,y)\pi(\diff k),
\end{align}
where $S_{R,k}$ are the elementary scores defined at \eqref{eq:S_{R,k}}, which assume finite values only. Indeed, the integral at \eqref{eq: score difference} might exist and might even be finite, even if $S_{R,\pi}(A,y)$ or $S_{R,\pi}(B,y)$ are $\infty$. This can be particularly helpful when working with translation invariant or positively homogeneous scores.
\end{rem}

\section{Elicitability of systemic risk measures based on Expected Shortfall}\label{sec:ES}

The two most common scalar risk measures in quantitative risk management are Value at Risk ($\VaR_\a$) and Expected Shortfall ($\ES_\a$) at some level $\a\in(0,1)$. Both are law-invariant scalar risk measures such that we can define them directly as functionals on appropriate classes of distributions. For a probability distribution function $F$ and $\a\in(0,1)$ we define
\begin{align}\label{eq:VaR}
\VaR_\a(F) &= - \inf\{x\in\R\,|\,\a\le F(x)\},\\ \nonumber
\ES_\a(F) &= \frac{1}{\a}\int_0^{\a} \VaR_{\beta}(F) \dint \beta \\ \nonumber
&= -\frac{1}{\a}\int_{(-\infty, - \VaR_\a(F)]} x\dint F(x) - \frac{1}{\a}\VaR_\a(F) \big( F(-\VaR_\a(F)) - \a\big).
\end{align}
The last decade has seen quite a lively debate about which scalar risk measure is best to use in practice where the debate has mainly focused on the dichotomy of $\VaR_\a$ and $\ES_\a$; see \cite{EmbrechtsETAL2014} and \cite{EmmerKratzTasche2015} for a comprehensive academic discussion and \cite{BIS2014} for a regulatory perspective. $\VaR_\a$ is robust in the sense of \citet{Hampel1971}, but ignores losses beyond the level $\a$. Moreover, \citet{ContDeguestETAL2010} showed that robustness and coherence are mutually exclusive implying that $\VaR_\a$ fails to be coherent. On the other hand, $\ES_\a$ is a coherent---thus non-robust---risk measure. As a tail expectation, it takes into account the losses beyond the level $\a$ by definition. 
Another layer of the joust between the two risk measures is their backtestability \citep{AcerbiSzekely2014, AcerbiSzekely2017}. While the identifiability of a risk measure is important, but not necessary for traditional backtesting, comparative backtesting relies on the elicitability of the risk measure at hand; see \cite{FisslerETAL2016} and \cite{NoldeZiegel2017}. \\
As the negative of a selection of the $\a$-quantile, $\VaR_\a$ is elicitable on any class of distributions with a unique $\a$-quantile. In stark contrast, \citet{Gneiting2011} demonstrated that $\ES_\a$ does generally not satisfy the CxLS property which rules out its elicitability; cf.\ \cite{Weber2006}. 

Recall that Theorem \ref{thm:identification R_0} and Theorem \ref{thm:elicitability} establish identifiability and elicitability results for systemic risk measures based on a scalar risk measure $\rho$ which is identifiable, and therefore---under weak regularity assumption---elicitable; see \cite{SteinwartPasinETAL2014}. Moreover, Proposition \ref{prop:R0 to rho} establishes that, under weak regularity conditions, the identifiability\,/\,elicitability of $\rho$ is also necessary for the identifiability and elicitability of the systemic risk measure based on $\rho$. 
Therefore, for some aggregation function $\Lambda\colon\R^d\to\R$, the systemic risk measure $R^{\ES_\a}(Y) = \{k\in\R^d\,|\,\ES_\a(\Lambda(Y+k))\le0\}$, $Y\in\Y^d$, generally fails to be elicitable. On the other hand, for scalar risk measures, \citet{FisslerZiegel2016} established that the pair $(\VaR_\a, \ES_\a)$ is elicitable under weak regularity conditions; cf.\ \cite{AcerbiSzekely2014}. This might trigger the suspicion that the pair $\big(R^{\VaR_\a}, R^{\ES_\a}\big)$ mapping to the product space $\F(\R^d; \R^d_+)\times \F(\R^d; \R^d_+)$ is exhaustively elicitable. We conjecture, however, that $\big(R^{\VaR_\a}, R^{\ES_\a}\big)$, in general, fails to have the exhaustive CxLS property for $d\ge2$, ruling out its exhaustive elicitability. 

Therefore, we need slightly more information than $R^{\VaR_\a}(Y) = \{k\in\R^d\,|\,\VaR_\a(\Lambda(Y+k))\le0\}$ in the other component to render the pair involving $R^{\ES_\a}$ elicitable. Note that $R^{\VaR_\a}(Y)$ only encodes information about the sign of $\VaR_\a(\Lambda(Y+k))$ for each $k\in\R^d$. Apart from $k$ in the boundary of $R^{\VaR_\a}(Y)$ we know nothing about the actual size of $\VaR_\a(\Lambda(Y+k))$. However, the positive result about the elicitability of the pair $(\VaR_\a, \ES_\a)$ actually exploits the fact that for the scoring function $S_\a(x,y) = -(\one\{y\le - x\} - \a)x/\a - \one\{y\le -x\}y/\a$, $x,y\in\R$, $\VaR_\a(F)$ is the \emph{minimiser} of the expected score while $\ES_\a(F)$ is its \emph{minimum}; see \cite{FrongilloKash2015}. Therefore, we shall consider the function-valued functional $T^{\VaR_\a}\colon \Y^d\to \R^{\R^d}$ where for each $Y\in\Y^d$
\be{eq:T VaR}
T^{\VaR_\a}(Y)\colon \R^d\to \R, \quad \R^d\ni k \mapsto T^{\VaR_\a}(Y)(k)=\VaR_\a(\Lambda(Y+k)).
\ee

\subsection{Identifiability results}

To simplify the exposition of the results, we shall make the following assumption about the class $\M$.

\begin{ass}\label{ass:regularity}
All distribution functions in $\M$ are continuous and strictly increasing. 
\end{ass}
\noindent
Note that this assumption imposes also implicit restrictions on the class $\M^d$ since we assume that for any $Y$ with distribution in $\M^d$, the random variable $\Lambda(Y+k)$ has a distribution in $\M$ for any $k\in\R^d$.

A strict $\M$-identification function $V\colon \R^2\times\R\to\R^2$ for the pair $(\VaR_\a,\ES_\a): \M\to \R^2$ is given in terms of 
\[
V(v,e,y) = 
\begin{pmatrix}
\a-\one\{y + v\le 0\} \\
e +\one\{y + v\le 0\}y/\a+ (\one\{y+ v\le 0 \} - \a)v/\a
\end{pmatrix},
\]
$(v,e)\in\R^2$, $y\in\R$,
which can be verified by a straightforward calculation.
This induces a (non-strict) selective $\M^d$-identification function $U\colon \R^{\R^d}\times\R^d\times \R^d\to\R^2$ for $(T^{\VaR_\a}, R^{\ES_\a}_0)\colon \M^d\to \R^{\R^d}\times 2^{\R^d}$. For $v\colon\R^d\to\R$, $k\in\R^d$ and $y\in\R^d$ it is defined by
\begin{multline}\label{eq:U a}
U(v,k,y) =\\
\begin{pmatrix}
\a-\one\{\Lambda(y+k) \le -v(k)\} \\[0.2em]
\one\{\Lambda(y+k)\le -v(k)\}\Lambda(y+k)/\a+\big(\one\{\Lambda(y+k)\le -v(k)\} - \a\big)v(k)/ \a
\end{pmatrix}.
\end{multline}

\begin{prop}\label{prop:ES ident}
For any $F\in\M^d$, the component $U_2$ of $U$ defined at \eqref{eq:U a} is oriented in the sense that for any $k\in\R^d$
\be{eq:U orientation}
\bar U_2(T^{\VaR_\a}(F),k,F)
\begin{cases}
<0, & \text{if } k\notin R^{\ES_\a}(F)\\
=0, & \text{if } k\in R^{\ES_\a}_0(F)\\
>0, & \text{if } k\in R^{\ES_\a}(F)\setminus R^{\ES_\a}_0(F).
\end{cases}
\ee
Under Assumption \eqref{ass:regularity} the map $U$ is a selective $\M^d$-identification function for the functional $(T^{\VaR_\a}, R^{\ES_\a}_0)\colon \M^d\to \R^{\R^d}\times 2^{\R^d}$. 
\end{prop}

\subsection{Elicitability results}

We introduce the following regularity assumption on $T^{\VaR_\a}$ defined at \eqref{eq:T VaR}.

\begin{ass}\label{ass:T VaR}
The functional $T^{\VaR_\a}\colon\M^d\to \R^{\R^d}$ takes only values in $\mathcal C(\R^d;\R)$, the space of continuous functions from $\R^d$ to $\R$.
\end{ass}
With a standard argument one can verify that Assumption \eqref{ass:regularity} together with the continuity of $\Lambda$ imply Assumption \eqref{ass:T VaR}.\\
In order to present the following theorem more compactly, let us introduce $S_{\a,g}(x,y) = (\one\{y\le x\} - \a)(g(x) - g(y))$ for any increasing function $g\colon\R\to\R$. Recall that $S_{\a,g}$ is a non-negative consistent selective scoring function for the $\a$-quantile. Moreover, if $g$ is strictly increasing, $S_g$ is a strictly consistent selective scoring function for the $\a$-quantile relative to any class $\M$ of distributions such that $g$ is $\M$-integrable; see \cite{Gneiting2011b}.

\begin{thm}\label{thm:ES elicitability}
\begin{enumerate}[\rm (i)]
\item
Under Assumption \eqref{ass:measurability}, for every $k\in\R^d$ the function
$S_k\colon \R^{\R^d}\times \widehat \P(\R^d;\R^d_+) \times \R^d\to[0,\infty)$,
\be{eq:S VaR ES elementary}
S_k(v,A,y) = -\one_A(k)U_2(v,k,y) - \one_{R^{\ES_\a}(y)}(k)\Lambda(y+k)
\ee
is a non-negative $\M^d$-consistent exhaustive scoring function for the functional $(T^{\VaR_\a}, R^{\ES_\a})\colon\M^d\to \R^{\R^d} \times \widehat \P(\R^d;\R^d_+)$.
\item
Under Assumption \eqref{ass:measurability} and if $\pi_1, \pi_2$ are $\sigma$-finite non-negative measures on $\mathcal B(\R^d)$, the map $S_{\pi_1,\pi_2} \colon \R^{\R^d}\times \widehat \P(\R^d;\R^d_+) \times \R^d\to[0,\infty]$,
\be{eq:S VaR ES}
S_{\pi_1,\pi_2}(v,A,y) = \int_{\R^d} S_{\a,g_k}(-v(k), \Lambda(y+k))\,\pi_1(\diff k) +\int_{\R^d} S_k(v,A,y)\,\pi_2(\diff k),
\ee
where for each $k\in\R^d$ the function $g_k\colon\R\to\R$ is non-decreasing and $S_k$ is given at \eqref{eq:S VaR ES elementary}, is a non-negative $\M^d$-consistent exhaustive scoring function for $(T^{\VaR_\a}, R^{\ES_\a})\colon\M^d\to \R^{\R^d} \times \widehat \P(\R^d;\R^d_+)$.
\item
If Assumptions \eqref{ass:closedness}, \eqref{ass:regularity}, and \eqref{ass:T VaR} hold, if $g_k$ is strictly increasing for all $k\in\R^d$ and if $\pi_1$, $\pi_2$ are strictly positive, then the restriction of $S_{\pi_1,\pi_2}$ defined at \eqref{eq:S VaR ES} to $\mathcal C(\R^d;\R)\times \F(\R^d;\R^d_+)\times \R^d$ is a non-negative strictly $\M_0^d$-consistent exhaustive scoring function for $(T^{\VaR_\a}, R^{\ES_\a})\colon\M_0^d\to \mathcal C(\R^d;\R)\times \F(\R^d;\R^d_+)$, where $\M_0^d\subseteq \M^d$ is such that the inequality $\bar S_{\pi_1,\pi_2}(T^{\VaR_\a}(F), R^{\ES_\a}(F),F)<\infty$ holds for all $F\in\M^d_0$.
\end{enumerate}
\end{thm}

Theorem \ref{thm:ES elicitability}(ii) suggests that there is again the possibility to consider Murphy diagrams to assess the quality of forecasts for $(T^{\VaR_\a}, R^{\ES_\a})$ simultaneously over all scoring functions given at \eqref{eq:S VaR ES}. However, a direct implementation would amount to defining them on the $2d$-dimensional Euclidean space. If one further decomposes the functions $g_k$ in the spirit of \cite{EhmETAL2016}, one would even end up with a map defined on $\R\times\R^d\times\R^d$. However, arguing along the lines of \cite{ZiegelETAL2019}, the measure $\pi_1$ only accounts for forecast accuracy in the Value at Risk component. Therefore, if interest focuses on the Expected Shortfall component, it makes sense to set $\pi_1=0$ to facilitate the analysis. This implies that one can consider the Murphy diagram
\[
\R^d\ni k \mapsto \E[S_k(v,A,Y)]
\]
with the elementary scores $S_k$ given at \eqref{eq:S VaR ES elementary}. The empirical formulation in the spirit of \eqref{eq:Murphy emp} is straight forward.

\section{Examples and simulations}\label{sec:simulations}
\subsection{Consistency of the exhaustive scoring function for $R$}
\label{subsec:simulation1}
In this subsection, we shall demonstrate the discrimination ability of the consistent exhaustive scoring functions constructed in Theorem \ref{thm:elicitability} via a simulation study. We shall do so in the context of the prediction space setting introduced in \cite{GneitingRanjan2013}. That means we explicitly model the information sets of each forecaster. For the sake of simplicity and following \cite{GneitingETAL2007} and \cite{FisslerZiegel2019_RVaR} we choose to consider prediction-observation-sequences that are independent and identically distributed over time. 
Despite this simplification, there is still a variety of parameters to consider in the simulation study:
\begin{enumerate}[(i)]
\item the dimension of the financial system $d$,
\item the (unconditional) distribution of $Y_t$,
\item the aggregation function $\Lambda$;
\item the scalar risk measure $\r$;
\item
the competing forecasts $A_t$ and $B_t$, along with their joint distributions with $Y_t$;
\item the measure $\pi$ (and thus the scoring function $S_{R,\pi}$);
\item the time horizon $N$.
\end{enumerate}
We confine ourselves to the following choices of these parameters.
\begin{enumerate}
\item[(i)--(iii)] 
We work with two different combinations of $Y_t$ and $\Lambda$. In both cases, we work with a system with $d=5$ participants. 
\begin{enumerate}[(a)]
\item The vector $Y_t$ models the gains and losses of the participants in the system. At any time point $t$, $Y_t=\mu_t+\epsilon_t$ where $\mu_t$ follows a $5$-dimensional normal distribution with mean 0, correlations $0.5$ and variances $1$, and $\epsilon_t$ follows a $5$-dimensional standard normal distribution. Moreover, $\mu_t$ and $\epsilon_t$ are independent for all $t$.
Thus, conditionally on $\mu_t$, $Y_t$ has distribution $\mathcal N_5(\mu_t,I_5)$, whereas unconditionally, $Y_t\sim \mathcal N_5(0,\Sigma)$ with $(\Sigma)_{ij}=0.5$ for $i, j=1,\ldots,5$, $i\neq j$ and 2 otherwise. The aggregation function $\Lambda_1$ is of the form $\Lambda_1(Y_t)=(1-\beta)\sum_{i=1}^dY_{i,t}^+-\beta\sum_{i=1}^dY_{i,t}^-$, as suggested in \cite{Amini2015}, and we set $\beta=0.75$. This way, both gains and losses influence the value of the aggregation function, however, the losses have a higher weight. Here and in what follows, $x^+$ and $x^-$ denote the positive and negative parts of $x$, such that $x^+=\max(0,x)$ and $x^-=-\min(0,x)$.
\item We consider an extended model of \citet{EisenbergNoe}; see \cite{FeinsteinRudloffWeber2017}: The participants have liabilities towards each other, $L_{ij,t}$ represents the nominal liability of participant $i$ towards participant $j$ at time point $t$, $i,j=1, \ldots, 5$. Moreover, each participant $i$ owes an amount $L_{is,t}$ to society at time point $t$. To simplify the simulations and shorten the computing time, we assume that the liabilities matrix is deterministic and constant in time, so that we can write $L_{is}$ instead of $L_{is,t}$. Moreover, we denote by $\bar{L}_s$ the sum of all payments promised to society, i.e., $\bar{L}_s=\sum_{i=1}^d L_{is}$. The vector $Y_t$ represents the endowments of the participants at time point $t$. As suggested in \cite{EisenbergNoe}, if some of the endowments are negative, we introduce a so called sink node and interpret the negative endowments as liabilities towards this node. The value of the aggregation function $\Lambda_2$ corresponds to the sum of all payments society obtains in the clearing process as described in \cite{EisenbergNoe}. To simulate the endowments of the participants $Y_t$, we assume that $Y_{it}=(\mu_{it}+\epsilon_{it})^2$ for $i=1,\ldots,5$ with $\mu_t$ and $\epsilon_t$ specified in (a). We construct the system in the following way:
\begin{itemize}
\item The probability of a participant owing to another participant is 0.8. If there is a liability from $i$ to $j$, its nominal value is 2.
\item In addition, each participant owes 2 to the society.
\end{itemize}
\end{enumerate}
\item[(iv)] In setting (a), we consider the scalar risk measures $\VaR_{\alpha}$, $\a\in(0,1)$, defined at \eqref{eq:VaR}, and its expectile-based version defined as $\EVaR_{\tau}(X)=-e_{\tau}(X)$, $\tau\in(0,1)$, where $e_{\tau}$ satisfies the equation $\tau\E[(X-e_{\tau})^+]=(1-\tau)\E[(X-e_{\tau})^-]$ \citep{NeweyPowell1987}. For the interpretation of expectile-based risk measures in finance we refer to \cite{BelliniDiBernardino2015} and to \cite{EhmETAL2016} for a novel economic angle on expectiles. 
In case (b), however, the aggregation function $\Lambda_2$ takes nonnegative values only and therefore any financial system would be deemed acceptable when working with $\r = \VaR_{\alpha}$ or $\r =\EVaR_{\tau}$. 
Following \cite{FeinsteinRudloffWeber2017} we overcome this issue by considering the shifted risk measure $\tilde{\r}(X) = \r(X) + 0.9\bar{L}_s$, where $\r = \VaR_\a$ or $\r = \EVaR_\tau$, thus considering the system acceptable if $\VaR_\alpha$ or $\EVaR_\tau$ of the amount that society obtains from the nodes is at most $-0.9\bar{L}_s$.
Using the standard identification functions for $\VaR_{\alpha}$ and $\EVaR_{\tau}$ \citep{Gneiting2011}, the selective identification functions for $R_0$ are the following:
\begin{itemize}
\item for $\r(X)=\VaR_{\alpha}(X)+a$: 
\be{eq:identVaR}
V_{R_0}(k,y)=\alpha-\one\{\Lambda(k+y)-a\leq0\};
\ee
\item for $\r(X)=\EVaR_{\tau}(X)+a$: 
\be{eq:identEVaR}
V_{R_0}(k,y)=\tau(\Lambda(k+y)-a)^+-(1-\tau)(\Lambda(k+y)-a)^-.
\ee
\end{itemize}
\item[(v)] We consider two ideal forecasters with different information sets: Anne has access to $\mu_t$ and uses the correct conditional distribution of $Y_t$ given $\mu_t$ for her predictions. That is, she issues $A_t = R(\mathcal N_5(\mu_t,I_5)) = R(\mathcal N_5(0_5,I_5)) - \mu_t$ in case (a) and $A_t=R(\mathcal N_5(\mu_t,I_5)^2)$ in case (b) 
 for each $t=1, \ldots, N$. Here, we use the notation $\mathcal N_d(m,\Sigma)^2$ for the distribution of a random variable $Y=X^2$ where $X\sim\mathcal N_d(m,\Sigma)$.  Bob is uninformed and issues the climatological forecast. That is, he uses the correct unconditional distribution of $Y_t$ for his forecasts. Therefore, he constantly predicts $B_t = R(\mathcal N(0_5,\Sigma))$ in case (a) and $B_t=R(\mathcal N(0_5,\Sigma)^2)$ in case (b). 
\item[(vi)] 
We choose $\pi$ to be a 5-dimensional Gaussian measure with mean $m\in\R^5$ and covariance $I_5$.
To enhance the discrimination ability of the score $S_{R,\pi}$, we aim at choosing $m$ close to the boundary of $R(Y_t)$. Here we work with $m=2\cdot\mathbf{1}$ as this value appears to be fairly close to the (deterministic) forecasts of Bob in all four cases. 
This choice of $\pi$ turns out to be beneficial with respect to the integrability considerations and renders our scores finite. Indeed, since $V_{R_0}$ for $\r=\VaR_\alpha+a$ is bounded, it is $\pi\otimes F$-integrable for any finite measure $\pi$. In the case of $\r=\EVaR_\tau+a$, more considerations are necessary. From the construction of $\Lambda_2$ it is clear that it is a bounded function, in particular, the values lie in the interval $\left[0,\sum_{i=1}^d L_{is}\right]$. This in turn implies that the identification function $V_{R_0}$ is bounded. 
Therefore $V_{R_0}$ is $\pi\otimes F$-integrable for any finite measure $\pi$.
Finally, since $\Lambda_1$ only grows linearly and both $\pi$ and $Y_t$ are Gaussian, the integrability is also guaranteed in this case.
\item[(vii)] We work with sample sizes $N=250$, being a good proxy for the number of working (and trading) days in a year.
\end{enumerate}
To compare Anne's with Bob's forecast performance, we employ the classical Diebold-Mariano test \citep{DieboldMariano1995} based on the scoring functions $S_{R,\pi}$ of the form at \eqref{eq:S_R,pi}
arising from our choice of $\pi$ and identification functions introduced in \eqref{eq:identVaR} and \eqref{eq:identEVaR}. We repeat the experiment $1\,000$ times for setting (a) and $100$ times for setting (b), since due to the presence of clearing, the computation time tends to be quite lengthy in setting (b). We approximate $\pi$ with a Monte Carlo draw of size $100\,000$. The computations are performed with the statistics software \texttt{R}, and in particular its \texttt{Rcpp} package to also integrate parts of \texttt{C++} code to enhance the computational speed.

We consider tests with two different one-sided null hypotheses. The null hypothesis $H_0\colon\E\left[S_{R,\pi}(A_1,Y_1)\right]\geq \E\left[S_{R,\pi}(B_1,Y_1)\right]$, or in short $H_0\colon A\succeq B$, means that Bob has a better forecast performance than Anne, evaluated in terms of $S_{R,\pi}$. On the contrary, $H_0\colon A\preceq B$ stands for $H_0\colon\E\left[S_{R,\pi}(A_1,Y_1)\right]\leq \E\left[S_{R,\pi}(B_1,Y_1)\right]$ asserting that Anne's forecasts are superior to Bob's in terms of $S_{R,\pi}$.
In Table \ref{tab:results} we report the relative frequencies of the rejections for the respective null hypotheses. Invoking the sensitivity of consistent scoring functions with respect to increasing information sets established in \cite{HolzmannEulert2014}, we expect that Anne's forecasts are deemed superior to Bob's predictions. And in fact, the null $A\preceq B$ is never rejected for either scenario, while $A\succeq B$ is rejected in between 74\% and 100\% of all experiments over the various scenarios. 
In particular, with rejection rates for $H_0\colon A\succeq B$ between 0.94 and 1, we observe that the discrimination ability between Bob and Anne is considerably higher for model (a) as opposed to (b) where we yield rejection rates ranging from 0.74 to 0.90. This might be due to the fact that $\Lambda_1$ is unbounded whereas $\Lambda_2$ only takes values between 0 and $\bar{L}_s$, which might translate into a smaller influence of the predictive distributions upon which the forecasts are based. Moreover, both in case (a) and (b), the number of instances when Anne's forecasts are preferred over Bob's ones is higher for $\rho=\EVaR_\a+a$ than for $\rho=\VaR_\a+a$. 
\begin{table}
\begin{center}
\begin{tabular}{c|c|cccc}
&$H_0$ & $\VaR_{0.01}$ & $\VaR_{0.05}$ & $\EVaR_{0.01}$ & $\EVaR_{0.05}$\\
\hline
\multirow{2}{*}{$\Lambda_1$}&$A\succeq B$& $0.995$ & $0.940$ & $1.000$ & $1.000$ \\
&$A\preceq B$& $0.000$ & $0.000$ & $0.000$ & $0.000$ \\
\hline
\multirow{2}{*}{$\Lambda_2$}&$A\succeq B$& $0.740$ &$0.870$ & $0.790$ & $0.900$  \\
&$A\preceq B$&  $0.000$ & $0.000$ & $0.000$ & $0.000$  \\
\hline
\end{tabular}
\caption{Ratios of rejections of the null hypotheses at significance level $0.05$.}
\label{tab:results}
\end{center}
\end{table}

\subsection{Murphy diagrams}\label{sec:Murphy_sim}

\begin{figure}
\begin{center}
\begin{tabular}{ c c }
      \includegraphics[width=6cm, height=5.5cm]{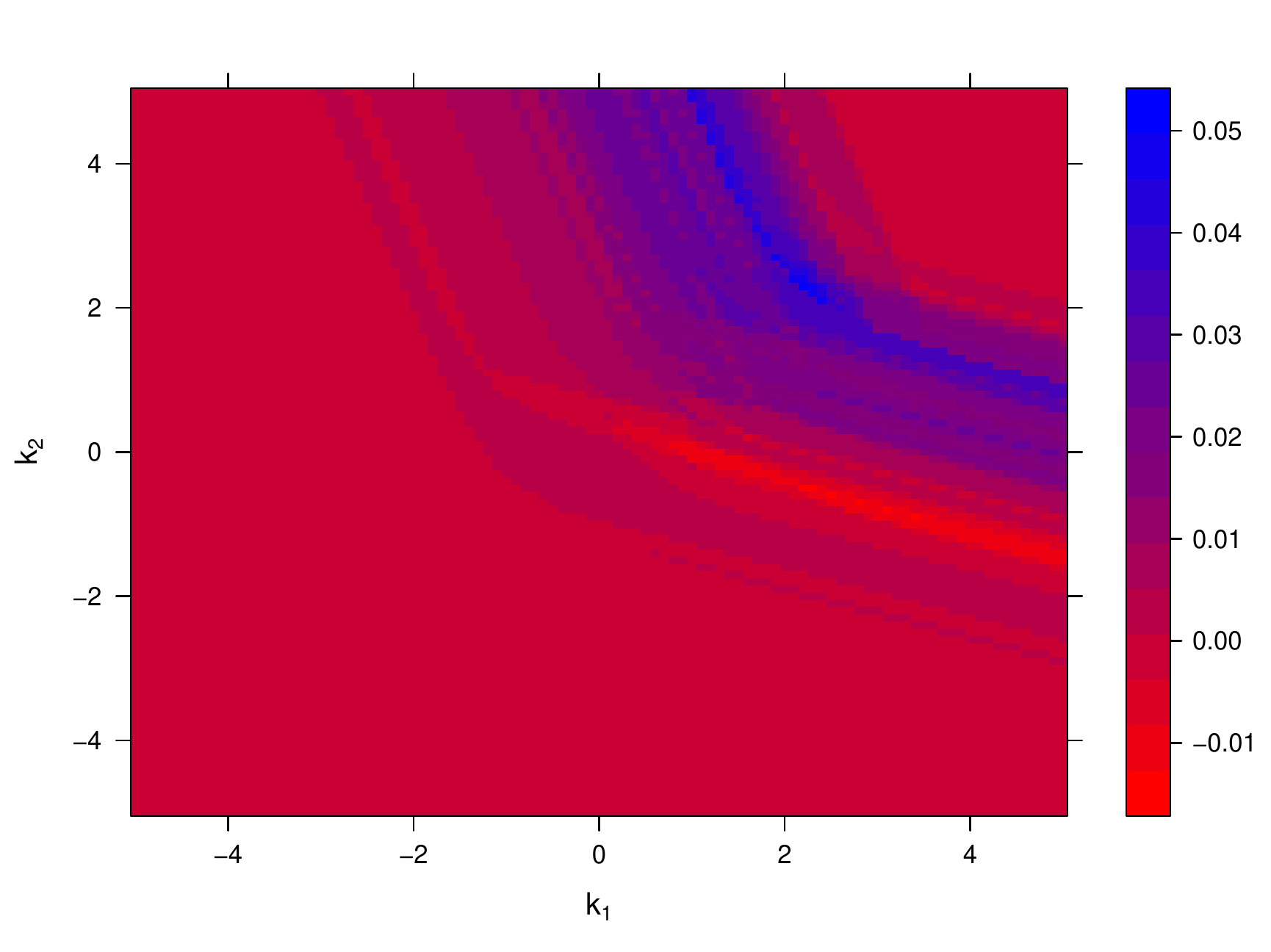} &
      \includegraphics[width=6cm, height=5.5cm]{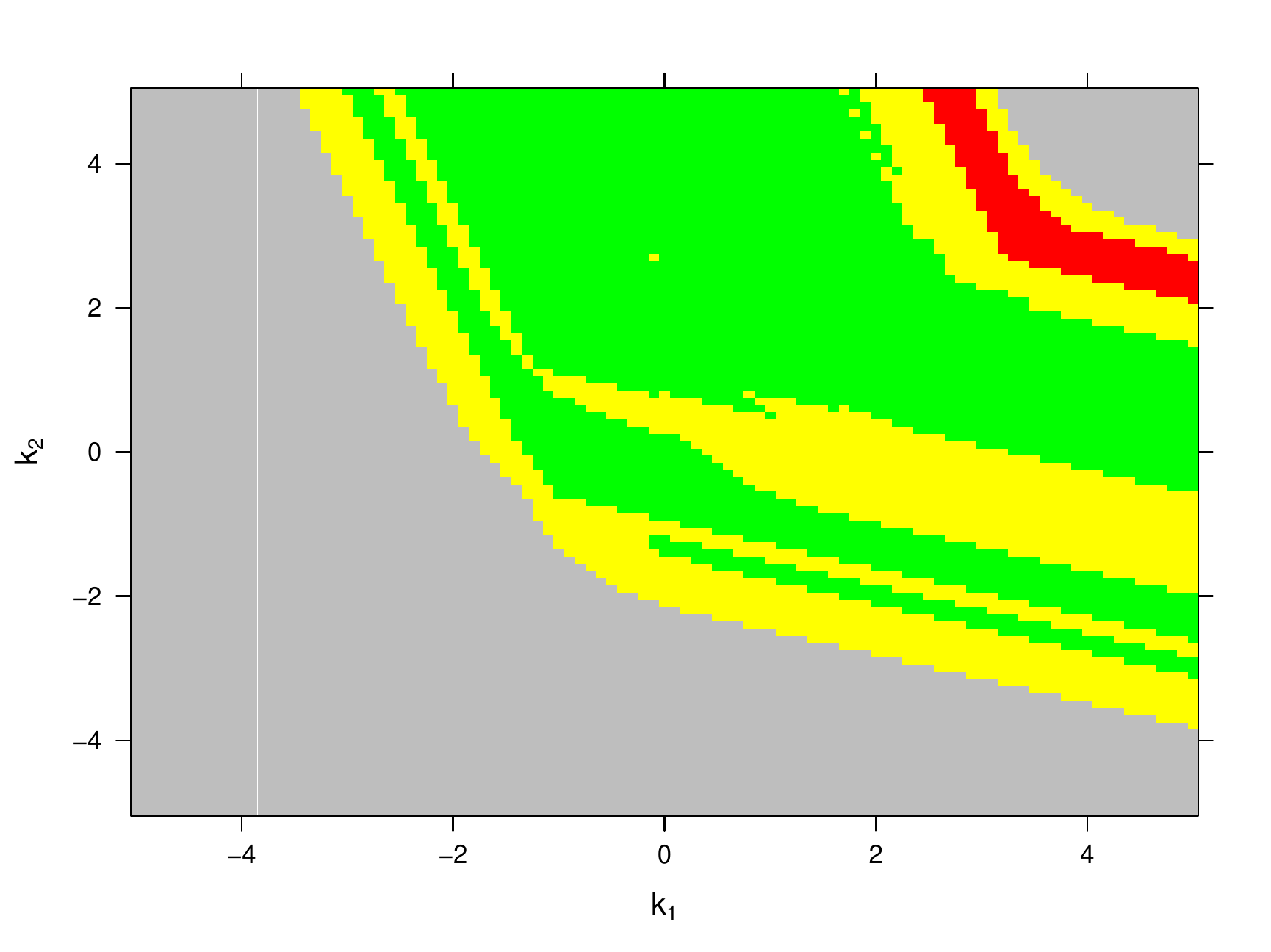} \\
			\multicolumn{2}{c}{\footnotesize $f_{1t}=B_t$, $f_{2t}=A_t$} \\
      \includegraphics[width=6cm, height=5.5cm]{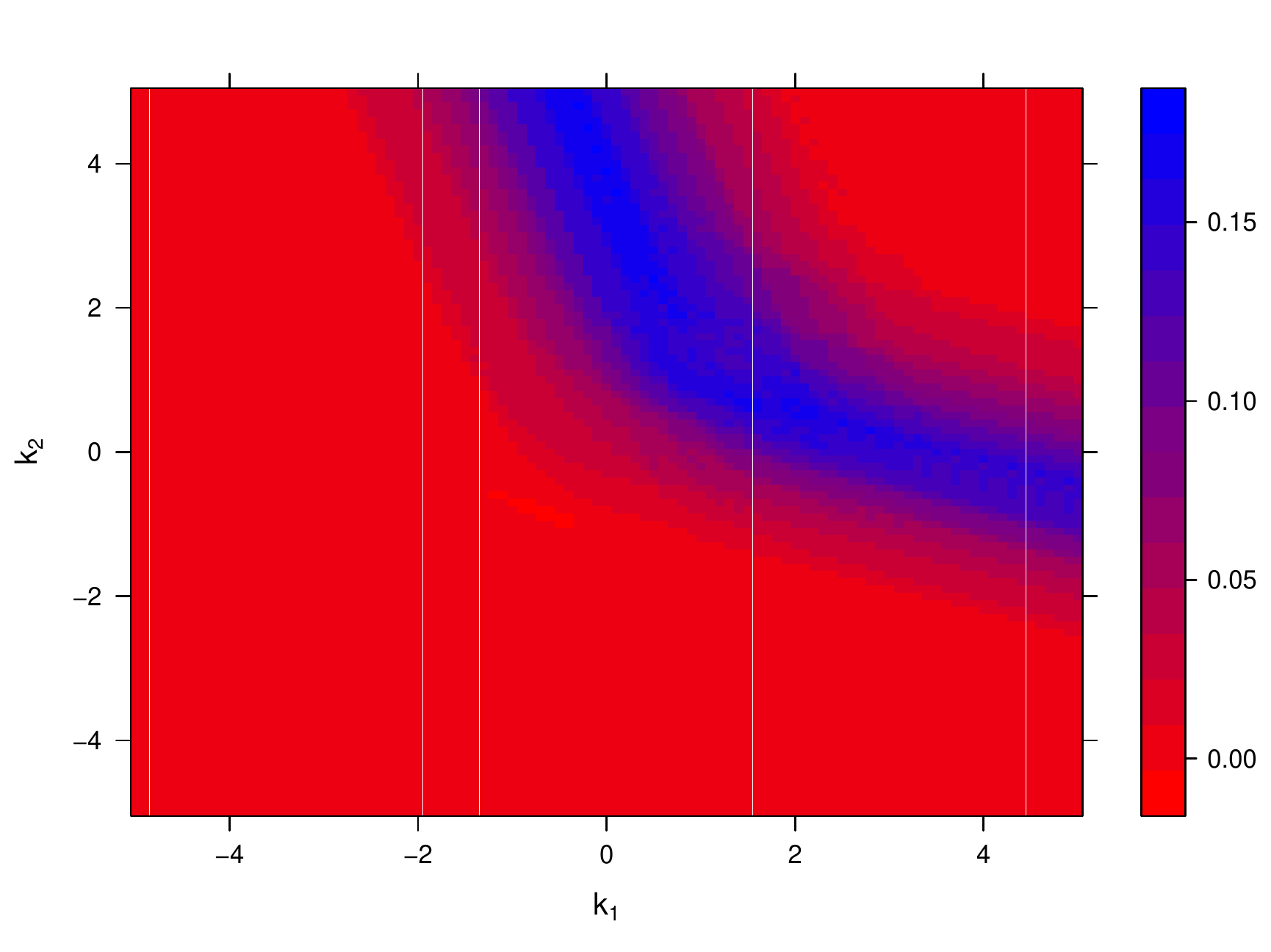} &
      \includegraphics[width=6cm, height=5.5cm]{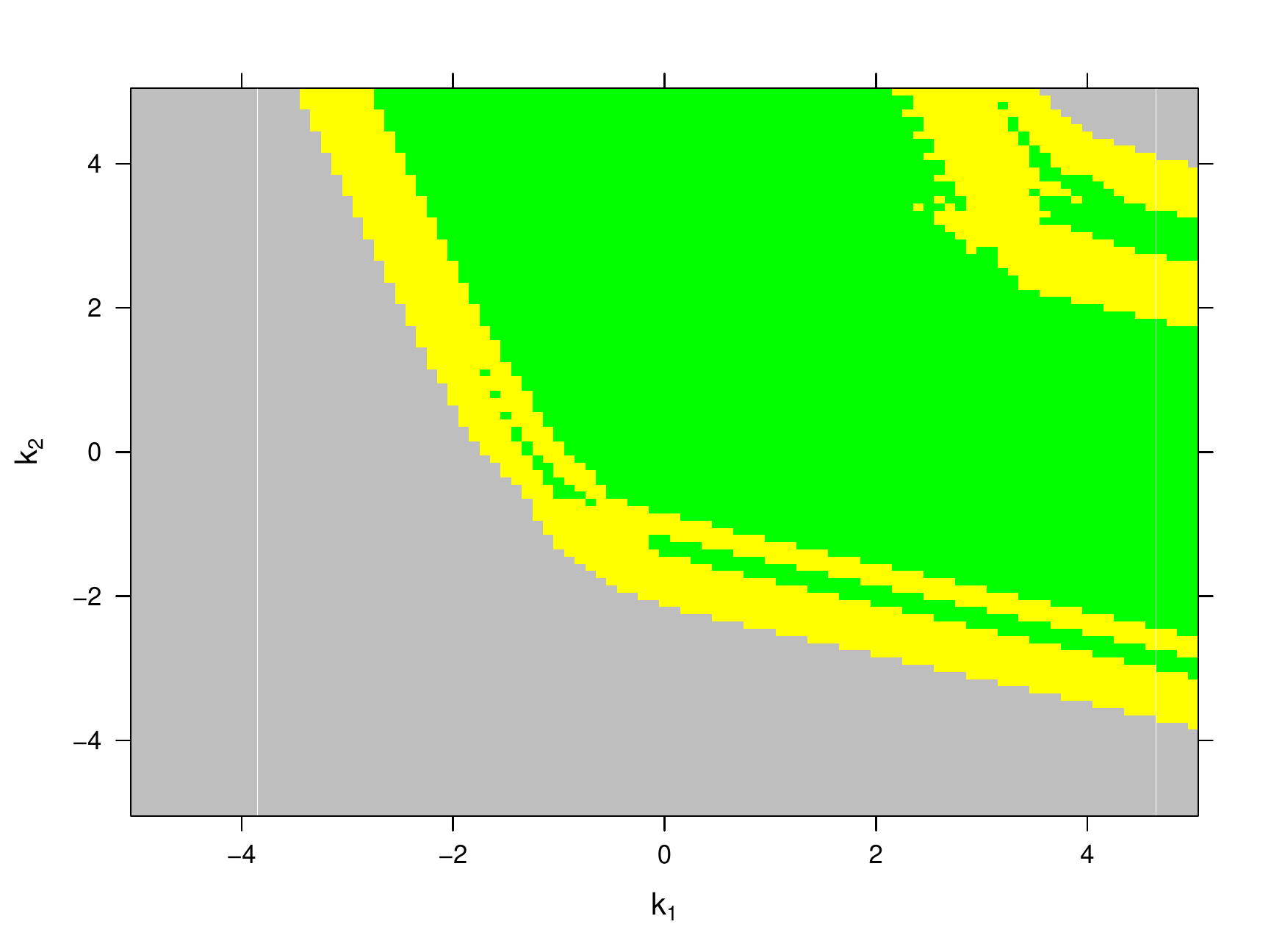} \\
\multicolumn{2}{c}{\footnotesize $f_{1t}=C_t$, $f_{2t}=A_t$}\\
      \includegraphics[width=6cm, height=5.5cm]{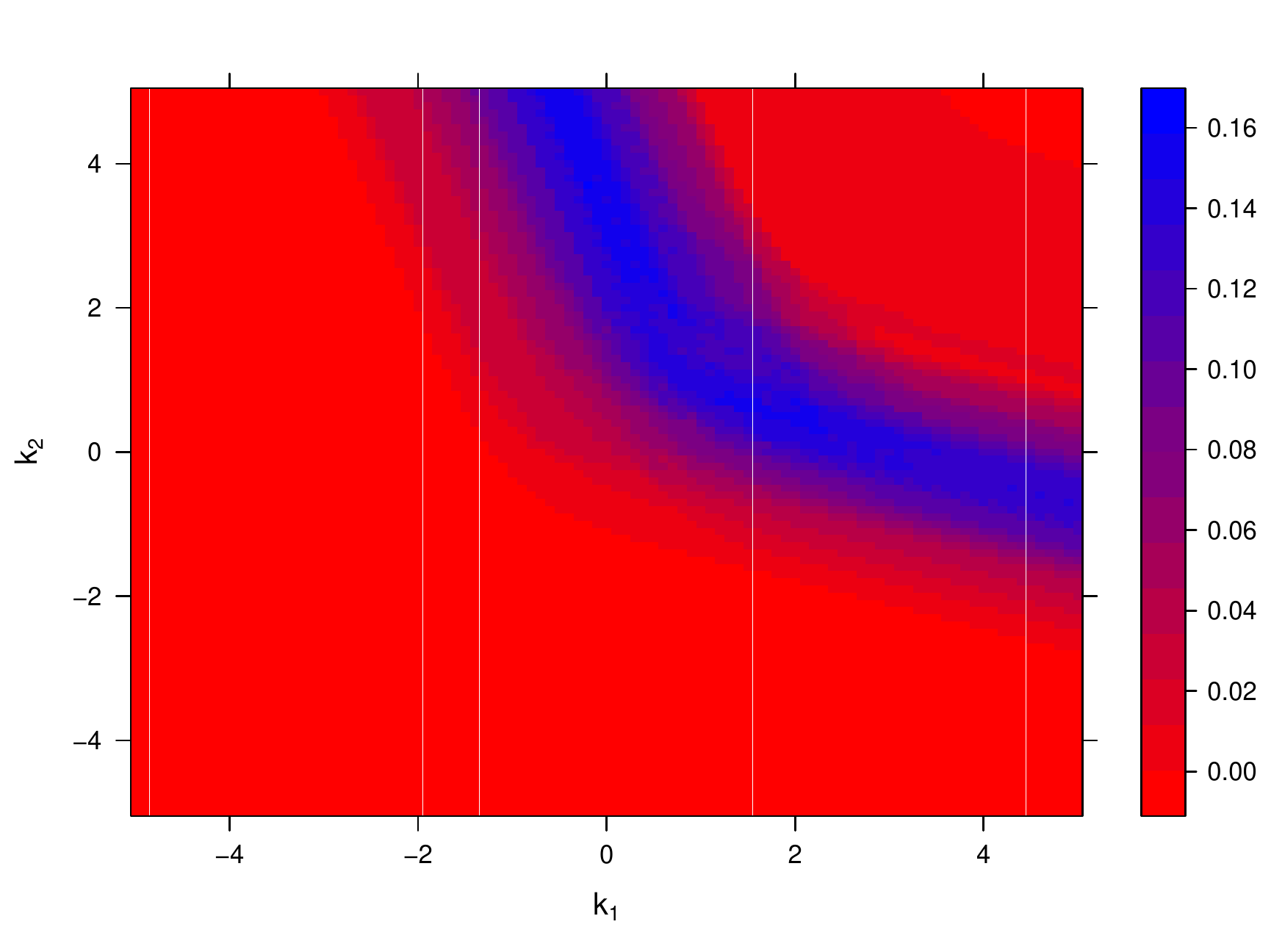} &
      \includegraphics[width=6cm, height=5.5cm]{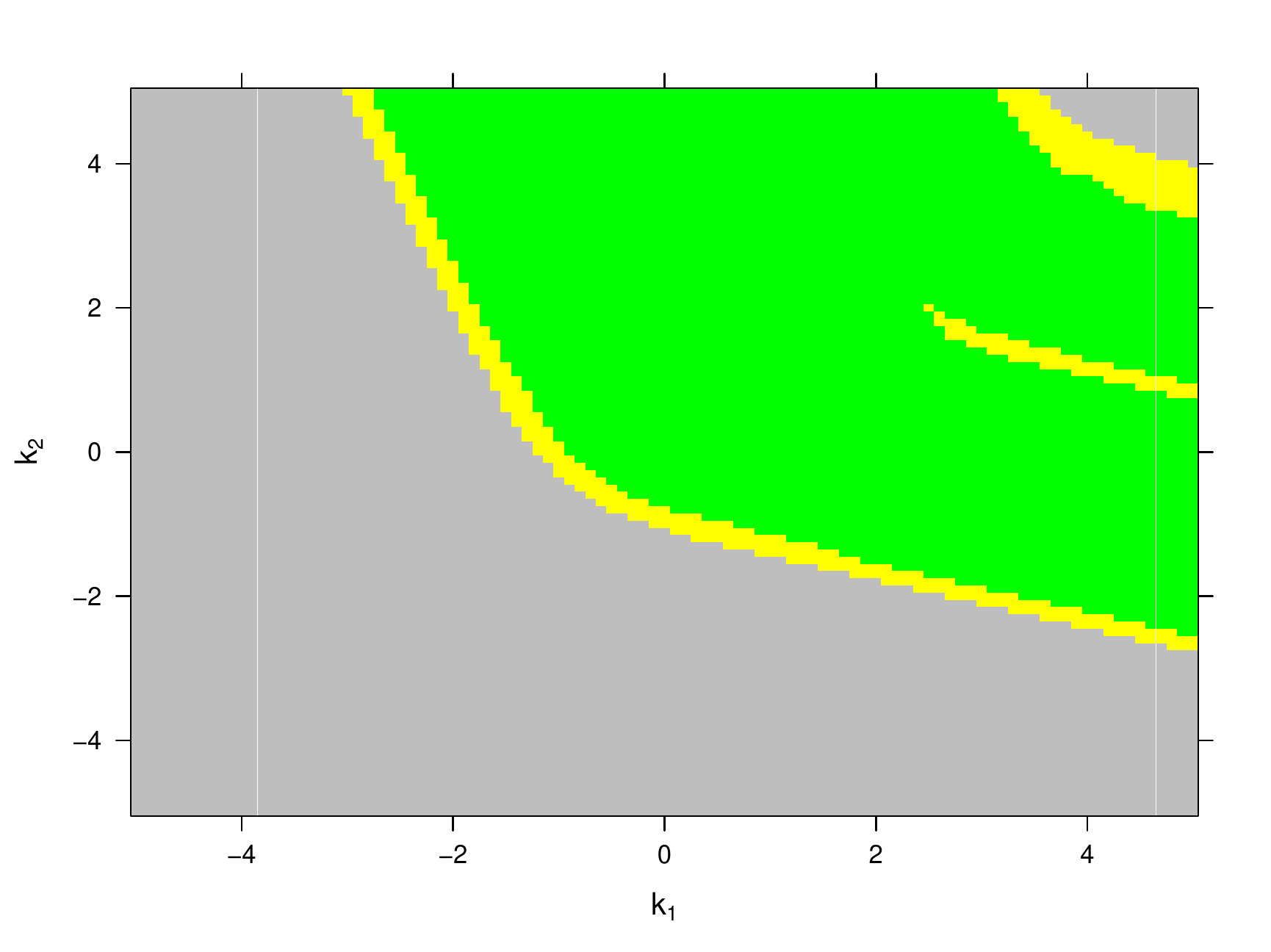} \\
			\multicolumn{2}{c}{\footnotesize $f_{1t}=C_t$, $f_{2t}=B_t$}      
    \end{tabular}
\end{center}
\caption{Left panel: Differences of empirical Murphy diagrams $\hat s_{250,f_1}(k) - \hat s_{250,f_2}(k)$ from \eqref{eq:Murphy emp} versus $k\in\R^2$. 
Right panel: Three-zone traffic light illustration of pointwise comparative backtests following \cite{FisslerETAL2016}. The green area corresponds to the region where the null $H_0^+\colon f_1 \preceq f_2$ is rejected, the red one is where $H_0^-\colon f_1 \succeq f_2$ is rejected, at level 0.05, respectively. Yellow means that neither $H_0^+$ nor $H_0^-$ are rejected. In the grey region, the two Murphy diagrams identically coincide.
}
\label{fig:Murphy}
\end{figure}

In this subsection, we illustrate the use of Murphy diagrams, following Corollary \ref{cor:Murphy}. To allow for graphical illustrations, we reduce the dimension to $d=2$, translating case (a) of subsection \ref{subsec:simulation1} to $d=2$. In particular, we have $Y_t=\mu_t+\epsilon_t$ where $\mu_t$ follows a 2-dimensional normal distribution with mean 0, variances 1 and correlations 0.5, and $\epsilon_t$ follows a 2-dimensional standard normal distribution. As the scalar risk measure $\rho$ we only consider $\VaR_{0.05}$ and we use the aggregation function $\Lambda_1\colon\R^2\to\R$, i.e., $\Lambda(x)=0.25(x_1^++x_2^+)-0.75(x_1^-+x_2^-)$. Besides focused Anne and climatological Bob introduced above both using their respective information sets ideally, we also consider Celia. Just like Anne, Celia has access to $\mu_t$ resulting in the same information set. However, she misinterprets it and issues sign-reversed forecasts $C_t$ assuming that $Y_t\sim\mathcal{N}_2(-\mu_t,I_2)$. That is, $C_t= R(\mathcal N_2(-\mu_t,I_2)) =  R(\mathcal N_2(0,I_2))+\mu_t$. Again, we consider a time horizon of $N=250$.

In the left panel of Figure \ref{fig:Murphy} we illustrate the differences of empirical Murphy diagrams  $[-5,5]^2\ni k\mapsto \hat s_{250,f_1}(k) - \hat s_{250,f_2}(k) = \frac{1}{250}\sum_{t=1}^{250} S_{R,k}(f_{1t},Y_t)-S_{R,k}(f_{2t},Y_t)$ where $f_{1t}$, $f_{2t}$ stand for one of the three considered forecasts, $A_t$, $B_t$ or $C_t$. In each pairwise comparison, we choose $f_{1t}$ to be inferior to $f_{2t}$ such that we expect a non-negative difference of the corresponding Murphy diagrams. Indeed, only in the comparison between Bob and Anne, there are some $k$ where $\hat s_{250,f_1}(k) - \hat s_{250,f_2}(k)<0$. For the remaining regions and situations, the Murphy diagrams behave consistently with our expectations. \\
For all three pairwise comparisons, one can nicely recognise the region where the respective two forecasts differ, resulting in a positive score difference depicted in blue. This bluish region seems to correspond to a blurred version of the boundary of the considered risk measure. Interestingly, while these regions illustrating positive score differences are similar in shape and location for the two pairs involving Celia, this region seems to be slightly translated to the upper right corner in the comparison between Anne and Bob. Quite intuitively, the magnitude of the score difference with a maximum of approximately 0.05 is smaller in the joust between the two ideal forecasts issued by Bob and Anne in comparison to the situations involving the sign-reversed Celia where the maximal difference between the Murphy diagrams is larger than 0.15.
We have performed this experiment several times and observed that the stylised facts are qualitatively stable. For transparency reasons, we have depicted the first experiment performed, but we report some more experiments in \cite{Supplementary}.\\
In the right panel of Figure \ref{fig:Murphy} we depict the results of pointwise comparative backtests using the traffic-light illustration suggested in \cite{FisslerETAL2016}, which is akin to the three-zone approach of the \citet[pp.\ 103--108]{BIS2013}. That is, we perform a Diebold-Mariano test using the elementary score $S_{R,k}$ for each $k$ in a grid of $[-5,5]^2$. This means, we would like to see whether the superiority of the forecasts $f_{1t}$ is recognised at a significance level of $0.05$ deploying the two possible one-sided null hypotheses $H_0^+\colon f_1 \preceq f_2$ and $H_0^-\colon f_1 \succeq f_2$, using the notation introduced in the previous subsection. If for a certain $k$ the null $H_0^+$ is rejected deeming $f_2$ significantly superior to $f_1$, we colour the corresponding $k$ in green. Similarly, if the null $H_0^-$ is rejected, considering $f_1$ to be superior to $f_2$, we illustrate $k$ in red. For all $k$ in the yellow region, none of the two nulls is rejected, meaning that the procedure is indecisive at the significance level 0.05. Finally, the grey area corresponds to those points where the score difference is constantly zero for all $t=1, \ldots, N$. Due to the vanishing variance, a Diebold-Mariano test is apparently not possible there. But clearly, this still means that the two forecasts are just equally good in that region.\\
The specific results nicely correspond to the situations obtained in the left panel of Figure \ref{fig:Murphy}. For all three pairwise comparisons and for $k$ close to the four corners of the area $[-5,5]^2$, the score differences identically vanish, resulting in a grey colouration. 
Again, in all three cases, there is a ``continuous'' behaviour in that the grey region adjoins a yellow stripe before turning into a fairly broad green stripe. For the comparisons involving Celia, clearly using an inferior predictive distribution to both Anne's and Bob's, it is reassuring that a substantial region is coloured in green. In this region, the procedure is decisive, deeming Celia significantly inferior to Anne and to Bob. Moreover, for this particular simulation, there is no red region.
The situation comparing the two ideal forecasters Anne and Bob is somewhat more involved. While most of the previous observations also apply to that situation, there is a small red stripe close to the upper right corner. For $k$ in that region and for this particular simulation, this means that Bob's forecasts outperform Anne's ones. While this observation is somewhat unexpected, it reflects the finite sample nature of the simulation, rendering such outcomes possible. Having a look at some more experiments, the results of which are again reported in \cite{Supplementary}, shows that this red region is not stable over different simulations (which would clearly violate the sensitivity of consistent scoring functions with respect to increasing information sets established in \cite{HolzmannEulert2014}), but it moves and occasionally also vanishes (on the region $[-5,5]^2$ considered). Interestingly, in all events with a red region present, this red region was still roughly located in a similar area.

\section{Discussion}\label{sec:discussion}

As mentioned in the introduction, the aim and main contribution of this paper consists of establishing selective identifiability results in Theorem \ref{thm:identification R_0} and exhaustive elicitability results in Theorem \ref{thm:elicitability} for systemic risk measures sensitive with respect to capital allocations. Notably, the construction of consistent exhaustive scoring functions relies on a mixture representation of easily computable elementary scores, which opens the way to the diagnostic tool of Murphy diagrams.

\textbf{Backtesting. }
A strictly consistent exhaustive scoring function $S_R$ for a systemic risk measure $R$ can be used for comparative backtests of competing exhaustive forecasts, that is, set-valued forecasts, as described in \cite{FisslerETAL2016} and \cite{NoldeZiegel2017}; see also Section \ref{sec:simulations}. 
More precisely, having competing forecasts $A_1, \ldots, A_N \in \P(\R^d;\R^d_+)$, $B_1, \ldots, B_N \in \P(\R^d;\R^d_+)$, and verifying observations of the gains and losses of the financial system $Y_1, \ldots, Y_N\in\R^d$, one can consider a properly normalised version of the test statistic
\[
\frac{1}{N} \sum_{t=1}^N S_R(A_t,Y_t) - S_R(B_t,Y_t)
\]
to assess which forecast sequence is superior under $S_R$. 

On the other hand, the fact that one can construct \emph{oriented } selective identification functions for risk measures might open the way to one-sided traditional backtests \cite[Subsection 2.21]{NoldeZiegel2017}. That is, if one has a sequence of vector-valued predictions for capital requirements $k_1, \ldots, k_N\in\R^d$ along with verifying observations $Y_1, \ldots, Y_N\in\R^d$, one might wonder if the forecasted capital requirements are adequate to eliminate the risk of the financial system under $R$. That means, we would like to judge if $k_t \in R(Y_t)$ for all $t = 1, \ldots, N$ with a certain level of certainty $\a$. If $V_{R_0}$ is an \emph{oriented} strict selective identification function for $R_0$, this amounts to testing the one-sided null hypothesis 
\[
H_0\colon \E[V_{R_0}(k_t,Y_t)]\le 0 \quad \text{for all }t=1, \ldots, N.
\]
Under suitable mixing conditions, one can construct an (asymptotic) level $\a$ test for this null hypothesis by considering a rescaled version of the test statistic
\[
\frac{1}{N} \sum_{t=1}^N V_{R_0}(k_t,Y_t)\,.
\]
Note that from a regulatory perspective, testing this one-sided null hypothesis is more sensible than testing the two-sided null 
\[
H'_0\colon \E[V_{R_0}(k_t,Y_t)]= 0 \quad \text{for all }t=1, \ldots, N.
\]
Indeed, this corresponds to assessing whether $\r(\Lambda(Y_t+k_t))=0$. However, overestimating the financial requirements to make the system acceptable is even more prudent from a regulatory angle.
\vspace{1em}

\textbf{M-Estimation. } 
If a systemic risk measure $R$ is exhaustively elicitable, one can make inference for it in the form of $M$-estimation \citep{HuberRonchetti2009}. That is, if one has a sample $Y_1, \ldots, Y_N\in\R^d$ of stationary observations fulfilling sufficient mixing conditions, one might estimate the set-valued risk measure $R(Y)\in\F(\R^d;\R^d_+)$, where $Y$ has the same distribution as the observations, via
\be{eq:M-estimation}
\widehat R(Y) = \argmin_{A\in\F(\R^d;\R^d_+)} \frac{1}{N} \sum_{t=1}^N S_R(A,Y_t),
\ee
where $S_R\colon \F(\R^d;\R^d_+)\times \R^d\to\R$ is a strictly consistent exhaustive scoring function for $R$. Under suitable conditions, the $M$-estimator $\widehat R(Y)$ at \eqref{eq:M-estimation} is consistent for $R(Y)$. However, computationally, the optimisation problem at \eqref{eq:M-estimation} might be rather expensive, if feasible at all. The reason is that one needs to optimise over the collection of all closed upper sets of $\R^d$. 
\vspace{1em}

\textbf{Regression. }
A closely connected concept to the notion of $M$-estimation is regression where it is possible to bypass the complication to optimise over a collection of sets. 
Consider a time series $(X_t,Y_t)_{t\in\N}$. Sticking to the usual denomination, let $Y_t$ denote the response variable, taking values in $\R^d$, and let $X_t$ be a $p$-dimensional vector of regressors.
The regressors might consist of quantities which seem relevant to the systemic risk of the financial system. Examples include macroeconomic quantities such as GDP, unemployment, inflation, net-investments etc. 
Let $\Theta\subseteq \R^q$ be a parameter space and let $M\colon \R^p \times \Theta \to \F(\R^d;\R^d_+)$ be a parametric model taking values in the collection of closed upper subsets of $\R^d$. Suppose the model is correctly specified in that there exists a unique parameter $\theta_0\in\Theta$ such that 
\be{eq:model specification}
R(F_{Y_t|X_t}) = M(X_t,\theta_0) \quad \mathbb{P}\text{-a.s. for all }t\in\N.
\ee
Here, $R\colon \M^d\to \F(\R^d;\R^d_+)$ is a law-invariant risk measure of the form at \eqref{eq:R} satisfying the conditions of Theorem 
\ref{thm:elicitability}(iii). Further, suppose that the (regular version of the) conditional distribution $F_{Y_t|X_t}$ of $Y_t$ given $X_t$ is an element of $\M_0^d$ almost surely, where we use the notation of Theorem \ref{thm:elicitability}. Note that the time series does not need to be strongly stationary, but only the conditional distribution $F_{Y_t|X_t}$ needs to satisfy the `semi-parametric stationarity condition' specified via \eqref{eq:model specification}. 
Let $S_{R,\pi}$ be a strictly $\M_0^d$-consistent exhaustive scoring function for $R$. Then, under certain mixing and integrability assumptions specified in \citep[Corollary 3.48]{White2001} one yields the following Law of Large Numbers 
\[
\frac{1}{N} \sum_{t=1}^N \big\{ S_{R,\pi}\left(M(X_t,\theta), Y_t\right) - \E\left[S_{R,\pi}\left(M(X_t,\theta), Y_t\right)\right] \big\} \to 0 \quad \mathbb{P}\text{-a.s.} \quad \text{as }N\to
\infty
\]
for all $\theta\in\Theta$. It is essentially a uniform version (in the parameter $\theta$) of this Law of Large Numbers result which yields the consistency for the empirical estimator
\be{eq:theta hat}
\widehat \theta_N:= \argmin_{\theta \in\Theta} \frac{1}{N} \sum_{t=1}^N S_{R,\pi}\left(M(X_t,\theta), Y_t\right)
\ee
for $\theta_0$; see \cite{VanderVaart1998, HuberRonchetti2009, NoldeZiegel2017} for details. The advantage of this regression approach in comparison to $M$-estimation is that the optimisation at \eqref{eq:theta hat} needs to be performed over a subset $\Theta$ of $\R^q$ only (which is often assumed to be compact). This makes the result computationally a lot more feasible than the optimisation procedure over a collection of upper sets. In other words, $M$-estimation can be considered as a special instance of regression where the regressor $X_t$ is constant and where $\Theta$ corresponds to $\F(\R^d;\R^d_+)$.

Besides the usual practical challenge of constructing reasonable parametric models $M$ to model the systemic risk of a financial system $Y_t$ given regressors $X_t$, we see some interesting theoretical problems related to this regression framework. While, under correct model specification given at \eqref{eq:model specification}, any strictly consistent scoring function $S_{R,\pi}$ induces a consistent estimator $\widehat \theta_N$ at \eqref{eq:theta hat}, the estimator will generally depend on the choice of $S_{R,\pi}$ (or $\pi$) in finite samples. Moreover, the efficiency of the estimator $\widehat \theta_N$, expressed in terms of the asymptotic variance of $\sqrt{N}(\widehat \theta_N - \theta_0)$, will depend on the choice of $S_{R,\pi}$, suggesting an interesting optimality criterion for $S_{R,\pi}$.\\ 
A very modern and interesting approach circumventing this issue is to perform regression \emph{simultaneously} with respect to the class of \emph{all} consistent scoring functions (or a reasonably large subclass), which is explored in the recent paper \cite{JordanMuehlemannZiegel2019}. To perform this efficiently, the mixture representation of scoring functions in terms of elementary scores might prove beneficial. We defer this interesting problem to future research.

\section*{Acknowledgement}
We would like to express our sincere gratitude to Tilmann Gneiting and Johanna Ziegel for insightful discussions and persistent encouragement. We are indebted to Timo Dimitriadis and to Peter Baran\v{c}ok who provided helpful comments in the context of equivariant scores, to Yuan Li for his careful proofreading of an earlier version of this paper, and to Luk\'a\v{s} \v{S}ablica for helpful programming advice on the simulation part of this project.\\
Tobias Fissler gratefully acknowledges financial support from Imperial College London via his Chapman Fellowship and the hospitality of the Institute for Statistics and Mathematics at Vienna University of Economics and Business during several research visits when main parts of the project have jointly been developed.



\appendix
\section{Appendix}

\subsection{Proofs of Section \ref{sec:main results}}\label{app:proofs main results}

\begin{proof}[Proof of Theorem \ref{thm:identification R_0}]
\begin{enumerate}[(i)]
\item
Let $V_\r\colon\R\times \R\to\R$ be a strict $\M$-identification function for $\r$. That means for all $Y\in\Y^d$ with distribution $F\in\M^d$, for all $k\in\R^d$ and for all $x\in\R$, one has that
\be{eq:equivalence}
\E_F\big[V_\r(x,\Lambda(Y+k))\big]=0 \quad \Longleftrightarrow \quad x = \r(\Lambda (Y+ k)).
\ee
Setting $x=0$ in \eqref{eq:equivalence} yields
\[
\E_F\big[V_\r(0,\Lambda(Y+k))\big]=0 \quad \Longleftrightarrow \quad 0 = \r(\Lambda (Y+ k))
\quad \Longleftrightarrow \quad k\in R_0(Y),
\]
which holds in particular for $R_0(Y) = \emptyset$.
Therefore $V_{R_0}(k,y) = V_\r(0,\Lambda(y+k))$ is a strict selective $\M$-identification function for $R_0$.
\item
Now assume that $V_\r\colon\R\times \R\to\R$ is an oriented strict $\M$-identification function for $\r$. That means for all $Y\in\Y^d$ with distribution $F\in\M^d$, for all $k\in\R^d$ and for all $x\in\R$, one has that
\be{eq:orientation proof}
\E_F\big[V_\r(x,\Lambda(Y+k))\big]
\begin{cases}
<0, & \text{if }x<\r(\Lambda (Y+k)) \\
=0, & \text{if }x = \r(\Lambda (Y+k)) \\
>0, & \text{if }x>\r(\Lambda (Y+k)).
\end{cases}
\ee
Setting $x=0$ in \eqref{eq:orientation proof} yields the claim.
\end{enumerate}
\end{proof}

\begin{proof}[Proof of Proposition \ref{prop:characterization V}]
The proof follows along the lines of the proof of Theorem 3.2 in \cite{FisslerZiegel2016}; cf.\ \cite{Osband1985}. The dimensionality of $x$ does not play any role in the proof. As our identification functions map to $\R$, we use $k=1$ in the proof of Theorem 3.2 of \cite{FisslerZiegel2016}. The assumption on the existence of $F_1, F_2\in\M^d$ such that the signs of $\bar{V}_{R_0}$ are different plus the convexity of $\M^d$ are equivalent to Assumption (V1) in \cite{FisslerZiegel2016}. If we replace $\nabla\bar{S}(x,F)$ by $\bar{V}'(x,F)$, we obtain that there is a function $h\colon\A\to\R$ such that 
\[
\bar{V}'(x,F)=h(x)\bar{V}(x,F)
\] 
for all $x\in\A$ and all $F\in\M^d$. Since the matrix $\mathbb{B}_G$ in the proof will be a $2\times3$ matrix of rank 1 for any $x\in\A$, $h(x)$ has to be nonzero for all $x\in\A$.
\end{proof}

\begin{proof}[Proof of Proposition \ref{prop:EAR identifiability}]
Let $F\in\M^d$ and $EAR_w(F)\neq\emptyset$.
Note that for any $k\in\R^d$, $\bar V_{{EAR}_w}(k,F)$ evaluates $\bar{V}_{R_0}(\,\cdot\,,F)\colon \R^d\to\R$ on the hyperplane orthogonal to $w$ containing $k$. 
Since $EAR_w(F) \subseteq R_0(F)$, $k\in EAR_w(F)$ implies that $\bar V_{R_0}(k,F)=0$. The orientation of $V_{R_0}$, and the facts that $EAR_w(F)$ is the intersection of $R(F)$ with the supporting hyperplane for $R(F)$ orthogonal to $w$ and that $R(F)$ is an upper set imply that $\bar{V}_{R_0}(k+x,F) = \bar{V}_{EAR_w}(k,F)(x)\leq0$ for all $x\in w^\bot$. If $k\notin EAR_w(F)$, there are two possibilities: 
\begin{enumerate}[(i)]
\item 
The orthogonal hyperplane containing $k$ has an empty intersection with $R(F)$ such that $\bar{V}_{R_0}(k+x,F)<0$ for all $x\in w^\bot$.
\item 
The orthogonal hyperplane containing $k$ has a non-empty intersection with $R(F)\setminus R_0(Y)$ such that there is some $x\in w^\bot$ with $\bar{V}_{R_0}(k+x,F)>0$.
\end{enumerate}
Now let $EAR_w(F)=\emptyset$. If $R(F)=\emptyset$, \eqref{eq:identification EAR} holds trivially. If $R(F)$ is non-empty, $EAR_w(F)$ is only empty if there is no supporting hyperplane for $R(F)$ orthogonal to $w$. Then for any $k\in\R^d$, there are $ x_1, x_2\in w^\bot$ such that $\bar{V}_{EAR_w}(k,F)(x_1)>0$ and $\bar{V}_{EAR_w}(k,F)(x_2)<0$, as depicted in the right part of Figure \ref{fig:proofEAR}.
\end{proof}

\begin{figure}
\centering
\includegraphics[width = 0.9\textwidth]{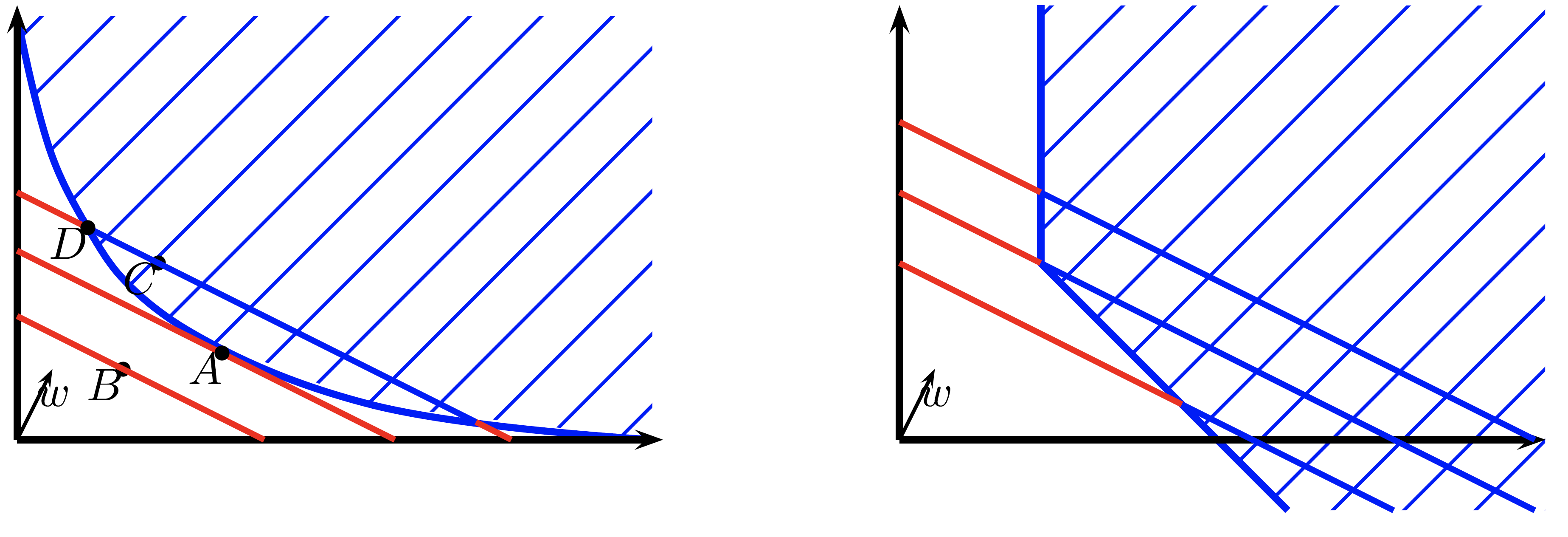}
\caption{A graphical illustration of the proof of Proposition \ref{prop:EAR identifiability} for dimension $d=2$. Suppose the blue region corresponds to the correctly specified risk measure $R(F)$. In the left picture, $EAR_w(F)$ is a singleton, containing only point $A$. Point $B$ corresponds to case (i), whereas points $C$ and $D$ are examples of case (ii) for points that are not in $EAR_w(F)$. In the right picture, $EAR_w(F)=\emptyset$. 
For any $k\in\R^d$ there is some $x\in w^\bot$ such that $\bar V_{R_0}(k+x,F)>0$.}
\label{fig:proofEAR}
\end{figure}

\begin{proof}[Proof of Proposition \ref{prop:R0 to rho}]
The `only if' part is a special case of Theorem \ref{thm:identification R_0}. 
For the `if' part, assume $V_{R_0}\colon\R^d\times\R^d\to\R$ is a strict selective $\M^d$-identification function for $R_0$. For any $Y\in\Y^d$ it holds that $\E[V_{R_0}(0,Y)]=0\Leftrightarrow 0\in R_0(Y)\Leftrightarrow\r (\Lambda(Y))=0$. Then we obtain that for any $s\in\R$ and any $X\in\Y$
\[
\r(X)=s\ \Longleftrightarrow \ \r(X+s)=0 \ \Longleftrightarrow \ \E[V_{R_0}(0,\eta(X+s))]=0.
\] 
Thus $\r$ is identifiable with a strict selective $\M$-identification function $V_{\r}\colon\R\times\R\to\R$, $V_{\r}(s,x)=V_{R_0}(0,\eta(x+s))$.
\end{proof}

For the proof of Theorem \ref{thm:elicitability}, we need the following lemma.

\begin{lem}\label{lem:non-empty interior}
Let $A_1,A_2 \in \F(\R^d;\R^d_+)$. Then, the symmetric difference $A_1\triangle A_2 = (A_1\setminus A_2) \cup (A_2\setminus A_1)$ is empty if and only if its interior, $\interior(A_1\triangle A_2)$, is empty.
\end{lem}

\begin{proof}[Proof of Lemma \ref{lem:non-empty interior}]
If $A_1\triangle A_2 = \emptyset$ it is clear that $\interior(A_1\triangle A_2 )= \emptyset$.

Assume that there is an $x\in A_1\triangle A_2$. Without loss of generality, we can assume that $x\in A_1\setminus A_2$. If $x\in\interior(A_1\setminus A_2)$, we are done. Hence, let $x\in(A_1\setminus A_2)\setminus \interior(A_1\setminus A_2)$ which implies that $x\in \partial (A_1\setminus A_2)$, where $\partial (A_1\setminus A_2)$ denotes the boundary of $A_1\setminus A_2$. It holds that $\partial (A_1\setminus A_2) = \partial (A_1\cap A_2^c) \subseteq \partial A_1 \cup \partial (A_2^c) = \partial A_1 \cup \partial A_2$. But since $\partial A_2\subseteq A_2$ and $x\in A_1\setminus A_2$, it follows that $x\in\partial A_1\setminus A_2$.\\
Due to the definition of the boundary, this means that for all $\eps>0$ it holds that $B_\eps(x)\cap A_1\neq \emptyset$, where $B_\eps(x)$ is the open ball with centre $x$ and radius $\eps$. Assume that for all $\eps>0$ we have $B_\eps(x)\cap A_2\neq \emptyset$, then $x\in\bar A_2 = A_2$, which is a contradiction to the assumption that $x\in \partial A_1\setminus A_2$.
That means there exists an $\eps_0>0$ such that $B_{\eps_0}(x)\cap A_2= \emptyset$. 
Moreover, since $A_1$ is an upper set, $x+\R^d_{++}$ is a non-empty open subset of $A_1$. Furthermore, we see that $B_{\eps_0}(x)\cap (x+\R^d_{++})$ is a non-empty open subset of $A_1$ which is disjoint from $A_2$. This means that $\interior(A_1\setminus A_2)\neq \emptyset$.
\end{proof}

\begin{proof}[Proof of Theorem \ref{thm:elicitability}]
\begin{enumerate}[(i)]
\item
Let $k\in\R^d$, $A\in\widehat \P(\R^d;\R^d_+)$ and $F\in\M^d$. 
A direct calculation yields that 
\begin{align}\label{eq:consistency}
\bar S_{R,k}(A,F) - \bar S_{R,k}(R(F),F) 
&=  \big(\one_{R(F) \setminus A}(k) - \one_{A\setminus R(F)}(k) \big)\bar V_{R_0}(k,F)
\ge0,
\end{align}
where the last inequality is a direct consequence of the weak form of orientation given at \eqref{eq:weak orientation}. The non-negativity of $S_{R,k}$ follows from the $\M^d$-consistency, exploiting that $\delta_y\in\M^d$ for all $y\in\R^d$ and $S_{R,k}(A,y) \ge S_{R,k}(R(y),y) = 0$.
\item
This is a direct consequence of the non-negativity and consistency of the scores $S_{R,k}$.
\item
Let $F\in\M^d$, and $A^* := R(F), A \in\F(\R^d;\R^d_+)$ with $A\neq A^*$. Assume that the inequality $\bar S_{R,\pi}(A,F), \bar S_{R,\pi}(A^*,F)<\infty$ holds (otherwise, there is nothing to show). Using Fubini's Theorem, we obtain
\begin{align}\nonumber
\bar S_{R,\pi}(A,F) - \bar S_{R,\pi}(A^*,F) 
&= \int_{A^*\setminus A} \bar V_{R_0}(k,F) \,\pi(\diff k) - \int_{A\setminus A^*} \bar V_{R_0}(k,F) \,\pi(\diff k).
\end{align}

Then Lemma \ref{lem:non-empty interior} yields that $\interior(A\setminus A^*)\neq \emptyset$ or $\interior(A^*\setminus A)\neq \emptyset$. If $\interior(A\setminus A^*)\neq \emptyset$, the fact that $\bar V(\cdot, F)$ is strictly negative on $(A^*)^c$ and the assumption that $\pi$ assigns positive mass to any non-empty open set in $\mathcal B(\R^d)$ implies that 
\[
\int_{A\setminus A^*} \bar V_{R_0}(k,F) \,\pi(\diff k)<0,
\]
which implies that $\bar S_{R,\pi}(A,F) - \bar S_{R,\pi}(A^*,F) >0$. \\
Assume $\interior(A^*\setminus A)\neq \emptyset$. The boundary $\partial A^* = R_0(F) = \{k\in\R^d\,|\,\bar V_{R_0}(k,F)=0\}$ is a closed set. That means that $\interior(A^*\setminus A)\setminus \partial A^*$ is open and non-empty. Moreover $\bar V_{R_0}(\cdot,F)$ is strictly positive on $\interior(A^*\setminus A)\setminus \partial A^*$. Hence, 
\[
\int_{A^*\setminus A} \bar V_{R_0}(k,F) \,\pi(\diff k)\ge\int_{\interior(A^*\setminus A)\setminus \partial A^*} \bar V_{R_0}(k,F) \,\pi(\diff k) >0,
\]
which implies that $\bar S_{R,\pi}(A,F) - \bar S_{R,\pi}(A^*,F) >0$. 
\end{enumerate}
\end{proof}


\begin{proof}[Proof of Proposition \ref{prop:order-sensitivity}]
For the first part of the Proposition, it is sufficient to show order-sensitivity for the elementary scores $S_{R,k}$ given at \eqref{eq:S_{R,k}}. 
Let $A\subseteq B\subseteq R(F)$. Then, for any $k\in\R^d$,
\[
\bar S_{R,k}(A,F) - \bar S_{R,k}(B,F) = \one_{B\setminus A} (k) \bar V_{R_0}(k,F) \ge0,
\]
due to the orientation given at \eqref{eq:weak orientation}. On the other hand, for $A\supseteq B\supseteq R(F)$ we obtain that 
\[
\bar S_{R,k}(A,F) - \bar S_{R,k}(B,F) = -\one_{A\setminus B} (k) \bar V_{R_0}(k,F) \ge0,
\]
where the inequality follows again by \eqref{eq:weak orientation}.

The second part of the proposition follows along the lines of the proof of Theorem \ref{thm:elicitability}(iii).
\end{proof}

\subsection{Proofs of Section \ref{sec:secondary}}\label{app:proofs secondary}

\begin{proof}[Proof of Lemma \ref{lem:R pos hom}]
Assume that $\r$ is a positively homogeneous scalar risk measure and $\Lambda$ is positively homogeneous of degree $b\in\R$. Let $c>0$ and $Y\in\Y^d$.
\begin{align*}
R(cY)=&\left\{k\in\R^d\,|\,\r(\Lambda(cY+k))\leq0\right\}=\left\{k\in\R^d\,|\,\r\left(c^b\Lambda\left(Y+k/c\right)\right)\leq0\right\}\\
=&\left\{k\in\R^d\,|\,c^b\r\left(\Lambda\left(X+k/c\right)\right)\leq0\right\}=\left\{k\in\R^d\,|\,\r\left(\Lambda\left(X+k/c\right)\right)\leq0\right\}\\
=& c\left\{k\in\R^d|\r(\Lambda(X+k))\leq0\right\}=cR(Y).
\end{align*}
\end{proof}

\begin{proof}[Proof of Lemma \ref{lem:ident}]
\begin{enumerate}[(i)]
\item 
Let $y,k,l\in\R^d$. Then 
\[
V_{R_0}(k+l,y-l)=V_\r(0,\Lambda(k+l+y-l))=V_\r(0,\Lambda(k+y))=V_{R_0}(k,y).
\]
\item
Let $V'_{R_0}$ be another translation invariant strict $\M^d$-identification function for $R_0$. Using Proposition \ref{prop:characterization V} there is a non-vanishing function $h\colon\R^d\to\R$ such that 
$\bar V'_{R_0}(x,F) = h(x) \bar V_{R_0}(x,F)$
for all $x\in\R^d$ and for all $F\in\M^d$. 
One can show that $h$ is constant along the lines of the proof of Proposition 4.7(ii) in \cite{FisslerZiegel2019}.
\item 
Let $c>0$, $k,y\in\R^d$. Then
\begin{align*}
V_{R_0}(ck, cy)&=V_\r(0,\Lambda(ck+cy))=V_\r(0,c^b\Lambda(k+y))\\
&=c^{ab}V_\r(0,\Lambda(k+y))=c^{ab}V_{R_0}(k,y).
\end{align*}
\end{enumerate}
\end{proof}

\begin{proof}[Proof of Proposition \ref{prop:translation invariance}]
First observe that $V_{R_0}(k,y) = V_\r(0,\Lambda(y+k))$ is an oriented selective translation invariant strict $\M^d$-identification function for $R_0$, invoking Lemma \ref{lem:ident}(i) and Theorem \ref{thm:identification R_0}. Then a direct computation yields that 
\be{eq:translation inv proof}
S_{R,k}(A+l,y - l) = S_{R,k-l}(A,l)
\ee
for all $A\in 2^{\R^d}$, $y, l\in\R^d$. Therefore, \eqref{eq:translation inv proof} and the translation invariance of the Lebesgue measure show part (i).\\
For part (ii), assume that $S$ is of the form at \eqref{eq:S_R,pi} and is translation invariant where, invoking the discussion above, we may assume without loss of generality that the elementary scores $S$ are based on the translation invariant identification function $V_{R_0}$. For $l\in\R^d$ define the measure $\pi_l(A) = \pi(A+l)$ for $A\in\mathcal B(\R^d)$. Note that since $S$ is assumed to be finite this implies that the score differences are well-defined and are translation invariant as well. For $A\subseteq B$ this implies that for any $l\in\R$
\[
\int_{B\setminus A}V_{R_0}(z,y)\pi_l(\diff z)=\int_{B\setminus A}V_{R_0}(z,y)\pi(\diff z).
\]
Any set of the form $I=[a_1,b_1)\times\cdots\times[a_d,b_d)$, $a_i,b_i\in\R$, $a_i\leq b_i$, can be represented as $B\setminus A$ for some $A,B\in\F(\R^d;\R^d_+)$ with $A\subseteq B$. 
The system of these sets $I$, however, is a generator of $\mathcal B(\R^d)$, and we conclude that for any $l\in\R^d$
\be{eq:identity proof}
\nu_{y,l}(D):=\int_{D}V_{R_0}(z,y)\pi_l(\diff z)=\int_{D}V_{R_0}(z,y)\pi(\diff z)=:\nu_{y}(D).
\ee
for all $D\in \mathcal B(\R^d)$. For each $D\in \mathcal B(\R^d)$ we obtain the decomposition (depending on $y$) $D = D_y^+ \cup D_y^- \cup D_y^0$, where $
D_y^- = D\cap R(y)^c$, $D_y^0 = D\cap R_0(y)$ and $D_y^+ = D\cap R(y) \setminus R_0(y)$. Hence, \eqref{eq:identity proof} and the strictness of the identification function $V_{R_0}$ imply that for $E = D_y^+$ or $E = D_y^-$
\be{eq:identity proof 2}
\pi_l(E) = \int_{E}\frac{1}{V_{R_0}(z,y)}\nu_{y,l}(\diff z)=\int_{E}\frac{1}{V_{R_0}(z,y)}\nu_{y}(\diff z)=\pi(E).
\ee
The translation equivariance of $R$ and Assumption \eqref{ass:closedness} imply that \eqref{eq:identity proof 2} holds for all $E\in\mathcal B(\R^d)$. That means that $\pi = \gamma\L^d$ for some $\gamma\ge0$.
\end{proof}

\begin{proof}[Proof of Proposition \ref{prop: positive homogeneity}]
Assume that $V_{R_0}$ is an oriented strict selective $\M^d$-identification function for $R_0$ which is positively homogeneous of degree $a\in\R$. A direct computation yields that 
\be{homogeneity proof}
S_{R,k}(cA,cy)=c^aS_{R,\frac{k}{c}}(A,y)
\ee
for all $A\in2^{\R^d}$, $y\in\R^d$ and $c>0$. If $\pi$ is positively homogeneous of degree $b\in\R$, \eqref{homogeneity proof} implies that $S_{R,\pi}$ in \eqref{eq:homogeneous} is positively homogeneous of degree $a+b$.
\\
For part (ii), assume that $S$ is of the form at \eqref{eq:S_R,pi} and is positively homogeneous of degree $a+b$. For $c\in\R$ define the measure $\pi_c(A)=\pi(cA)$ for $A\in\mathcal B(\R^d)$. Note that since $S$ is assumed to be finite, the positive homogeneity of $S$ implies that the score differences are well-defined and are also positively homogeneous of degree $a+b$. For $A\subseteq B$ a direct computation shows that for any $c>0$
\[
\int_{B\setminus A}V_{R_0}(z,y)\pi_c(\diff z)=c^b\int_{B\setminus A}V_{R_0}(z,y)\pi(\diff z).
\]
With the same arguments as in the proof of Proposition \ref{prop:translation invariance}, we conclude that 
\[
\nu_{y,c}(D):=\int_DV_{R_0}(z,y)\pi_c(\diff z)=c^b\int_D V_{R_0}(z,y)\pi(\diff z) =:c^b\nu_y(D)
\]
for all $D\in\mathcal B(\R^d)$ and
\be{homogeneity proof 4}
\pi_c(E)=\int_{E}\frac{1}{V_{R_0}(z,y)}\nu_{y,c}(\diff z)=\int_E\frac{1}{V_{R_0}(z,y)}c^b\nu_y(\diff z)=c^b\pi(E)
\ee for $E=D\cap R(y)^c$ or $E=D\cap R(y)\setminus R_0(y)$. Finally, the translation equivariance of $R$ and Assumption \eqref{ass:closedness} imply that \eqref{homogeneity proof 4} holds for all $E\in\mathcal B(\R^d)$. That means $\pi$ is positively homogeneous of degree $b$.
\end{proof}

\subsection{Proofs of Section \ref{sec:ES}}\label{app:proofs ES}

\begin{proof}[Proof of Proposition \ref{prop:ES ident}]
Let $F\in\M^d$ and $k\in\R^d$. Then
\begin{align*} 
\bar U_2(T^{\VaR_\a}(F),k,F)
&= \frac{1}{\a}\E_F[\Lambda(Y+k)\,\one\{\Lambda(Y+k)\le - T^{\VaR_\a}(F)(k)\}] \\ 
& + \frac{1}{\a} \VaR_\a(\Lambda(Y+k))\big(F_{\Lambda(Y+k)}(-\VaR_\a(\Lambda(Y+k))) - \a\big)\\
&= -\ES_\a(\Lambda(Y+k)) \\ 
&\begin{cases}
<0, & \text{if } k\notin R^{\ES_\a}(F)\\
=0, & \text{if } k\in R^{\ES_\a}_0(F)\\
>0, & \text{if } k\in R^{\ES_\a}(F)\setminus R^{\ES_\a}_0(F),
\end{cases}
\end{align*}
where $F_{\Lambda(Y+k)}$ is the distribution function of $\Lambda (Y+k)$. 
Under Assumption \eqref{ass:regularity} it holds that  $\bar U_1(T^{\VaR_\a}(F),k,F)= 0$. Therefore, one ends up with the second assertion.
\end{proof}

\begin{proof}[Proof of Theorem \ref{thm:ES elicitability}]
\begin{enumerate}[(i)]
\item
Let $F\in\M^d$, $v\in\R^{\R^d}$, $A\in\widehat \P(\R^d;\R^d_+)$ and $v^* = T^{\VaR_\a}(F)$, $A^* = R^{\ES_\a}(F)$ and $k\in\R^d$. If $\bar S_k(v,A,F) = \infty$ there is nothing to show. So we assume that $\bar S_k(v,A,F)$ is finite. Consider
\begin{multline*}
\bar S_k(v,A,F) - \bar S_k(v^*,A,F)
= \one_A(k)\big[-\bar U_2(v,k,F) + U_2(v^*,k,F)\big] \\
= \one_A(k)\E_F\big[ S_{\a,\mathrm{id}}(-v(k), \Lambda(Y+k)) - S_{\a,\mathrm{id}}(-v^*(k), \Lambda(Y+k))  \big]
\ge0,
\end{multline*}
since $S_{\a,\mathrm{id}}$ is consistent for the $\a$-quantile.
If $\bar S_k(v^*,A,F) = \infty$ we are done. Otherwise, consider
\[
\bar S_k(v^*,A,F)  - \bar S_k(v^*,A^*,F) 
= \big( \one_{A^* \setminus A}(k) - \one_{A\setminus A^*}(k)\big)\bar U_2(v,k,F)\ge0,
\]
where the inequality follows from \eqref{eq:U orientation}. The non-negativity follows from the consistency and the fact that $S_k(T^{\VaR_\a}(\delta_y), R^{\ES_\a}(\delta_y),y) = 0$.
\item
Due to part (i), the score $S_{0,\pi_2}$ is $\M^d$-consistent for $(T^{\VaR_\a}, R^{\ES_\a})\colon\M^d\to \R^{\R^d} \times \widehat \P(\R^d;\R^d_+)$.
Since $S_{\a,g_z}$ is a consistent selective scoring function for the $\a$-quantile, the assertion follows invoking Fubini's Theorem.
\item
Let $F\in\M_0^d$, $v\in \mathcal C(\R^d;\R)$, $A\in\F(\R^d;\R^d_+)$ and $v^* = T^{\VaR_\a}(F)$, $A^* = R^{\ES_\a}(F)$. If $v\neq v^*$ then $K = \{k\in\R^d\,|\,v(k)\neq v^*(k)\} \neq \emptyset$ is open. If $\bar S_{\pi_1,\pi_2}(v,A,F)=\infty$ there is nothing to show. Otherwise
\begin{align} \nonumber 
&\E_F[S_{\pi_1,\pi_2}(v,A,Y) - S_{\pi_1,\pi_2}(v^*,A,Y)] \\ \nonumber
&\ge \int_{K} \E_F\big[S_{\a,g_k}(-v(k),\Lambda(Y+k)) - S_{\a,g_k}(-v^*(k),\Lambda(Y+k))\big] \,\pi_1(\diff k)\\ \nonumber
&\quad+\frac{1}{\a}\int_{A\cap K} \E_F\big[S_{\a,\mathrm{id}}(-v(k),\Lambda(Y+k)) - S_{\a,\mathrm{id}}(-v^*(k),\Lambda(Y+k))\big] \,\pi_2(\diff k) 
>0,
\end{align}
where the first integral is strictly positive and the second one is non-negative (and strictly positive if and only if $\pi_2(A\cap K)>0$).\\
If $A\neq A^*$, then $\E_F[S_{\pi_1,\pi_2}(v^*,A,Y) -S_{\pi_1,\pi_2}(v^*,A^*,Y)]>0$, which follows with similar arguments as in the proof of Theorem \ref{thm:elicitability}(iii).
\end{enumerate}
\end{proof}

\bibliographystyle{apacite}

\end{document}


\title{Supplementary Material to:\\
Elicitability and Identifiability of Systemic Risk Measures}
\author{Tobias Fissler\thanks{Vienna University of Economics and Business, Institute for Statistics and Mathematics, Welthandels\-platz 1, 1020 Vienna, Austria, 
\texttt{tobias.fissler@wu.ac.at},
\texttt{jana.hlavinova@wu.ac.at} and \texttt{birgit.rudloff@wu.ac.at}}
\and Jana Hlavinov\'{a}\footnotemark[1]
\and Birgit Rudloff\footnotemark[1]
}
\maketitle

\begin{abstract}
\textbf{Abstract.}
This note contains supplementary material to the paper \cite{FisslerHlavinovaRudloff_RM}. In particular, it gathers results on systemic risk measures insensitive to capital allocations and gives technical results related to the \emph{revelation principle}.
Moreover, additional simulation results are reported, complementing the ones in the previous paper.
\end{abstract}
\noindent
\textit{Keywords:}
Consistency; 
Elicitability; 
Identifiability; 
Forecast evaluation; 
Murphy diagrams; 

\noindent
\textit{MSC 2010 Subject Classification: } 62F07; 62F10; 91G70 

\section{Systemic risk measures insensitive to capital allocations}\label{app:Rins}

Throughout this note, we use the same notation introduced in \cite{FisslerHlavinovaRudloff_RM}.
We give some elicitability and identifiability results for systemic risk measures which are insensitive with respect to capital allocation. They are similar to some results in Section 3 of \cite{FisslerHlavinovaRudloff_RM}, however often less technically involved. To this end, we first introduce $R^{ins}$ and other law-invariant risk measures derived from $R^{ins}$:
\begin{align}
\label{eq:Rins}
&R^{ins}\colon \Y^d \to 2^{\R^d}, && Y \mapsto  R^{ins}_0(Y) = \{k\in\R^d\,|\,\r(\Lambda (Y)+\bar k)\le0\},\\
\label{eq:Rins0}
&R^{ins}_0\colon \Y^d \to 2^{\R^d}, && Y \mapsto  R^{ins}_0(Y) = \{k\in\R^d\,|\,\r(\Lambda (Y)+\bar k)=0\},\\
\label{eq:r}
&r\colon \Y^d\to \R, && Y \mapsto  r(Y) = \r(\Lambda(Y)),
\end{align}
where we recall the shorthand $\bar k = \sum_{i=1}^d k_i$ for some vector $k = (k_1, \ldots, k_d)\in\R^d$.
Due to the cash-invariance of $\r$, $R_0^{ins}$ is a bijection of $r$. Note that $R^{ins}$ is always a closed half-space above the hyperplane with normal $(1,\ldots, 1)^\top\in\R^d$ and $R^{ins}_0$ corresponds to its topological boundary. This implies a one-to-one relationship between $R_0^{ins}$ and $R^{ins}$. Thanks to these facts, one can make use of an extension of the so called \emph{revelation principle} which originates from \citeauthor{Osband1985}'s (\citeyear{Osband1985}) seminal thesis.

\begin{prop}[Revelation principle]\label{prop:revelation}
Let $\T\colon\M^d\to\R$ be an identifiable and elicitable functional, $g\colon\R^d\to\R$ some map, and $h\colon \A\to\R$, $\A\subseteq2^{\R^d}$ a bijection with inverse $h^{-1}$. Then the following assertions hold true:
\begin{enumerate}[\rm (i)]
\item
$\T\colon\M^d\to\R$ is identifiable if and only if $\T_{h^{-1}}=h^{-1}\circ \T\colon\M^d\to\A$ is exhaustively identifiable. The function $V\colon\R\times\R^d\to\R$ is a strict $\M^d$-identification function for $T$ if and only if 
\begin{align*}
V_{h^{-1}}\colon\A\times\R^d\to\R, \qquad (B,y)\mapsto V(h^{-1}(B),y)
\end{align*}
is a strict $\M^d$-identification function for $\T_{h^{-1}}=h^{-1}\circ \T\colon\M^d\to\A$.
\item
$\T\colon\M^d\to\R$ is elicitable if and only if $\T_{h^{-1}}=h^{-1}\circ \T\colon\M^d\to\A$ is exhaustively elicitable. The function $S\colon\R\times\R^d\to\R$ is a strictly $\M^d$-consistent scoring function for $T$ if and only if 
\begin{align*}
S_{h^{-1}}\colon\A\times\R^d\to\R, \qquad (B,y)\mapsto S(h^{-1}(B),y)
\end{align*}
is a strictly $\M^d$-consistent exhaustive scoring function for $\T_{h^{-1}}=h^{-1}\circ \T\colon\M^d\to\A$.
\item
If $S\colon\R\times\R^d\to\R$ is a (strictly) consistent scoring function for $\T$, then 
\[
S_{g^{-1}}\colon\R^d\times\R^d\to\R,\qquad (x,y)\mapsto S(g(x),y)
\]
is a (strictly) $\M^d$-consistent selective scoring function for $\T_{g^{-1}}=g^{-1}\circ \T\colon\M^d\to2^{\R^d}$.
\item
If $V\colon\R\times\R^d\to\R$ is a (strict) $\M^d$-identification function for $\T$, then
\[
V_{g^{-1}}\colon\R^d\times\R^d\to\R,\qquad (x,y)\mapsto V(g(x),y)
\]
is a (strict) selective $\M^d$-identification function for $\T_{g^{-1}}$.
\item
If $V\colon\R\times\R^d\to\R$ is an oriented strict identification function for $\T$ and $g$ is strictly increasing with respect to the componentwise order, then $V_{g^{-1}}\colon\R^d\times\R^d\to\R$ is an oriented strict selective $\M^d$-identification function for $\T_{g^{-1}}$ in the following sense:
\[
\bar{V}_{g^{-1}}(x,Y) \begin{cases}
<0, & \text{if } x\in (g^{-1}(\{\T(F)\})-\R^d_+)\setminus g^{-1}(\{\T(F)\}) \\
=0, & \text{if } x\in g^{-1}(\{\T(F)\}) \\
>0, & \text{if } x\in  (g^{-1}(\{\T(F)\})+\R^d_+)\setminus g^{-1}(\{\T(F)\}).
\end{cases}
\]
\end{enumerate}
\end{prop}

\begin{proof}[Proof.\ \ ]
\begin{enumerate}
\item[(i)--(ii)]
These statements are a special case of Lemma 2.3.2 in \cite{Fissler2017}.
\item[(iii)]
This is a special case of Lemma 2.3.3 in \cite{Fissler2017}.
\item[(iv)]
Let $F\in\M^d$. If $\T_{g^{-1}}(F)=g^{-1}(\{\T(F)\})=\emptyset$, there is nothing to show. Assume that $x\in\T_{g^{-1}}(F)$. Then we have $g(x)=\T(F)$ and thus $\bar{V}_{g^{-1}}(x,F)=\bar{V}(g(x),F)=0$. Moreover, if $x'\notin \T_{g^{-1}}(F)$, we have $g(x')\neq \T(F)$ and thus if $V$ is a strict identification function for $\T$, we have $\bar{V}_{g^{-1}}(x',F)=\bar{V}(g(x'),F)\neq0$.
\item[(v)]
From the previous part we already have $\bar{V}_{g^{-1}}(x,F)=0$ for $x\in g^{-1}(\{\T(F)\})$. Now assume $x\in (g^{-1}(\{\T(F)\})-\R^d_+)\setminus g^{-1}(\{\T(F)\})$. Then there is some $x'\in g^{-1}(\{\T(F)\})$ such that $x\leq x'$ componentwise and $x\neq x'$ and thus $g(x)<g(x')=\T(F)$. Therefore, using the orientation of $V$, we get $\bar{V}_{g^{-1}}(x,F)<0$. The last part follows by similar considerations.
\end{enumerate}
\end{proof}

\begin{lem}\label{lem:identification r}
Let $\r\colon\M\to\R$ be identifiable and elicitable with a strict $\M$-identification function $V_\r\colon\R\times\R\to\R$ and a strictly $\M$-consistent scoring function $S_\r\colon\R\times\R\to\R$. Then the following assertions hold for $r\colon\M^d\to\R$ defined at \eqref{eq:r}:
\begin{enumerate}[\rm (i)]
\item $r$ is identifiable and
\be{eq:def V_r}
V_r\colon\R\times\R^d\to\R, \qquad (x,y)\mapsto V_{\r}(x,\Lambda(y))
\ee
is a strict $\M^d$-identification function for $r$.
\item If $V_\r\colon\R\times \R\to\R$ is an oriented strict $\M$-identification function for $\r$, then $V_r$ defined at \eqref{eq:def V_r} is oriented for $r$.
\item $r$ is elicitable and
\[
S_r\colon\R\times\R^d\to\R,\qquad (x,y) \mapsto S_r(k,y)=S_\r(x,\Lambda(y))
\]
is a strictly $\M^d$-consistent scoring function for $r$.
\end{enumerate}
\end{lem}

\begin{proof}
Obvious.
\end{proof}

\begin{cor}
Let $\rho\colon\M\to\R$ be identifiable and elicitable with a strict $\M$-identification function $V_\r\colon\R\times\R\to\R$ and a strictly $\M$-consistent scoring function $S_\r\colon\R\times\R\to\R$. Then the following assertions hold for $R^{ins}$ defined at \eqref{eq:Rins} and $R^{ins}_0$ defined at \eqref{eq:Rins0}:
\begin{enumerate}[\rm (i)]
\item
$R^{ins}$ is exhaustively identifiable and exhaustively elicitable. Define $\bar{k}_{min}(B)=\min\left\{\bar{k}\,|\,k\in B\right\}$ for $B\in\A$ where $\A$ is the collection of all closed half-spaces above the hyperplane with normal $(1,\ldots, 1)^\top\in\R^d$. Then
\begin{align*}
V_{R^{ins}}\colon\A\times\R^d\to\R,\qquad (B,y)\mapsto V_\r(\bar{k}_{min}(B),y),\\
S_{R^{ins}}\colon\A\times\R^d\to\R,\qquad (B,y)\mapsto S_\r(\bar{k}_{min}(B),y)
\end{align*}
are a strict exhaustive $\M^d$-identification function and a strictly $\M^d$-consistent exhaustive scoring function for $R^{ins}$, respectively.
\item
$R^{ins}_0$ is selectively identifiable and
\[
V_{R_0^{ins}}\colon \R^d\times \R^d\to\R, \qquad (k,y) \mapsto V_{R_0^{ins}}(k,y) = V_\r(\bar k,\Lambda(y))
\]
is a strict selective $\M^d$-identification function for $R_0^{ins}$. Moreover, if $V_\rho$ is oriented meaning that $\bar V_{\rho}(x, F)\ge0$ iff $x\ge \rho(F)$ for all $F\in\M$ and for all $x\in\R$, then $V_{R_0^{ins}}$ is oriented in the sense that for all $k\in\R^d$ and for all $F\in\M^d$
\[
\bar{V}_{R_0}^{ins}(k,F) \begin{cases}
<0, & \text{if } k\notin R^{ins} (F) \\
=0, & \text{if } k\in R_0^{ins}(F) \\
>0, & \text{if } k\in  R^{ins}(F) \setminus R_0^{ins}(F).
\end{cases}
\]
\item
$R^{ins}_0$ is selectively elicitable and
\[
S_{R_0^{ins}}\colon\R^d\times\R^d\to\R,\qquad (k,y) \mapsto S_{R_0^{ins}}(k,y)=S_\r(\bar k,\Lambda(y))
\]
is a strictly $\M^d$-consistent selective scoring function for $R_0^{ins}$.
\end{enumerate}
\end{cor}

\begin{proof}
This is a direct consequence of Proposition \ref{prop:revelation} and Lemma \ref{lem:identification r}.
\end{proof}

\begin{prop}\label{prop:r0ins to rho}
Let $\r\colon\Y\to\R$ be a risk measure, $\Lambda\colon\R^d\to\R$ an aggregation function and $r\colon\Y^d\to \R$ as in \eqref{eq:r}. Assume that there exists a measurable right inverse $\eta\colon \Lambda(\R^d) \to\R^d$ such that $\Lambda\circ \eta = \mathrm{id}_\R$, and for any $X\in\Y$, $\eta(X)$ belongs to $\Y^d$. Then $\r$ is identifiable if and only if $r$ is identifiable. 
\end{prop}
\begin{proof}
The proof is analogous to the proof of Proposition 3.7 in \cite{FisslerHlavinovaRudloff_RM}.
\end{proof}

\section{Further simulation results}\label{app:simulations}
In Figures \ref{fig:Murphy_additional} and \ref{fig:Murphy_additional2}, we depict 4 more experiments as described in Subsection 6.2 of \cite{FisslerHlavinovaRudloff_RM}.

\begin{sidewaysfigure}
\begin{center}
\begin{tabular}{ c c c c c}
      \multicolumn{2}{c}{Experiment 2} & \hspace{40mm}&
			\multicolumn{2}{c}{Experiment 3}\\
      \includegraphics[width=4.0cm, height=4.0cm]{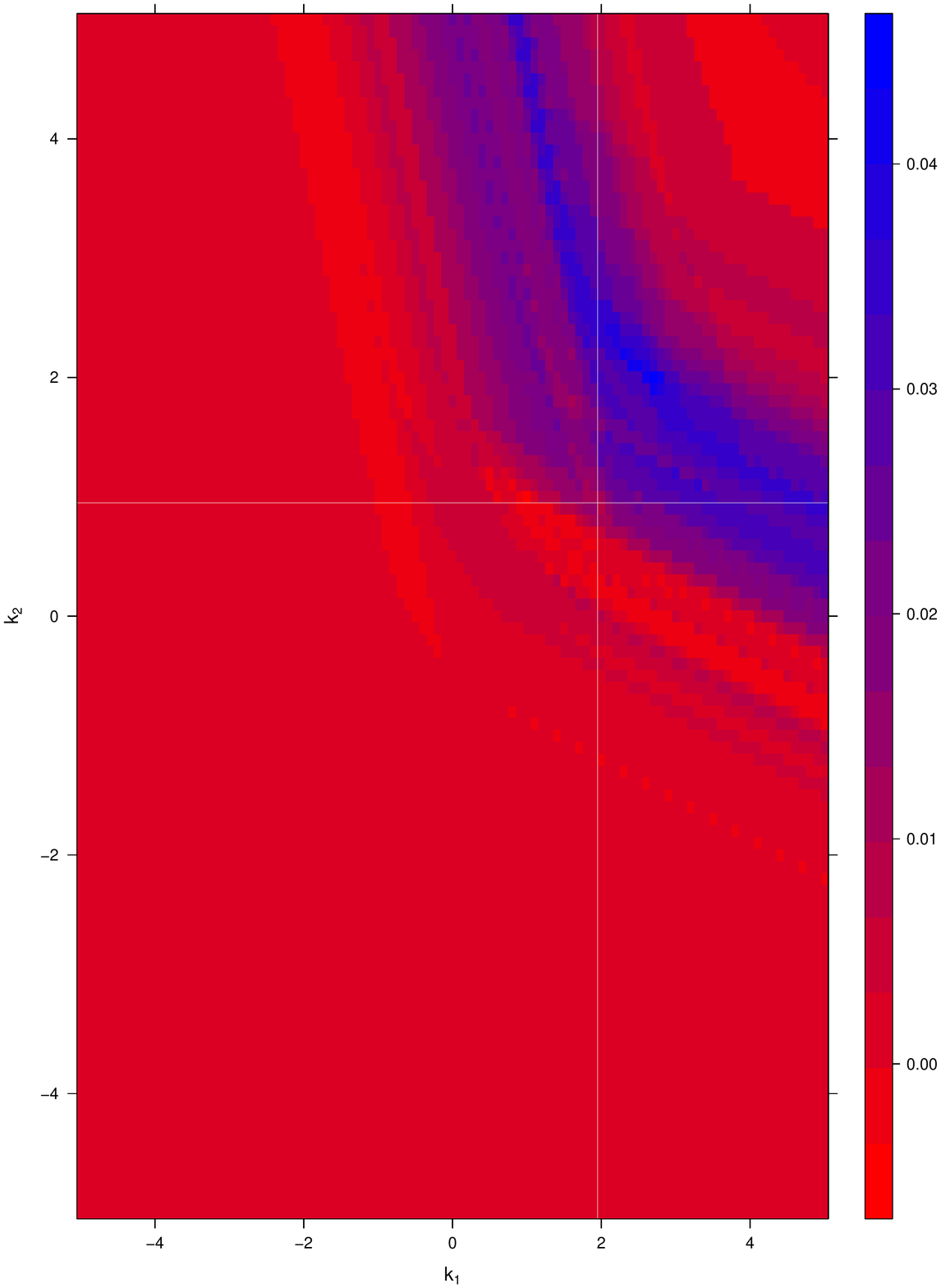} &
      \includegraphics[width=4.0cm, height=4.0cm]{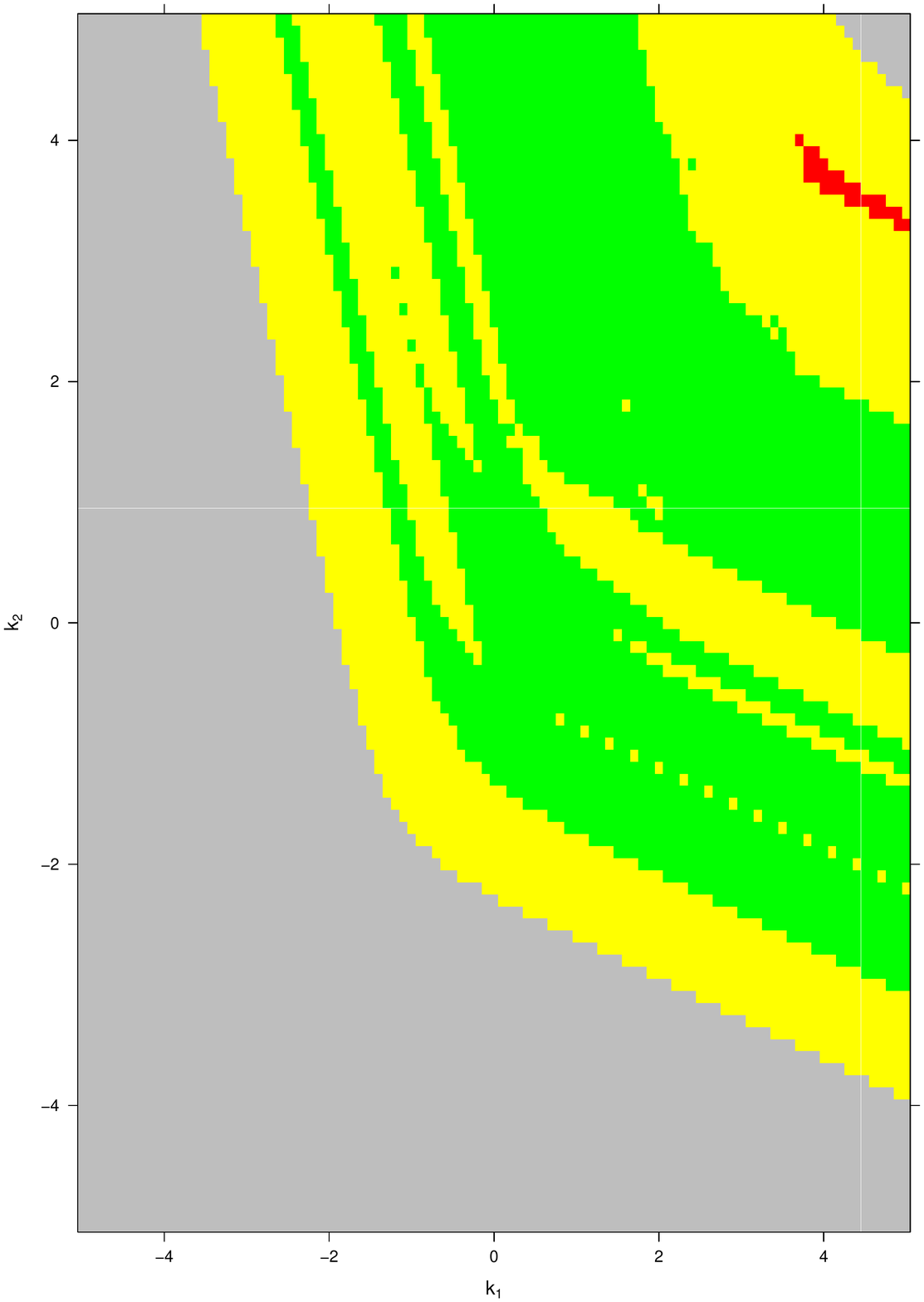} &\hspace{30mm}&
			\includegraphics[width=4.0cm, height=4.0cm]{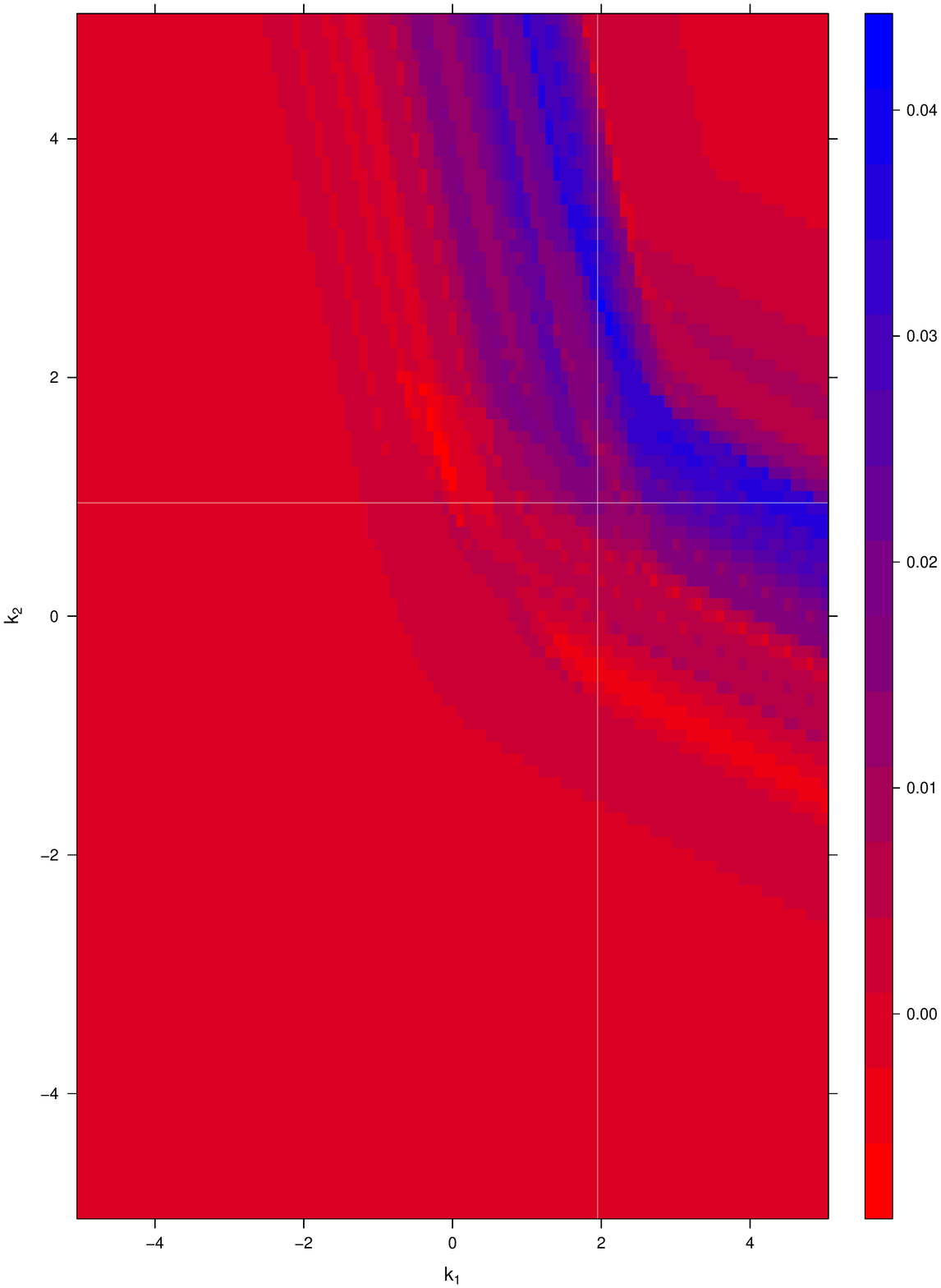} &
      \includegraphics[width=4.0cm, height=4.0cm]{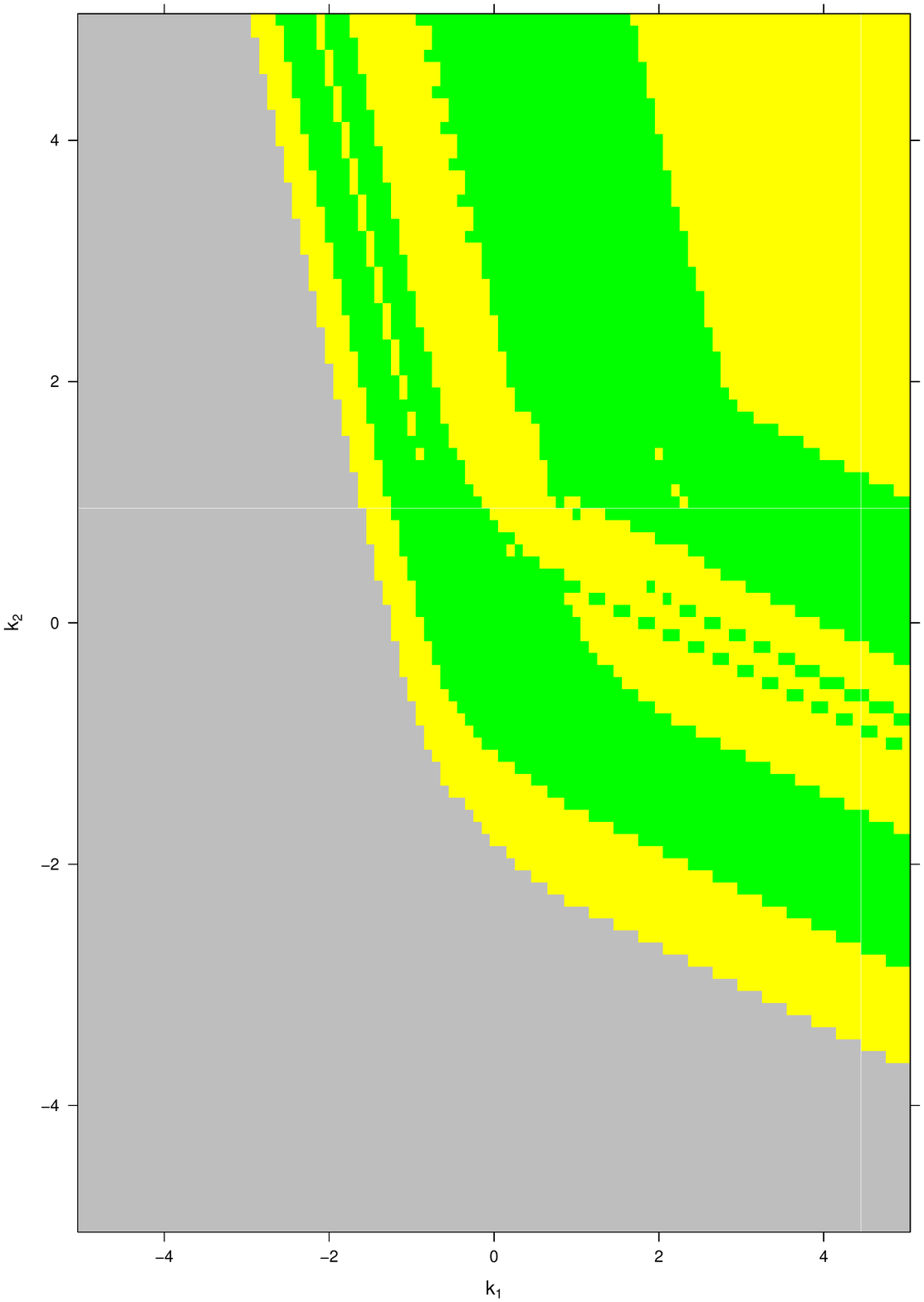}\\
			\multicolumn{2}{c}{\footnotesize $f_{1t}=B_t$, $f_{2t}=A_t$} &\hspace{30mm}&
			\multicolumn{2}{c}{\footnotesize $f_{1t}=B_t$, $f_{2t}=A_t$}\\
      \includegraphics[width=4.0cm, height=4.0cm]{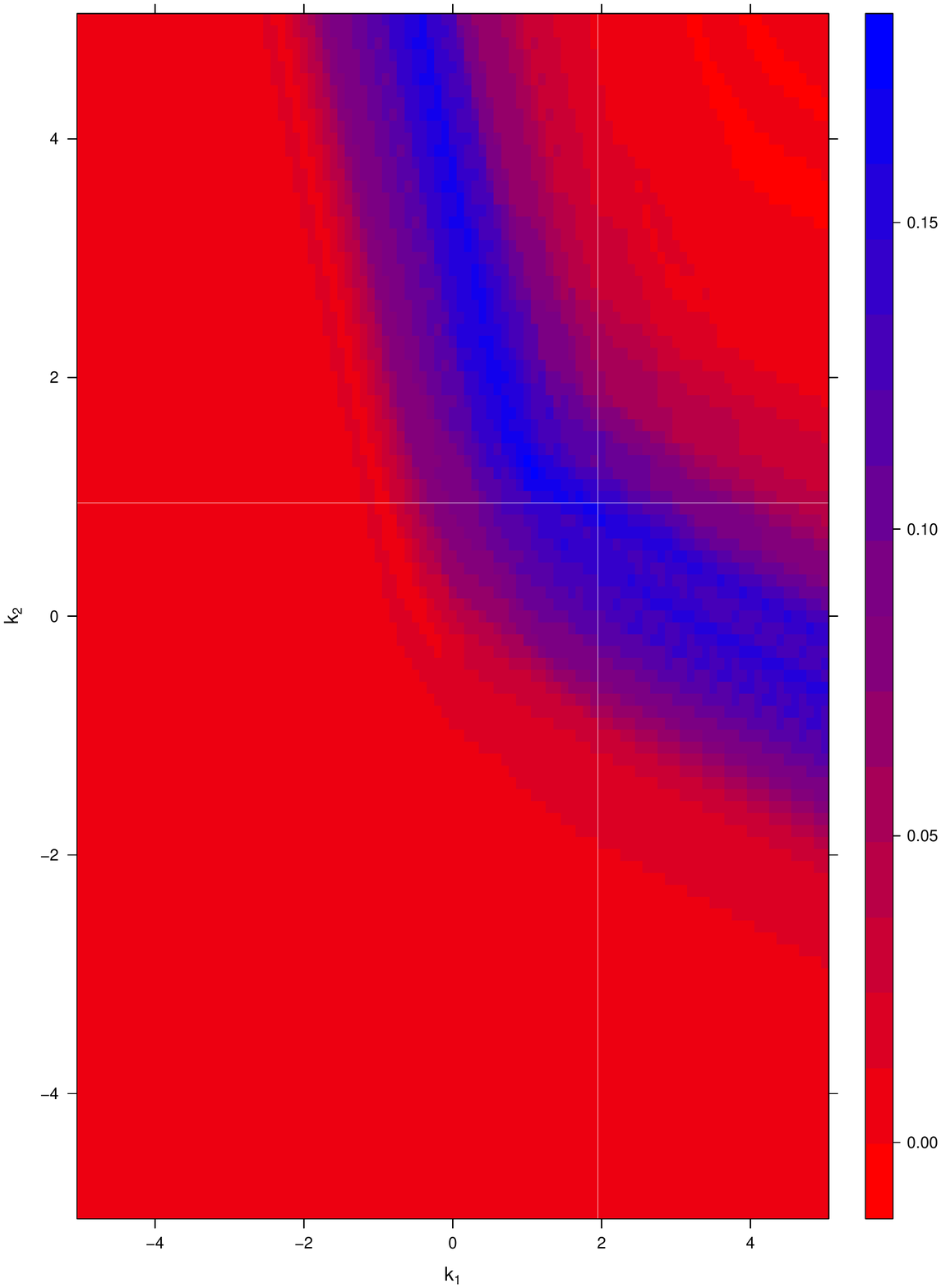} &
      \includegraphics[width=4.0cm, height=4.0cm]{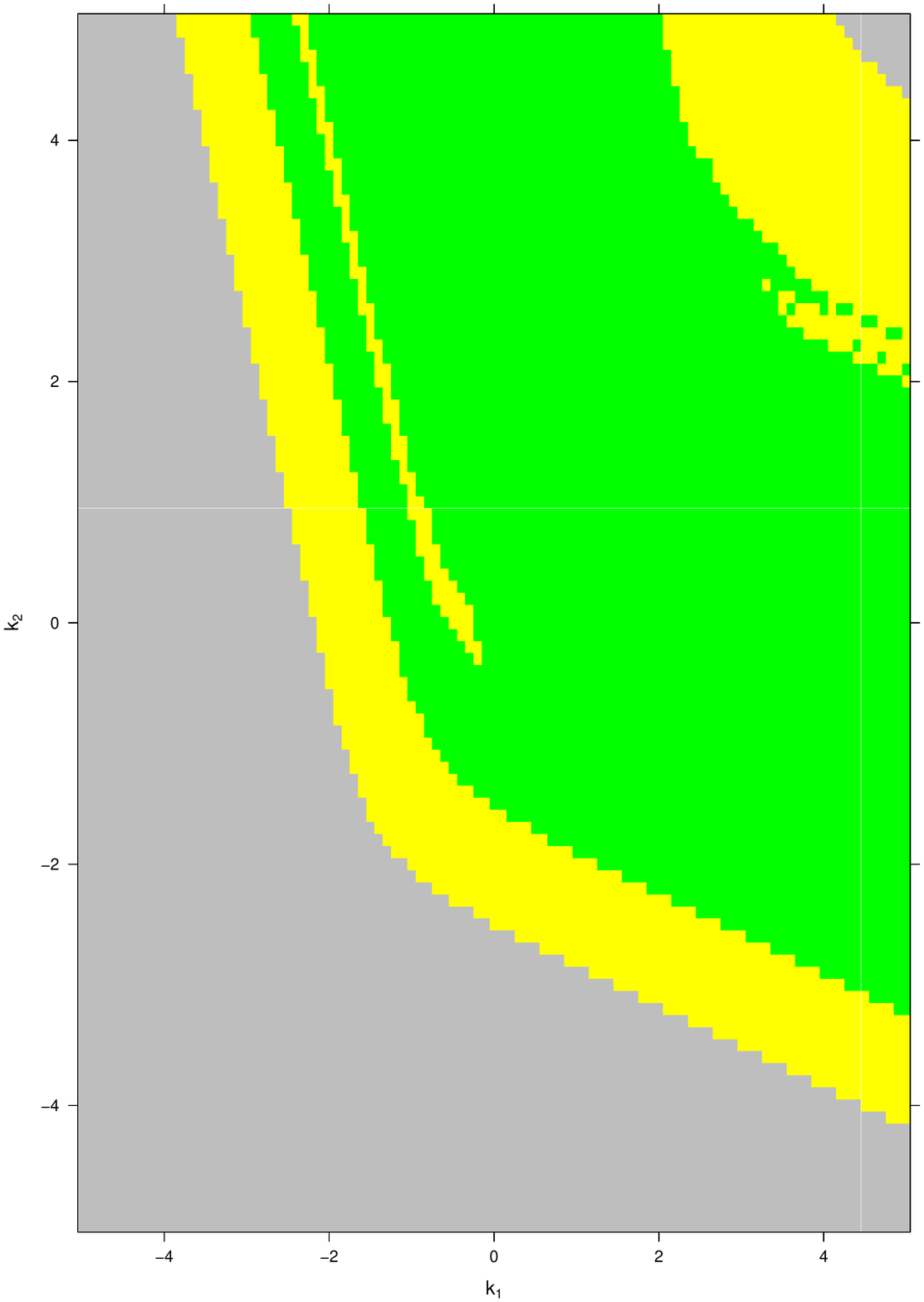} &\hspace{40mm}&
			\includegraphics[width=4.0cm, height=4.0cm]{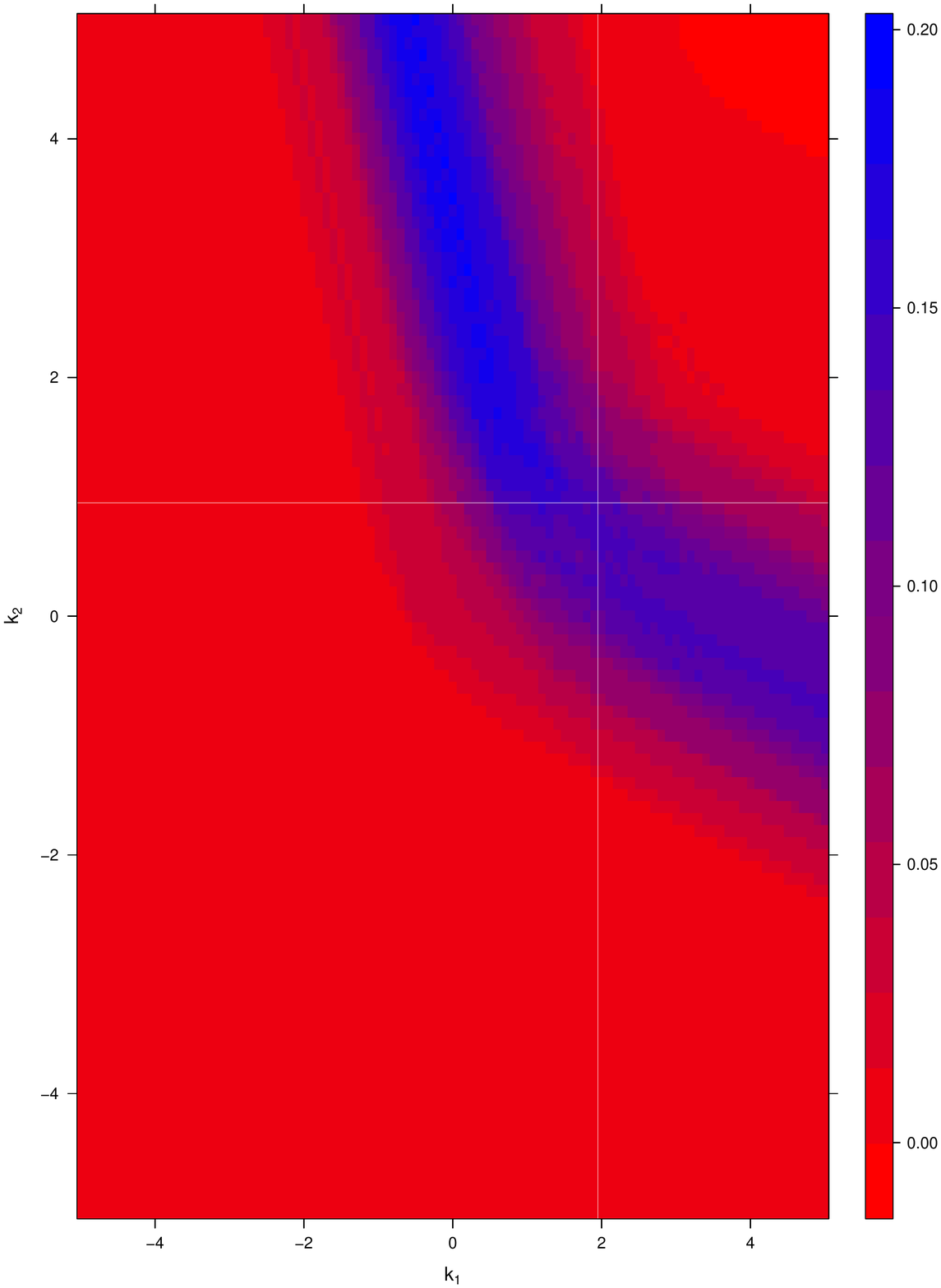} &
      \includegraphics[width=4.0cm, height=4.0cm]{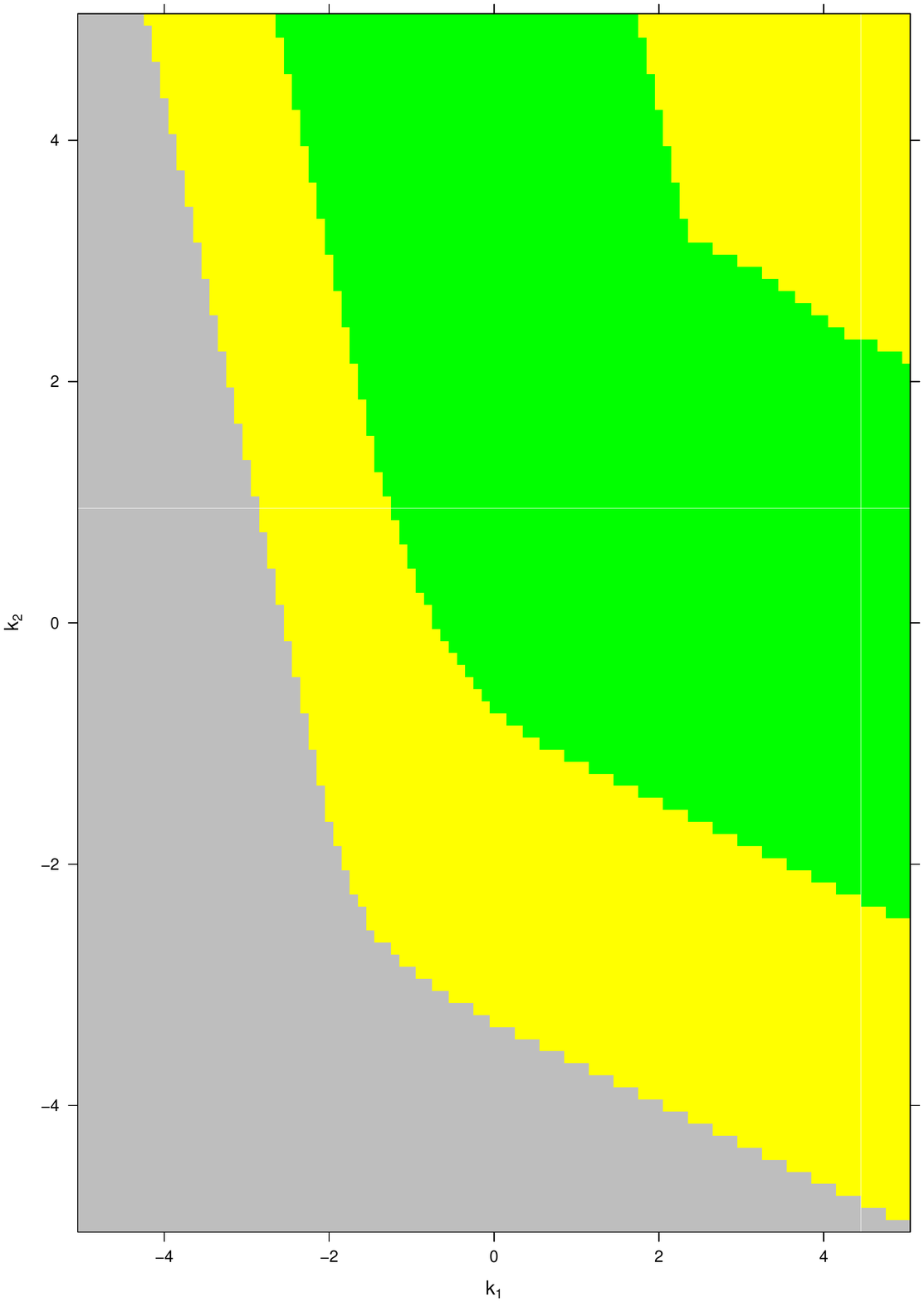}\\
      \multicolumn{2}{c}{\footnotesize $f_{1t}=C_t$, $f_{2t}=A_t$} &\hspace{40mm}&
			\multicolumn{2}{c}{\footnotesize $f_{1t}=C_t$, $f_{2t}=A_t$}\\
      \includegraphics[width=4.0cm, height=4.0cm]{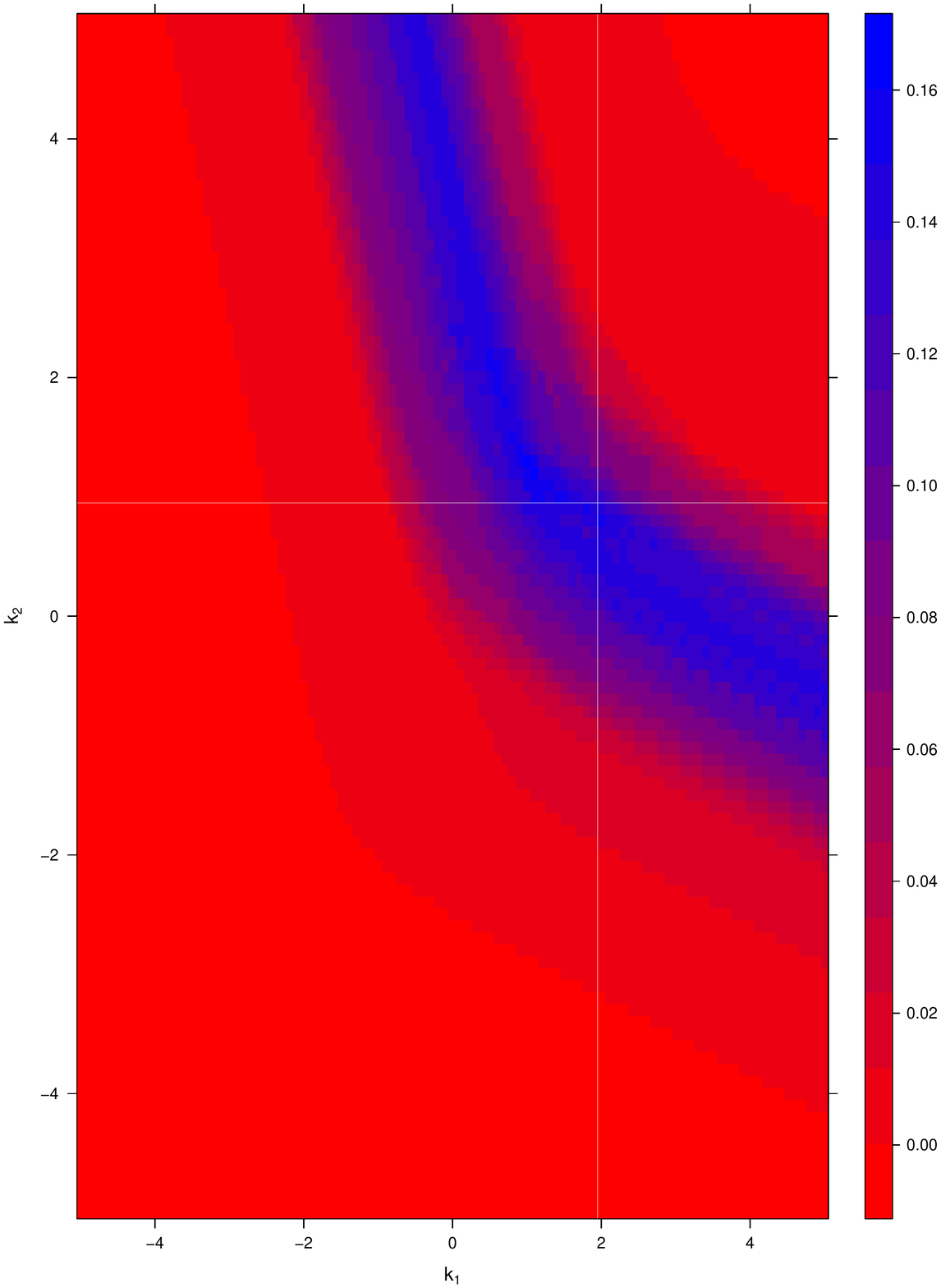} &
      \includegraphics[width=4.0cm, height=4.0cm]{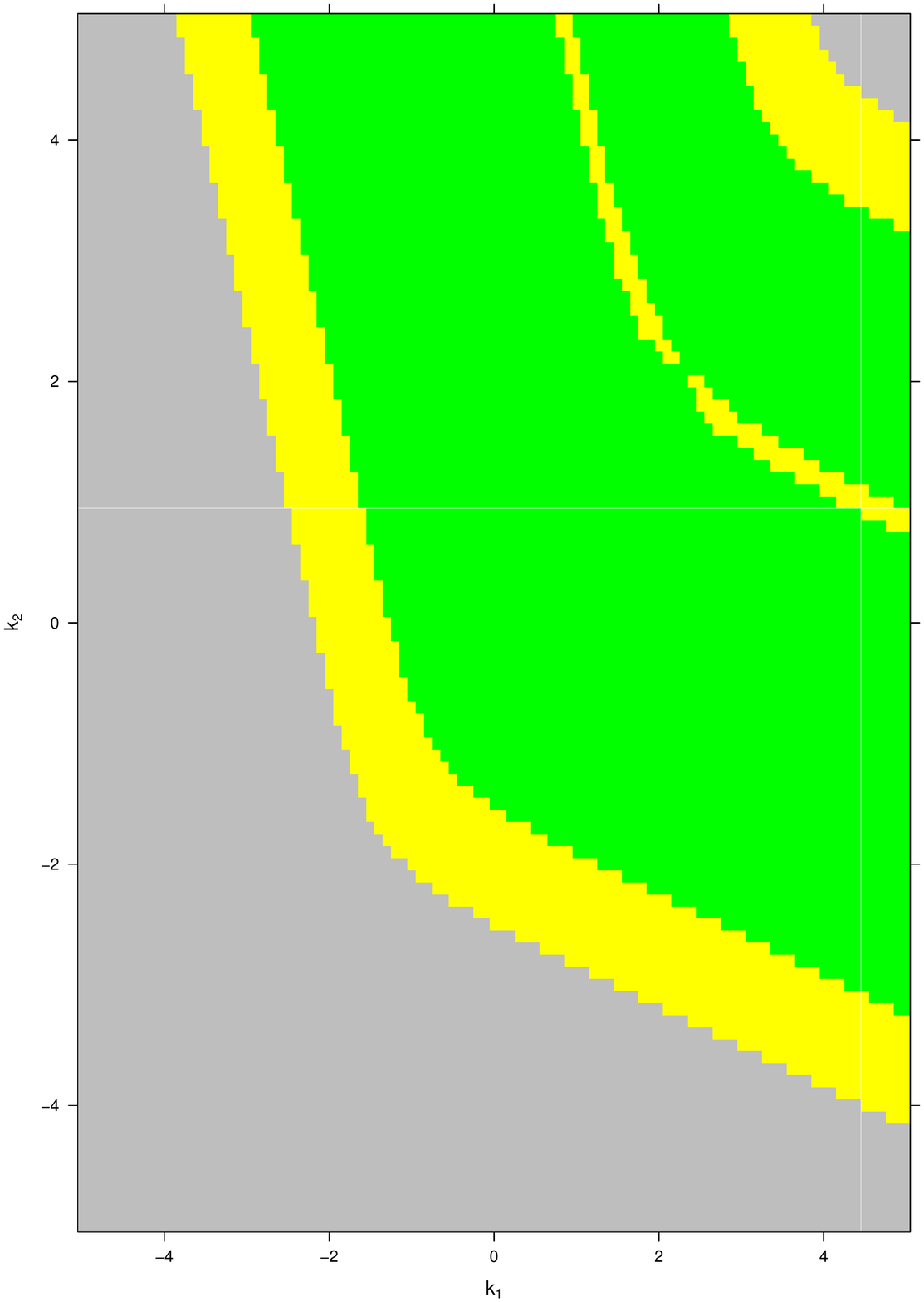} &\hspace{40mm}&
			\includegraphics[width=4.0cm, height=4.0cm]{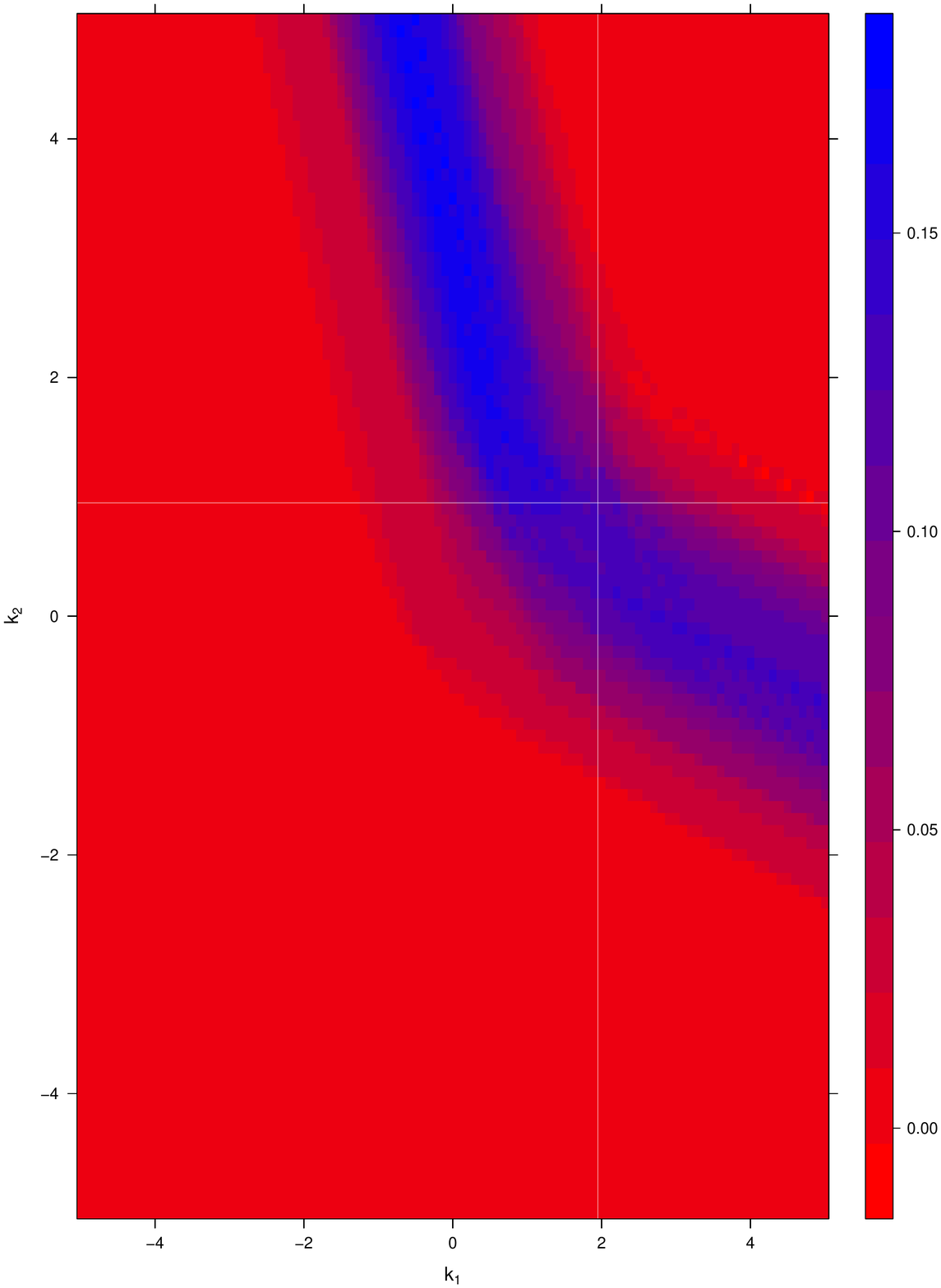} &
      \includegraphics[width=4.0cm, height=4.0cm]{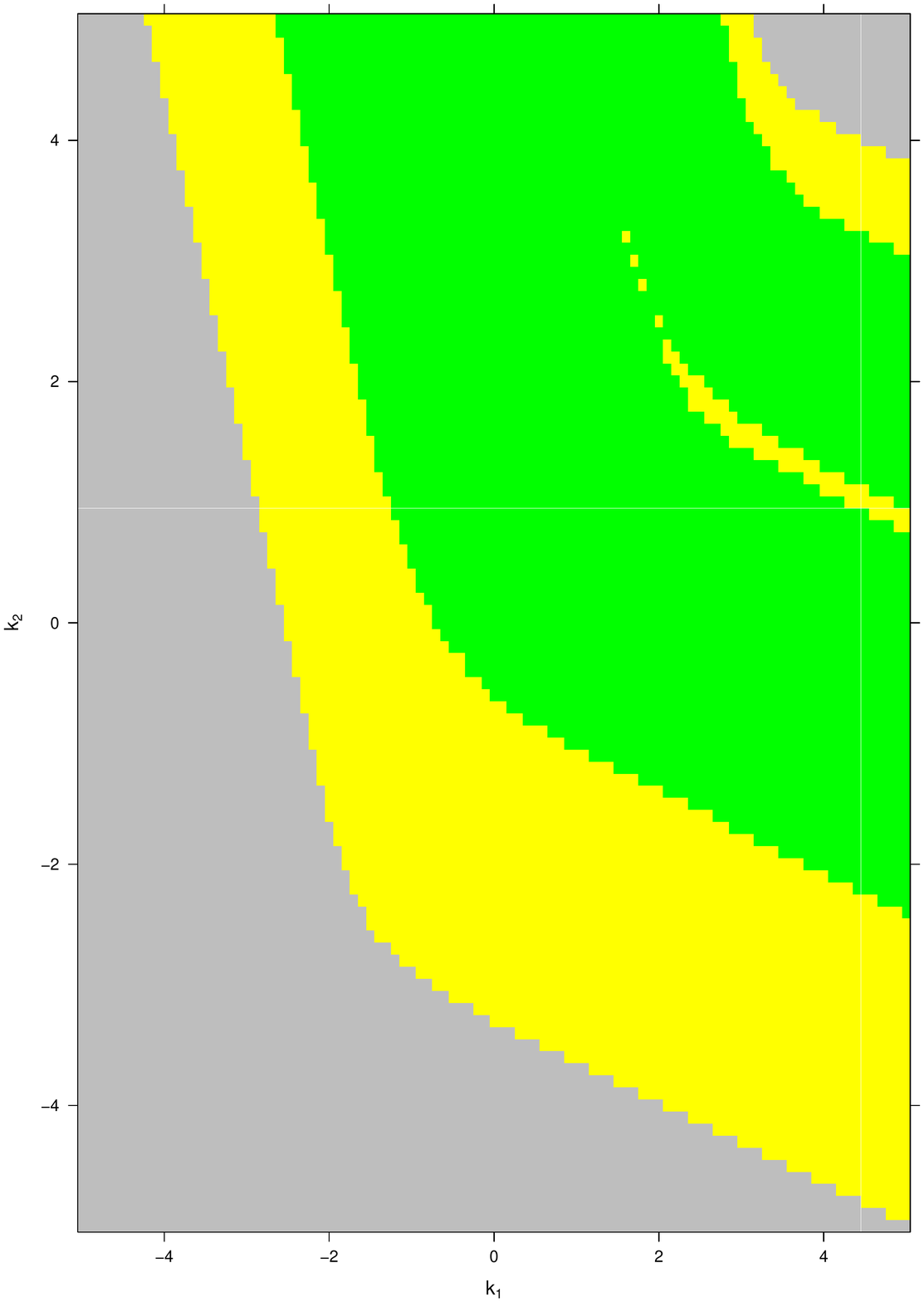} \\
			\multicolumn{2}{c}{\footnotesize $f_{1t}=C_t$, $f_{2t}=B_t$} &\hspace{40mm}&
			\multicolumn{2}{c}{\footnotesize $f_{1t}=C_t$, $f_{2t}=B_t$}
    \end{tabular}
\end{center}
\caption{See description of Figure 2 in \cite{FisslerHlavinovaRudloff_RM}.}
\label{fig:Murphy_additional}
\end{sidewaysfigure}

\begin{sidewaysfigure}
\begin{center}
\begin{tabular}{ c c c c c}
      \multicolumn{2}{c}{Experiment 4} & \hspace{40mm}&
			\multicolumn{2}{c}{Experiment 5}\\
      \includegraphics[width=4cm, height=4cm]{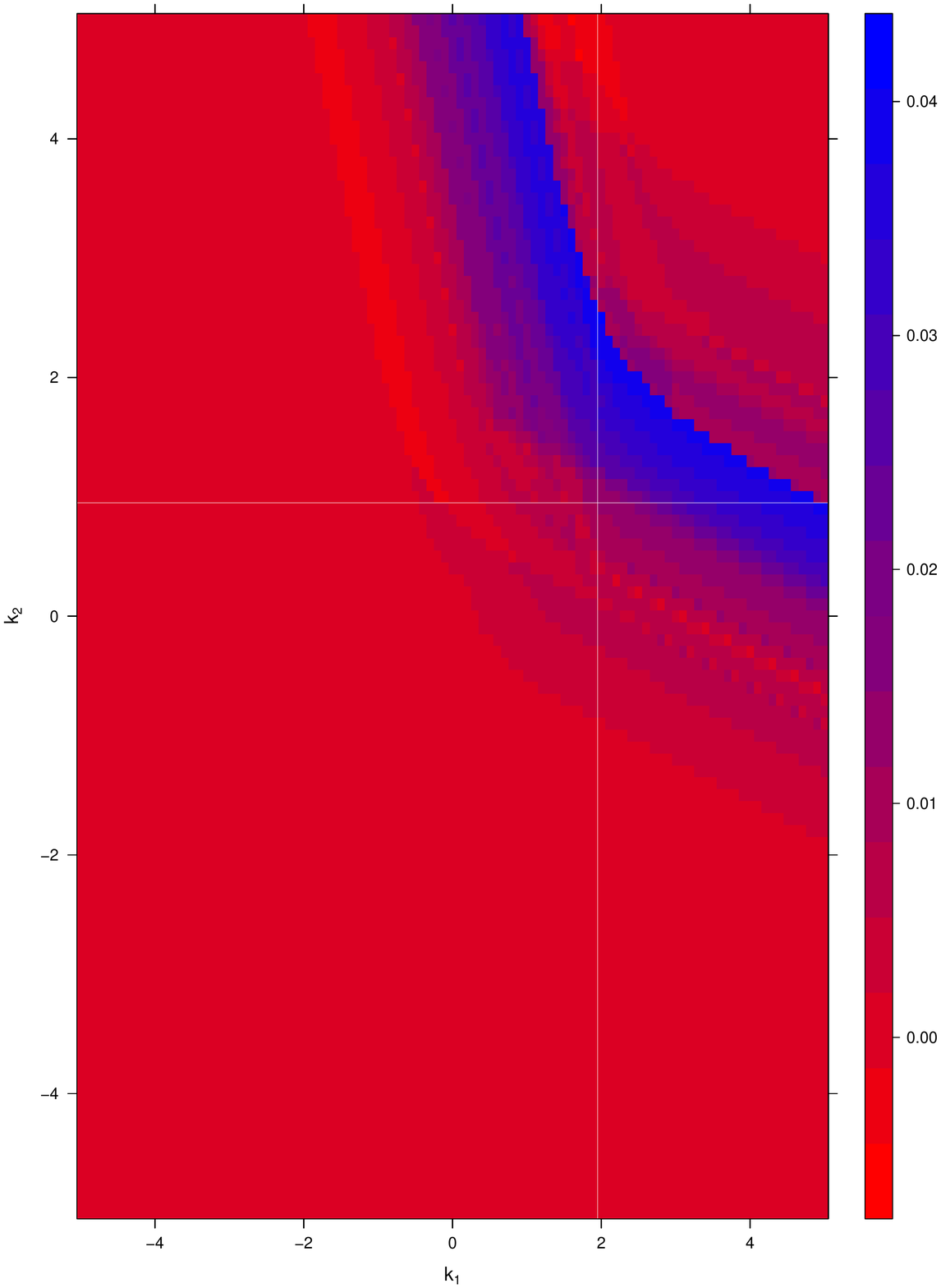} &
      \includegraphics[width=4cm, height=4cm]{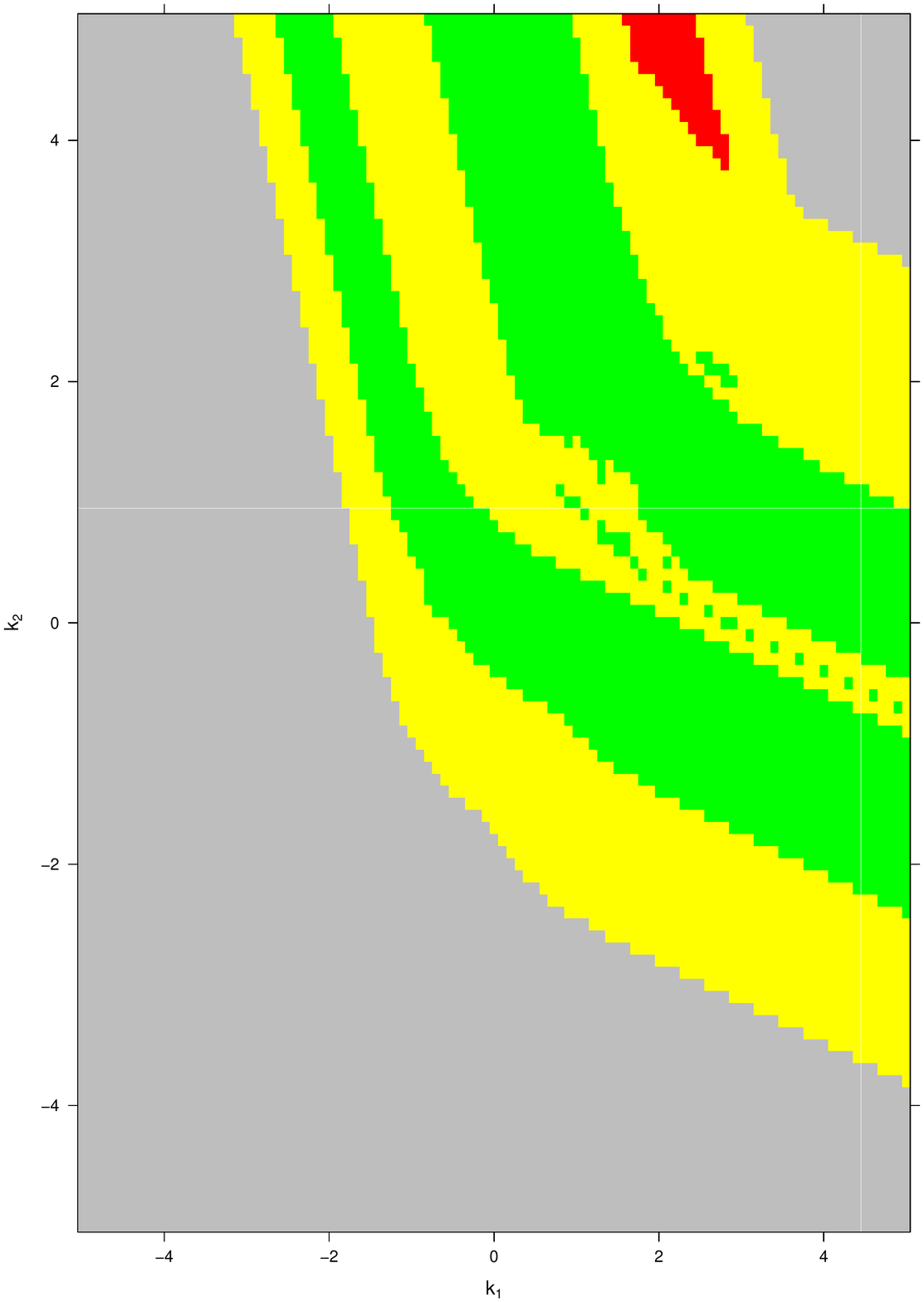} &\hspace{40mm}&
			\includegraphics[width=4cm, height=4cm]{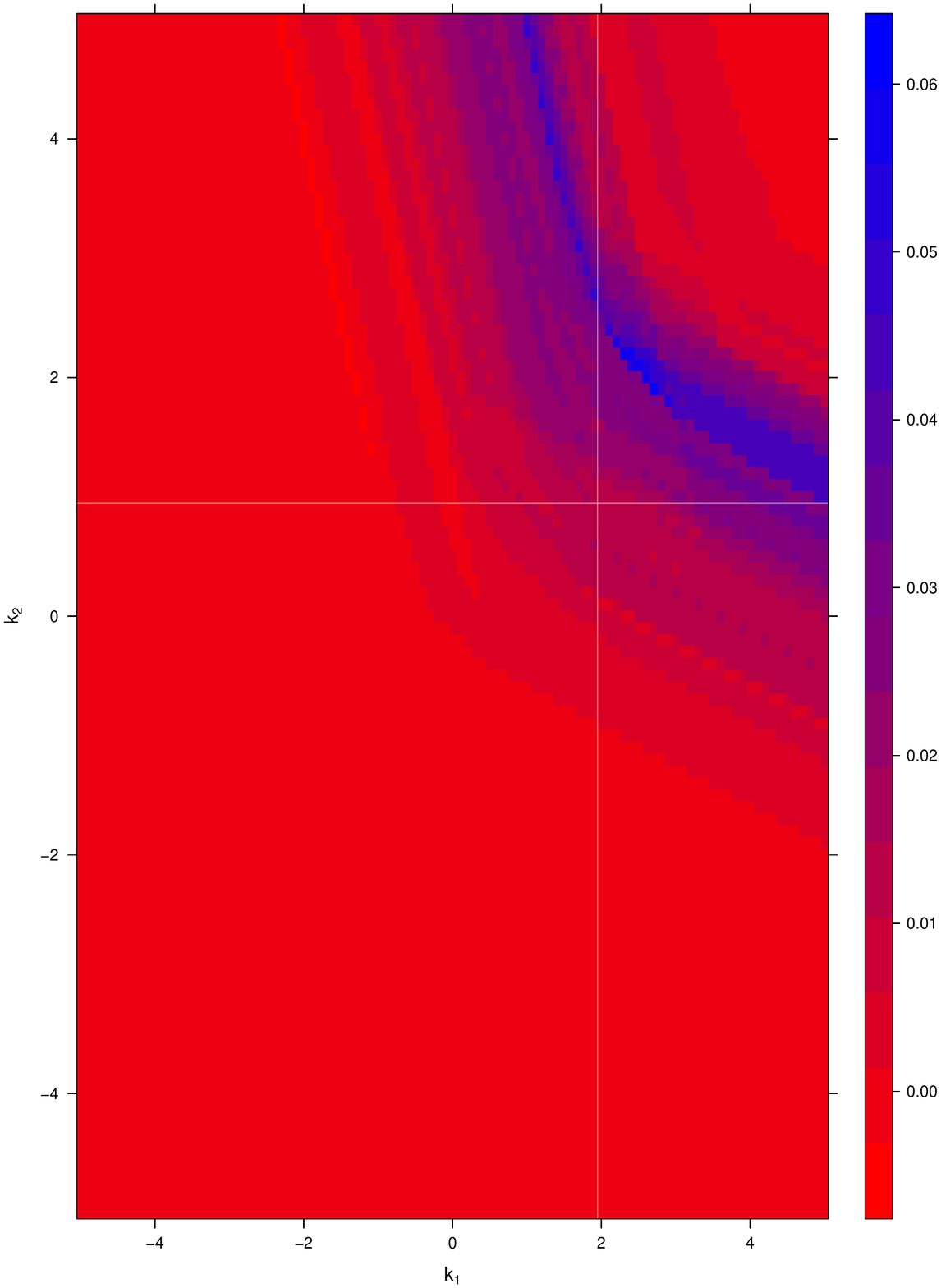} &
      \includegraphics[width=4cm, height=4cm]{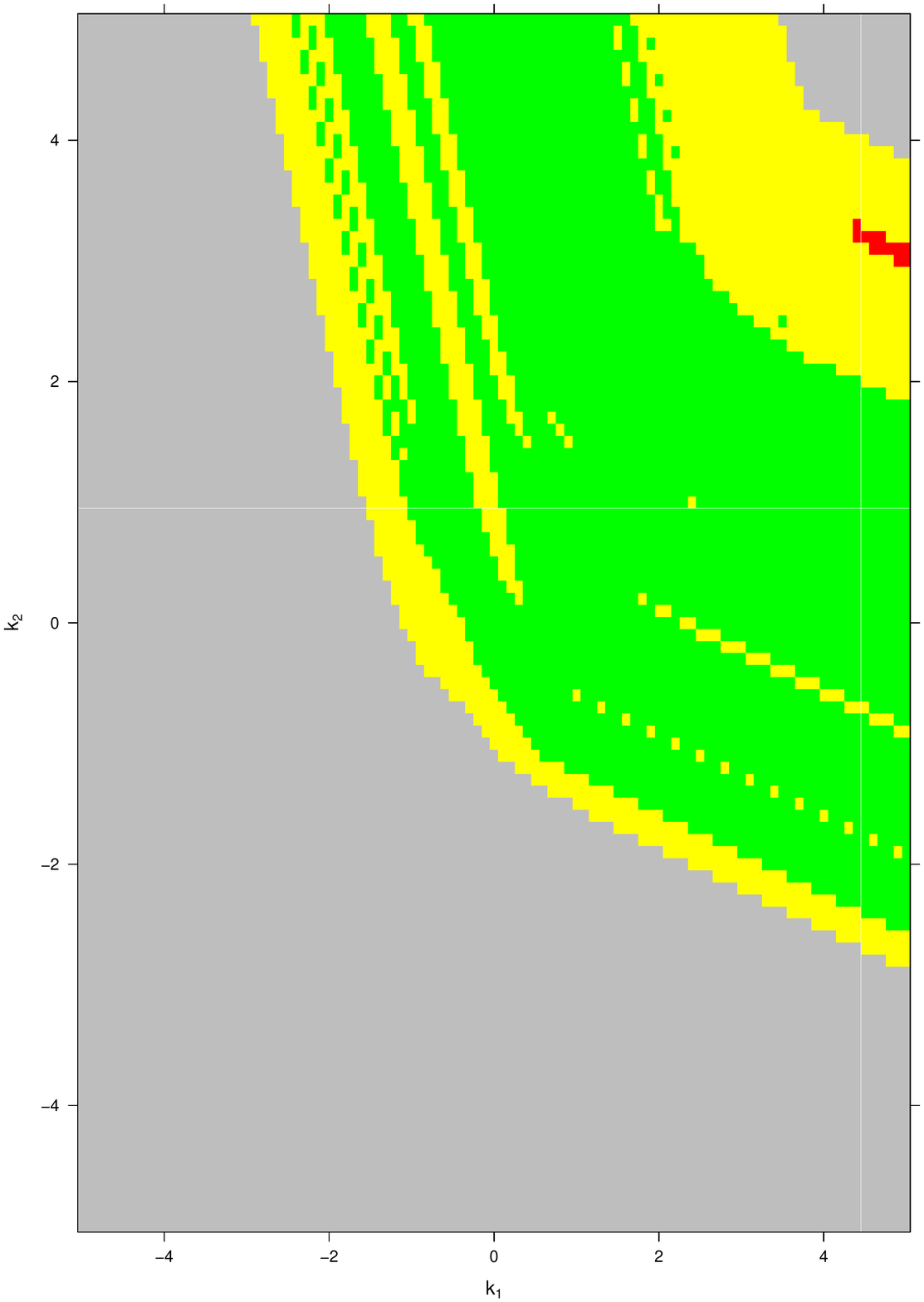}\\
			\multicolumn{2}{c}{\footnotesize $f_{1t}=B_t$, $f_{2t}=A_t$} &\hspace{40mm}&
			\multicolumn{2}{c}{\footnotesize $f_{1t}=B_t$, $f_{2t}=A_t$}\\
      \includegraphics[width=4cm, height=4cm]{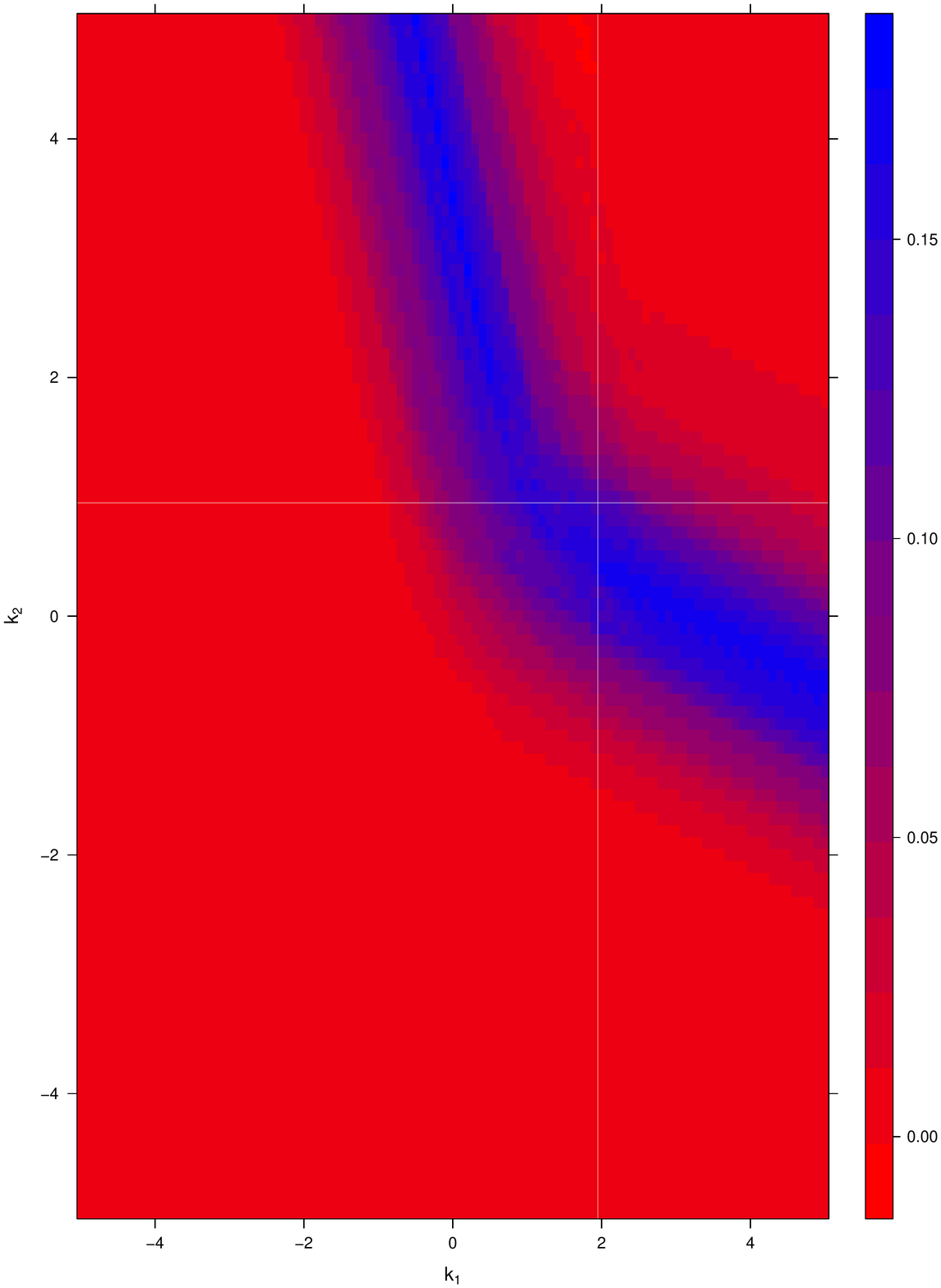} &
      \includegraphics[width=4cm, height=4cm]{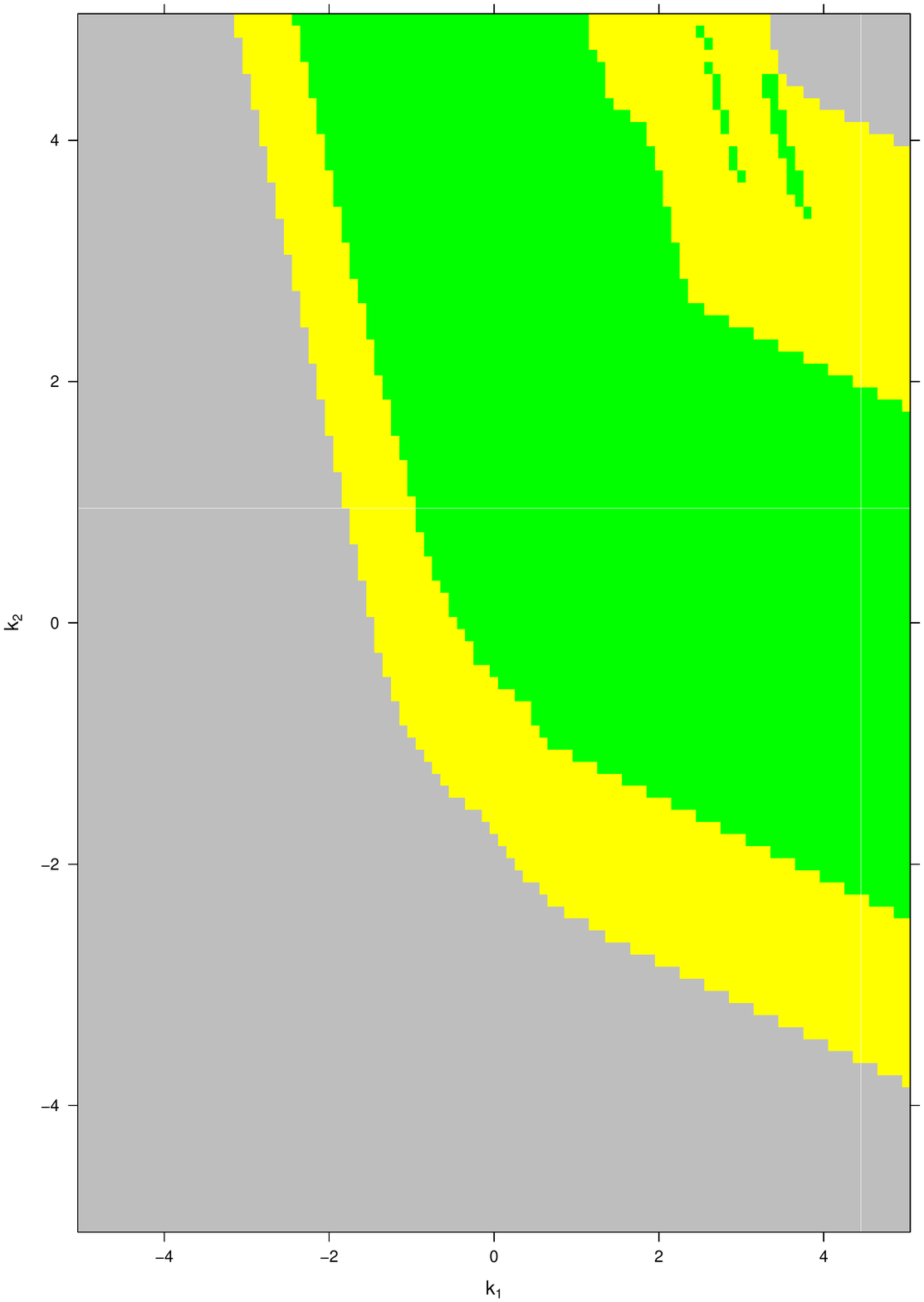} &\hspace{40mm}&
			\includegraphics[width=4cm, height=4cm]{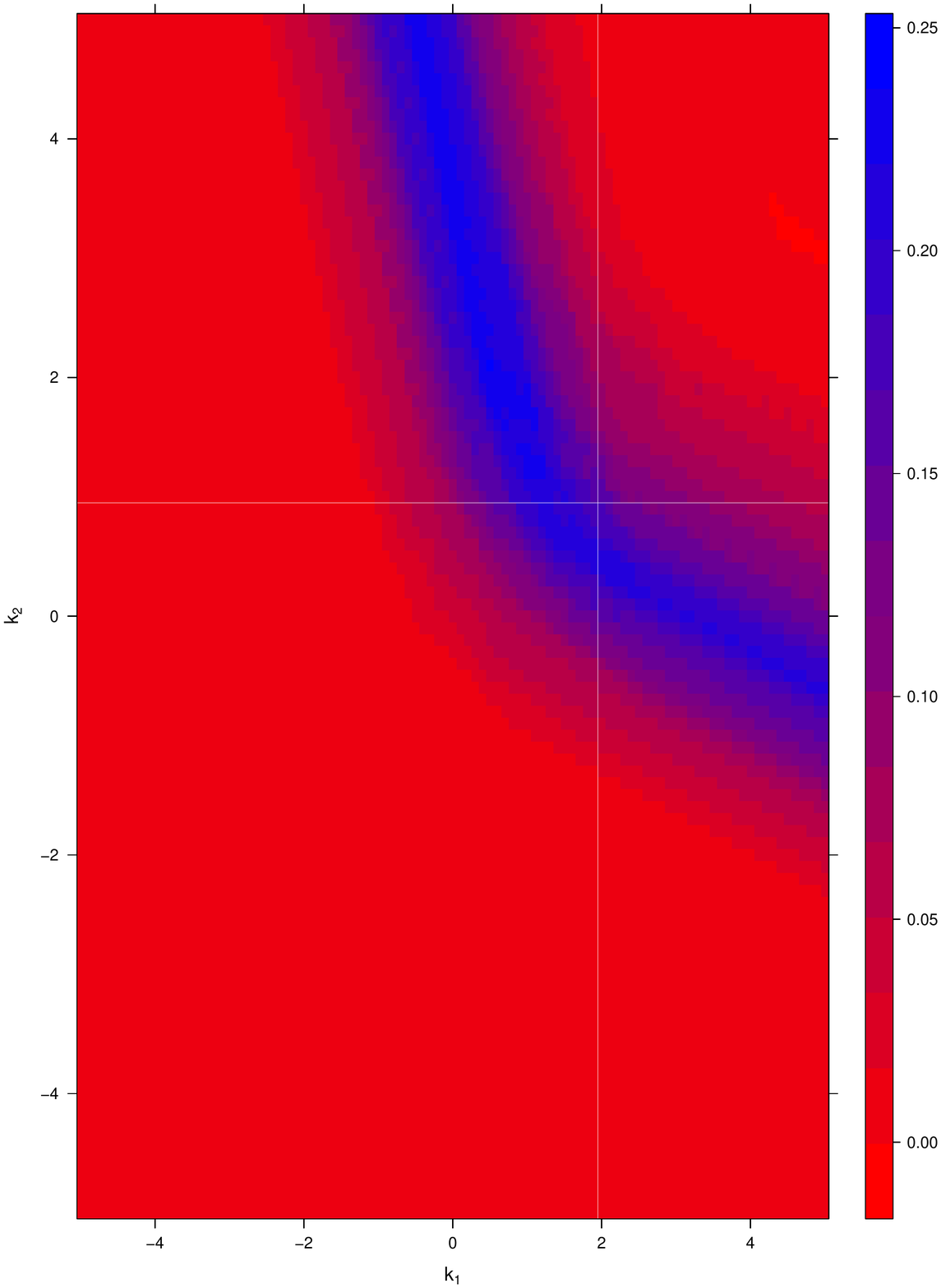} &
      \includegraphics[width=4cm, height=4cm]{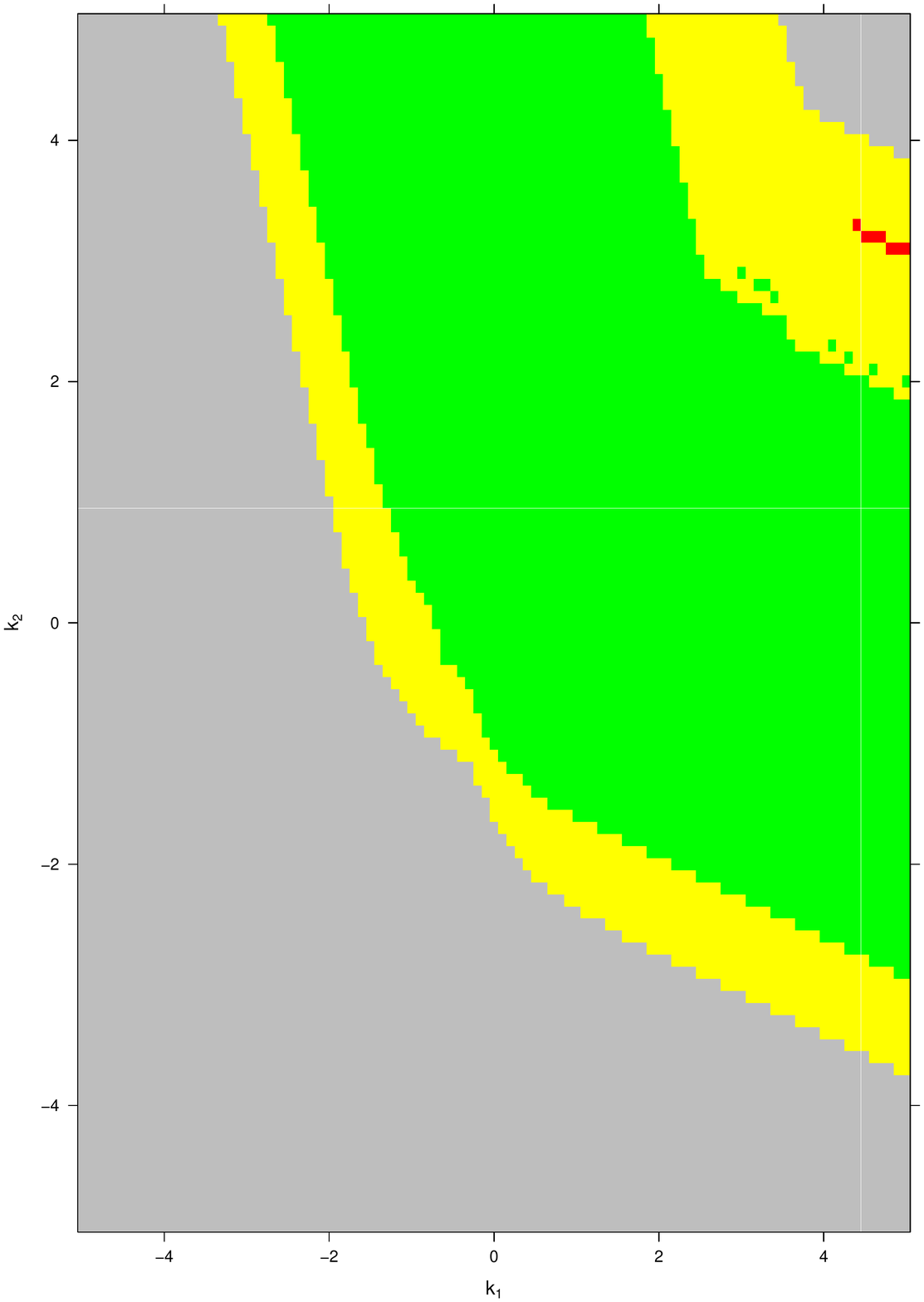}\\
      \multicolumn{2}{c}{\footnotesize $f_{1t}=C_t$, $f_{2t}=A_t$} &\hspace{40mm}&
			\multicolumn{2}{c}{\footnotesize $f_{1t}=C_t$, $f_{2t}=A_t$}\\
      \includegraphics[width=4cm, height=4cm]{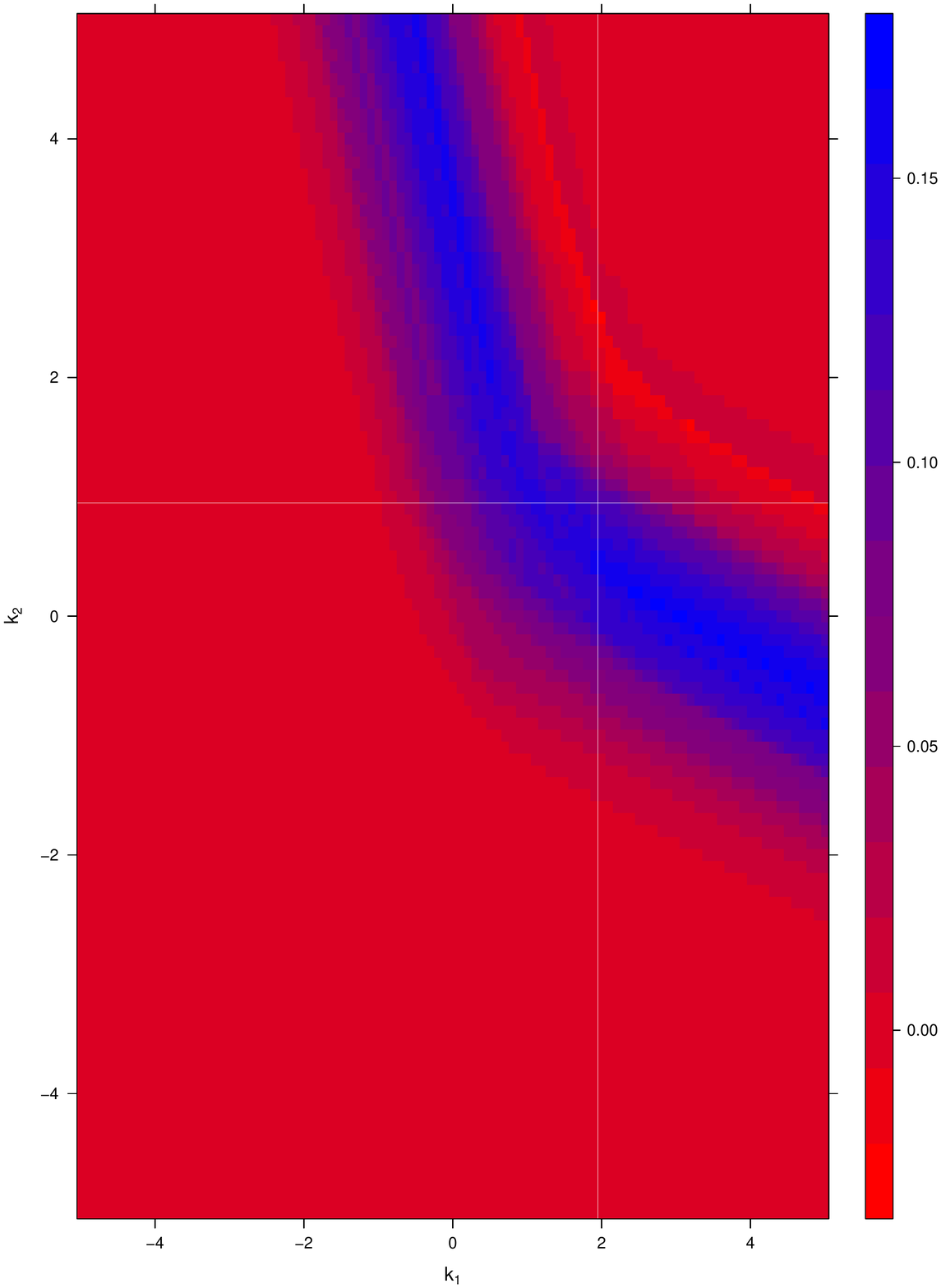} &
      \includegraphics[width=4cm, height=4cm]{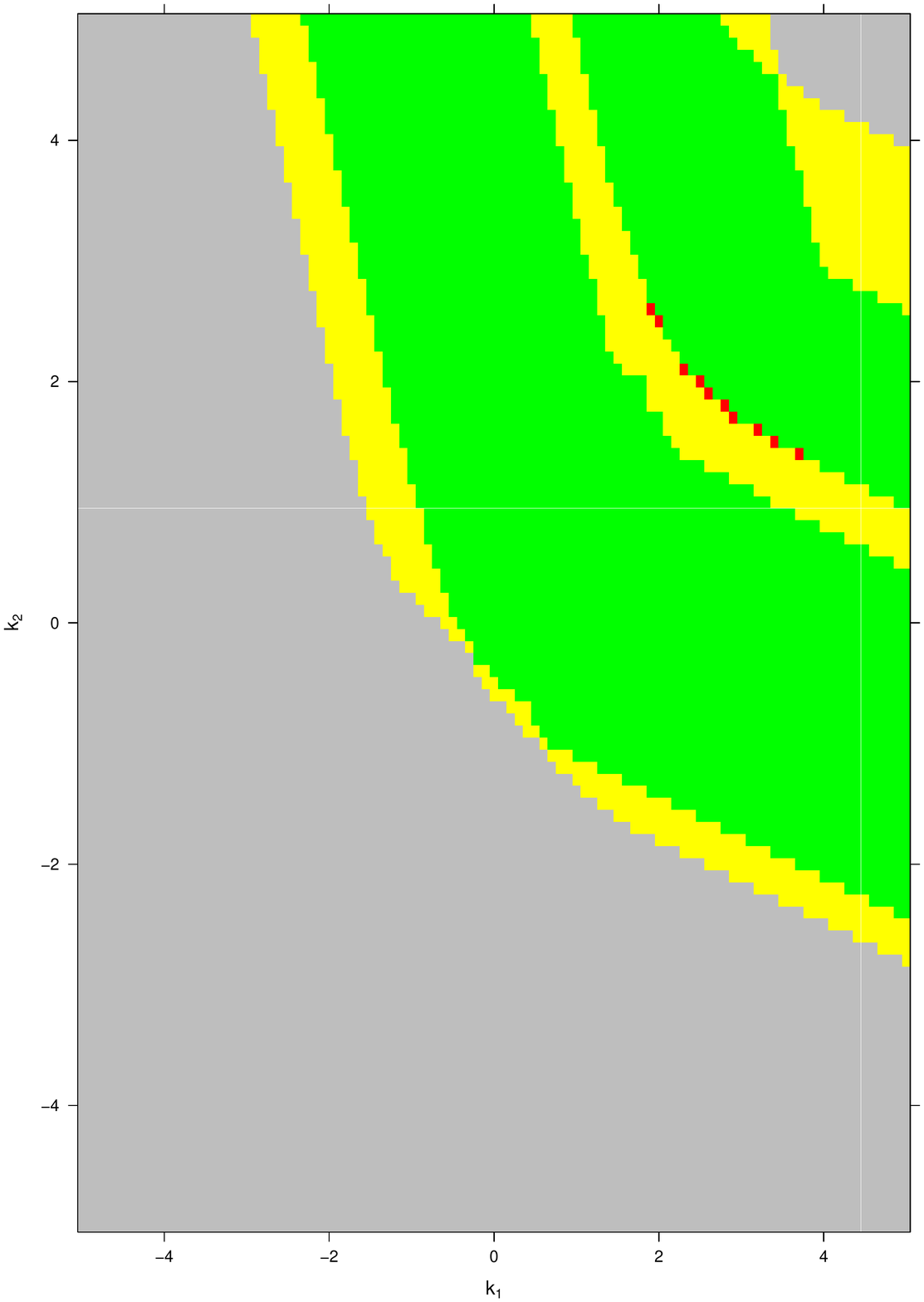} &\hspace{40mm}&
			\includegraphics[width=4cm, height=4cm]{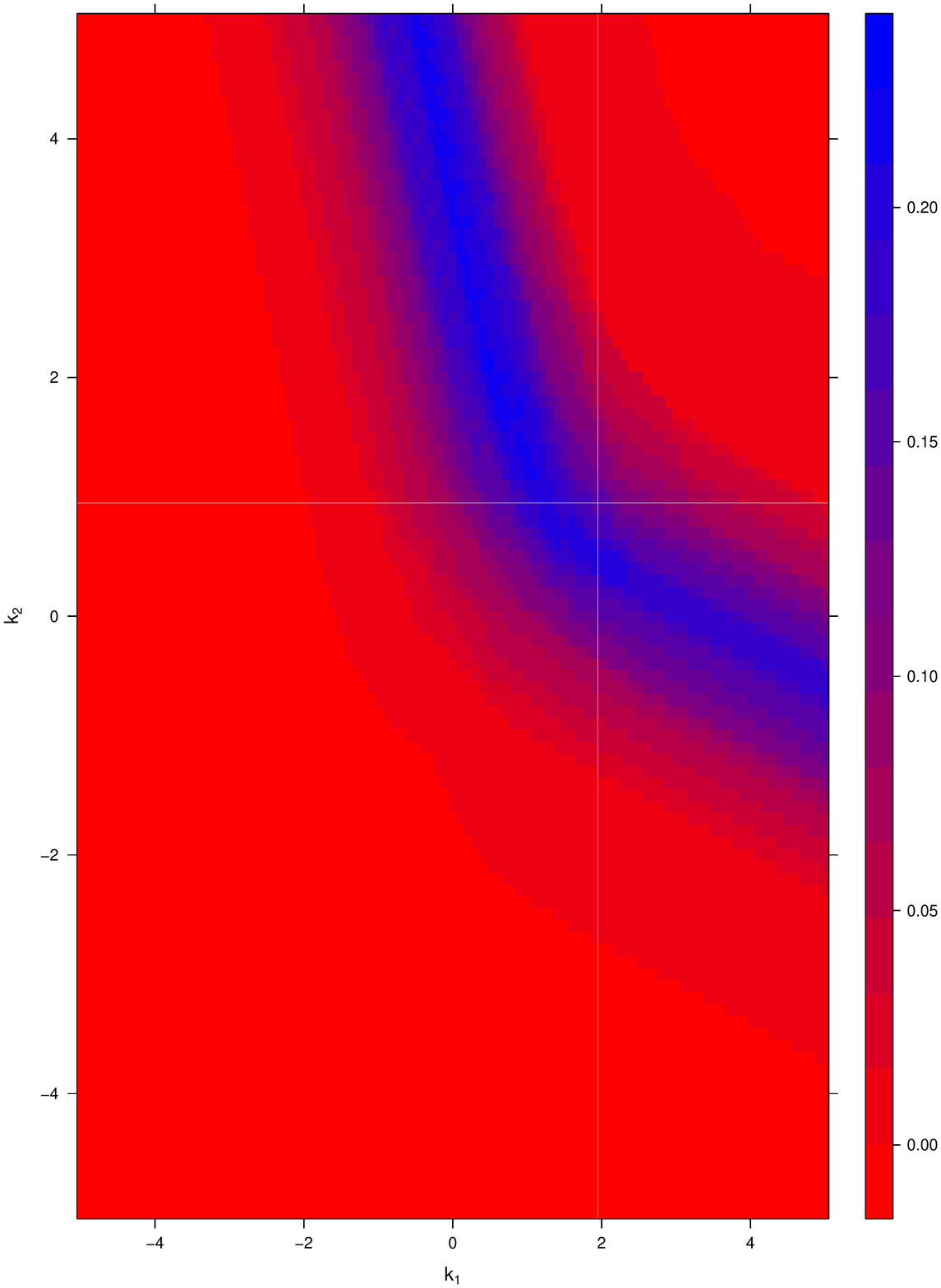} &
      \includegraphics[width=4cm, height=4cm]{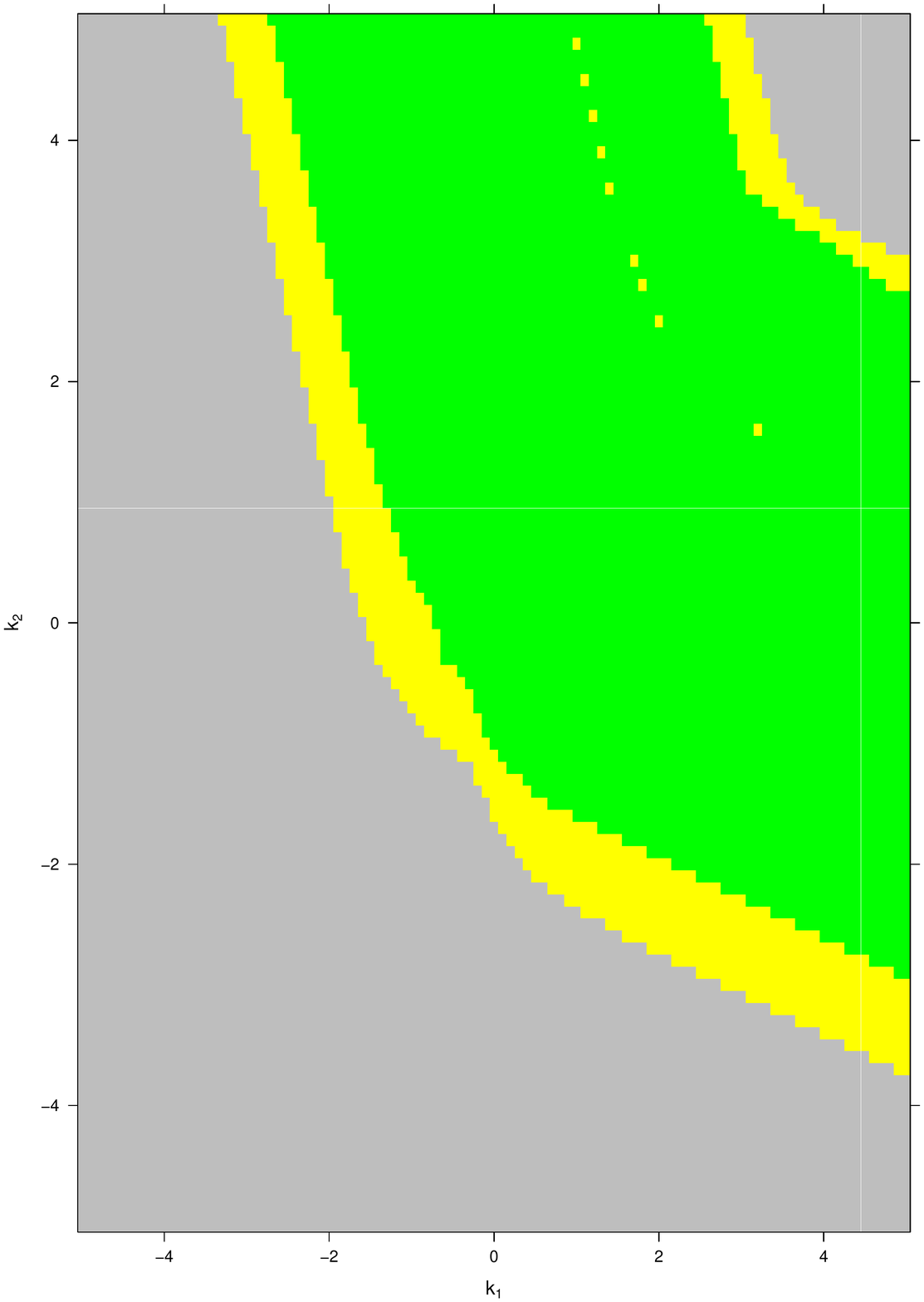} \\
			\multicolumn{2}{c}{\footnotesize $f_{1t}=C_t$, $f_{2t}=B_t$} &\hspace{40mm}&
			\multicolumn{2}{c}{\footnotesize $f_{1t}=C_t$, $f_{2t}=B_t$}
    \end{tabular}
\end{center}
\caption{See description of Figure 2 in \cite{FisslerHlavinovaRudloff_RM}.}
\label{fig:Murphy_additional2}
\end{sidewaysfigure}
%
%

\bibliographystyle{apacite}